\documentclass[12pt]{article}
\usepackage[left=3cm,right=3cm, top=2.5cm,bottom=2.5cm,bindingoffset=0cm]{geometry}
\usepackage{cmap}

\usepackage{amsmath,amsthm,amssymb}
\usepackage[T2A]{fontenc}
\usepackage[english]{babel}
\usepackage[cp1251]{inputenc}
\usepackage{amscd}

\usepackage{graphicx}

\usepackage{caption}
\usepackage{subcaption}

\usepackage{hyperref}

\usepackage[all,cmtip]{xy}

  \newskip\prethm \prethm3.0pt plus1.3pt minus.4pt
  \newskip\posthm \posthm2.7pt plus1.4pt minus.3pt
  \newtheoremstyle{STATEMENT}%
       {\prethm}{\posthm}{\itshape}{\parindent}{\scshape}
       {.}{.6em plus.2em minus.1em}{}
  \newtheoremstyle{EXPLANATION}%
       {\prethm}{\posthm}{}{\parindent}{\scshape}
       {.}{.6em plus.2em minus.1em}{}
\theoremstyle{STATEMENT}

\newtheorem {theorem}{Theorem}
\newtheorem {lemma}{Lemma}

\newtheorem    {assertion} {Assertion}

\theoremstyle{EXPLANATION}

\newtheorem {definition} {Definition}
\newtheorem   {remark}{Remark}

\begin{document}

\title{The topology of Liouville foliation for the Kovalevskaya integrable case on the
Lie algebra so(4).}

\author{Kozlov~I.\,K.\thanks{No Affiliation, Moscow, Russia.} \\  ikozlov90@gmail.com}
\date{}

\maketitle

\begin{abstract}
In this paper we study topological properties of an integrable case
for Euler's equations on the Lie algebra $\textrm{so}(4)$, which can
be regarded as an analogue of the classical Kovalevskaya case in
rigid body dynamics. In particular, for all values of the parameters
of the system under consideration the bifurcation diagrams of the
momentum mapping are constructed, the types of critical points of
rank $0$ are determined, the bifurcations of Liouville tori are
described and the loop molecules for all singular points of the
bifurcation diagram are computed. It follows from the obtained
results that some topological properties of the classical
Kovalevskaya case can be obtained from the corresponding properties
of the considered integrable case on the Lie algebra
$\textrm{so}(4)$ by taking a natural limit.

Bibliography: 21 titles.

\end{abstract}

\begin{flushleft}
\textbf{Keywords:}  integrable Hamiltonian systems, Kovalevskaya case, Liouville foliation,
bifurcation diagram, topological invariants.
\end{flushleft}

\footnotetext[0]{This work was supported by the Russian Foundation for Basic Research (grant nos.
13-01-00664-a and 12-01-31497), the programme “Leading Scientific Schools” (grant no.
НШ-581.2014.1) and a grant from the Government of the Russian Federation for the State Support
of Research at Russian Institutes of Higher Education Supervised by Leading Scientists (contract
no. 11.G34.31.0054).}

\tableofcontents

\section{Introduction} \label{s1}

The Kovalevskaya top is one of the most well-known integrable
Hamiltonian systems in classical mechanics. It was proved by Sofia
Kovalevskaya in her paper \cite{Kowalewski1889Acta14} that besides
the cases of Euler, Lagrange and her own, opened earlier in
\cite{Kowalewski1889Acta12}, there is no other rigid body systems
that would be integrable in the same way for any value of the area
constant. The Kovalevskaya top is more complex than the Euler and
Lagrange tops hence various methods that would allow us to simplify
the work with this top are of interest. In this paper we demonstrate
one of such possible methods. Namely, we consider a one-parameter
family of integrable Hamiltonian systems defined on the pencil of
Lie algebras $\textrm{so}(4) - \textrm{e}(3) - \textrm{so}(3, 1)$
found in the paper \cite{Komarov81} and show that some information
about the classical Kovalevskaya case, which is an integrable
Hamiltonian system on the Lie algebra $\textrm{e}(3)$, can be
obtained by studying the integrable Hamiltonian systems on the Lie
algebra $\textrm{so}(4)$. To be more precise, in this paper we
calculate some topological invariants of these integrable cases
using the theory of topological classification of integrable
Hamiltonian systems (see \cite{BolsinovFomenko99}). The obtained
results for the Lie algebra $\textrm{so}(4)$ allow us to make some
conclusions about the topological properties of the classical
Kovalevskaya case.

In this paper for the integrable cases on the algebra
$\textrm{so}(4)$  under consideration it is done the following:
\begin{enumerate}

\item the bifurcation diagrams of the
momentum mapping are constructed (Theorem ~\ref{T:Bif_Diag}),

\item the types of critical
points of rank $0$ are determined (Lemma~\ref{L:RankZeroType}),

\item the bifurcations of Liouville tori are described (Theorem~\ref{T:Bifurcations})
and the loop molecules for all singular points of the bifurcation
diagram are computed (Theorem~\ref{T:Molecules}).

\end{enumerate}

Using the results of this paper it is not hard to obtain the
classification of isoenergetic surfaces up to the rough Liouville
equivalence. In other words, for each isoenergetic surface
$H=\textrm{const}$ it is not hard to construct the corresponding
unmarked molecule (the Fomenko invariant). We do not do it in this
paper because the description of all possible molecules would be
rather cumbersome. The molecules depend not only on the type of
bifurcation diagrams, but also on the value of the Hamiltonian at
the critical points of rank $0$. The knowledge of loop molecules
also allows us to easily restore a number of marks for these rough
molecules. Recall that the marked molecules (the Fomenko--Zieschang
invariants) are important topological invariants of integrable
Hamiltonian systems, which completely describe the structure of the
Liouville foliation on non-singular three-dimensional isoenergetic
surfaces. Namely, the Fomenko--Zieschang theorem states that two
integrable Hamiltonian systems on two isoenergetic surfaces are
Liouville equivalent (that is, there exists a diffeomorphism taking
one Liouville foliation to another) if and only if their marked
molecules coincide (for more details about the Fomenko--Zieschang
invariants see, for example, \cite{BolsinovFomenko99}).

The idea to consider integrable Hamiltonian systems on compact Lie
algebras can become fruitful: the coadjoint orbits of compact Lie
algebras are compact which greatly simplifies the analysis of
integrable Hamiltonian systems on them. For example, in the case
under consideration, since the orbits are compact, for the
construction of bifurcation diagrams and the calculation of loop
molecules it suffices to find the curves that contain the image of
critical points and find the types of critical points of rank $0$.
Earlier integrable Hamiltonian systems on the Lie algebra
$\textrm{so}(4)$ were studied in the papers \cite{Oshemkov87}
(compact analogue of the Clebsch case), \cite{Oshemkov09} (the
Sokolov case) and \cite{Chorshidi06} (compact analogue of the
Steklov case). Let us also note that algebraic and topological
properties of integrable systems related to the Lie algebra
$\textrm{so}(4)$ and other compact Lie algebras were studied in
\cite{Mischenko78} and \cite{Fomenko88}.

The topology of the classical Kovalevskaya case was studied in
detail in the book \cite{Harlamov88} (see also \cite{Harlamov83App}
and \cite{Harlamov83Dan}). In particular, in this book the
bifurcation diagrams of the momentum mapping and the bifurcations of
Liouville tori for the critical values of the momentum mapping are
described. The loop molecules (as well as fine Liouville
classification) for the classical case Kovalevskaya are contained in
the paper \cite{BolsinovFomenkoRichter00}. All necessary results on
the classical Kovalevskaya case are collected in a convenient for us
form in the book \cite{BolsinovFomenko99}. The connection between
the classical Kovalevskaya case and the considered in this paper
system on the Lie algebra $\textrm{so}(4)$ is discussed in
Section~\ref{S:Classival_Kovalevskaya}.

The case of the zero area constant ($b=0$) was described earlier in
the paper \cite{Kozlov12}.

\section{Basic Definitions and Problem Formulation} \label{s2}

Recall that there exists a natural linear Poisson bracket on the
dual space $\mathfrak{g}^*$ to any finite-dimensional Lie algebra
$\mathfrak{g}$ given by the formula
\begin{equation} \label{Eq:Lie_Poisson_Bracket} \{ f, g\} = \langle
x, [df|_x, \, dg|_x ]\rangle.
\end{equation} Here $\langle \cdot  , \cdot \rangle $ denotes the
value of the covector in $\mathfrak{g}^*$ on a vector in
$\mathfrak{g}$, and $[ \cdot, \cdot ]$  denotes the commutator in
the Lie algebra $\mathfrak{g}$. In the formula
\eqref{Eq:Lie_Poisson_Bracket} we use the canonical isomorphism
$(\mathfrak{g}^*)^* = \mathfrak{g}$. The bracket
\eqref{Eq:Lie_Poisson_Bracket} is called the Lie-Poisson bracket.

\begin{definition} Let $\mathfrak{g}$ be a finite-dimensional Lie algebra, $x_1, \dots,
x_n$ be linear coordinates in the dual space $\mathfrak{g}^*$, and
$H$ be a smooth function on $\mathfrak{g}^*$. The equations
\begin{equation} \dot{x}_i = \{ x_i, H \},\end{equation} which define a dynamical system on $\mathfrak{g}^*$, are called {\it
Euler's equations} for the Lie algebra $\mathfrak{g}$.
\end{definition}

It is well-known (see, for example, \cite{BolsinovFomenko99}), that
the classical Kovalevskaya case can be defined by Euler's equations
on the Lie algebra $\textrm{e}(3)$. It turns out that the classical
Kovalevskaya case can be included in a one-parameter family of
integrable systems defined on the pencil of Lie algebras
$\textrm{so}(4) - \textrm{e}(3) - \textrm{so}(3, 1)$.

Consider the six-dimensional space $\mathbb{R}^6$ and fix a basis
$e_1, e_2, e_3, f_1, f_2, f_3$ in it. Consider the following
one-parameter family of commutators in $\mathbb{R}^6$ depending on
the parameter $\varkappa \in \mathbb{R}$:
\begin{equation}[e_i, e_j] = \varepsilon_{ijk}e_k, \quad [e_i, f_j] =
\varepsilon_{ijk}f_k, \quad [f_i, f_j] = \varkappa
\varepsilon_{ijk}e_k,
\end{equation} where $\varepsilon_{ijk}$ is the sign of the permutation
$\{123\} \rightarrow \{ijk\}$. It is easy to check that the cases
$\varkappa>0$, $\varkappa=0$ and $\varkappa<0$ correspond to the Lie
algebras $\textrm{so}(4)$, $\textrm{e}(3)$ and $\textrm{so}(3, 1)$
respectively.

In the coordinates $J_1, J_2, J_3, x_1, x_2, x_3$ on the dual linear
space corresponding to the basis $e_1, e_2, e_3, f_1, f_2, f_3$ the
Lie--Poisson bracket has a similar form
\begin{equation} \label{Eq:SO4_Poisson_Lie_bracket} \{J_i, J_j\} =
\varepsilon_{ijk}J_k, \quad \{J_i, x_j\} = \varepsilon_{ijk}x_k,
\quad \{x_i, x_j\} = \varkappa \varepsilon_{ijk}J_k. \end{equation}
For any value of the parameter $\varkappa \in \mathbb{R}$ the
bracket \eqref{Eq:SO4_Poisson_Lie_bracket} has two Casimir
functions:
\begin{equation}\label{Eq:Kasimirs}  f_1 = \textbf{x}^2 + \varkappa
\textbf{J}^2, \qquad  f_2 = \langle \textbf{x}, \textbf{J} \rangle,
\end{equation} where $\textbf{x}$ and $\textbf{J}$ denote the three-dimensional
vectors $(x_1, x_2, x_3)$ and $(J_1, J_2, J_3)$ respectively and
$\langle \cdot, \cdot \rangle$ denotes the Euclidean scalar product
of two vectors in $\mathbb{R}^3$. Recall that functions are called
Casimir functions of a Poisson bracket if they commute with any
other function with respect to this bracket. The joint level
surfaces \begin{equation} M_{a, b} = \{ (\textbf{J}, \textbf{x})|
\quad  f_1 (\textbf{J}, \textbf{x}) = a, \quad f_2 (\textbf{J},
\textbf{x}) = b \}
\end{equation} are the orbits of the coadjoint representation except
for the case $\varkappa \leq 0, a=0, b=0$ (in this case, the level
surface is a union of several orbits of the coadjoint
representation). In all other cases the surfaces $M_{a, b}$ are
symplectic leaves of the bracket \eqref{Eq:SO4_Poisson_Lie_bracket},
in particular, the bracket \eqref{Eq:SO4_Poisson_Lie_bracket}
defines a symplectic structure on them. If $\varkappa>0$ and $a> 2
\sqrt{\varkappa} |b|$, then the orbits $M_{a, b}$ are
four-dimensional submanifolds of $\mathbb{R}^6(\textbf{J},
\textbf{x})$  diffeomorphic to the product of two two-dimensional
spheres $\mathbb{S}^2 \times \mathbb{S}^2$. If $\varkappa>0$, then
the singular orbits $a = 2 \sqrt{\varkappa} |b|$ are diffeomorphic
to the two-dimensional sphere $\mathbb{S}^2$. If $\varkappa > 0$,
then there is no orbits satisfying the condition $a < 2
\sqrt{\varkappa} |b|$. Let us also note that if $\varkappa =0 $,
then the non-singular orbits $a> 0$ are diffeomorphic to the
cotangent bundle of the two-dimensional sphere $T^* \mathbb{S}^2$
(in particular, they are not compact).

In this paper we examine the following integrable case for Euler's
equations defined on the pencil of Lie algebras
$\textrm{so}(4)-\textrm{e}(3)-\textrm{so}(3, 1)$ described above
(see, for example, \cite{Komarov81} or \cite{Borisov03}). In the
coordinates $(J_i, x_j)$ described above the Hamiltonian is equal to
\begin{equation}\label{Eq:Hamiltonian} H = J_1^2 + J_2^2 + 2J_3^2 + 2 c_1 x_1\end{equation} and the integral
has the form \begin{equation}\label{Eq:First_Integral} K = (J_1^2 -
J_2^2-2c_1 x_1 + \varkappa c_1^2)^2 + (2J_1J_2 - 2 c_1 x_2)^2,
\end{equation} where $c_1$ is an arbitrary constant.

Note that without loss of generality we may assume that $c_1=1$ and
$\varkappa =-1, 0$ or $1$ since the change of coordinates and
parameters of the system given by the formulas
\[J' = \mu J, \quad x' = \lambda \mu x, \quad a' = \mu^2 \lambda^2
a, \quad b' = \mu^2 \lambda b, \quad c_1' = \frac{\mu}{\lambda} c_1,
\quad \varkappa' = \lambda^2 \varkappa,
\] where $\lambda, \mu$ are arbitrary constants, multiplies the Hamiltonian
and the integral by $\mu^2$ and $\mu^4$ respectively. Nevertheless
we will preserve both $c_1$ and $\varkappa$ in order to get
``homogeneous'' formulas (for example, the Hamiltonian and the
integral are homogeneous functions of $\textbf{J}, \textbf{x},
c_1$).

\begin{remark}
It is not hard to verify that if $\varkappa = 0$ (and $c_1 =1$),
then we obtain the classical Kovalevskaya case in the form in which
it was described, for example, in the book \cite{BolsinovFomenko99}.
More precisely, in the book \cite{BolsinovFomenko99} the Hamiltonian
is $2$ times less than the Hamiltonian \eqref{Eq:Hamiltonian} and
the first integral is $4$ times less than the integral
\eqref{Eq:First_Integral}. Note also that in the book
\cite{BolsinovFomenko99} the following notation is used: $S_i = J_i,
R_j = x_j$, the value $b$ of the integral $f_2$ is denoted by $g$
and the value $a$ of the integral $f_1$ is assumed to be equal to
$1$.
\end{remark}

\begin{remark} \label{R:minusB} The change of coordinates $(\textbf{J}, \textbf{x}) \to (-\textbf{J},
\textbf{x})$ preserves the integral $f_1$, the Hamiltonian
\eqref{Eq:Hamiltonian} and the integral \eqref{Eq:First_Integral}
whereas it changes the sign of the integral $f_2$. Therefore,
without loss of generality, we can assume that $b \geq 0$.
\end{remark}

\begin{remark} Note also that the system has the following two natural symmetries
that preserve the Hamiltonian \eqref{Eq:Hamiltonian}, the first
integral \eqref{Eq:First_Integral} and both integral $f_1$ and
$f_2$. The first symmetry $\sigma_2$ changes the signs of the
coordinates $J_2$ and $x_2$ and preserves the remaining coordinates:
\[ \sigma_2: (J_1, J_2, J_3, x_1, x_2, x_3) \to (J_1, -J_2, J_3,
x_1, -x_2, x_3).\] Similarly, the second symmetry $\sigma_3$
simultaneously changes the signs of the coordinates $J_3$ and
$x_3$:\[ \sigma_3: (J_1, J_2, J_3, x_1, x_2, x_3) \to (J_1, J_2,
-J_3, x_1, x_2, -x_3).\] \end{remark}

Further we construct the bifurcation diagrams of the momentum
mapping, determine the types of critical points, describe the
bifurcations of Liouville tori and calculate the loop molecules for
singular points of the bifurcation diagrams for the given
Hamiltonian systems with Hamiltonian \eqref{Eq:Hamiltonian} and
integral \eqref{Eq:First_Integral} on non-singular orbits $M_{a,
b}$. We use some facts and notation from the theory of topological
classification developed by A.\,T.~Fomenko and his disciples. The
definitions of topological invariants (atoms, molecules) as well as
the basic facts about this theory can be found in the book
\cite{BolsinovFomenko99}. For various applications of this theory in
rigid body dynamics see, for example, \cite{Fomenko96}. Let us also
note that this theory has recently played an important role in the
study of symmetries of Liouville tori bifurcations and in the
construction of a classification theory for such symmetries (see
\cite{Kudryavtseva08}, \cite{Fomenko12}, \cite{Kudryavtseva12} and
\cite{Kudryavtseva13}).

In this paper we only briefly recall the idea of the method of loop
molecules. A smooth curve without self-intersections in the plane
$\mathbb{R}^2(H, K)$ is called admissible if it intersects the
bifurcation diagram transversally and does not pass through its
singular points. The preimage of any admissible curve is a
three-dimensional manifold equipped with a Liouville foliation. An
invariant of this foliation is a marked molecule, which is a graph
whose edges correspond to the one-parameter families of Liouville
tori and whose vertices correspond to the singular leaves of this
foliation. Symbols are placed in the vertices of the graph which
specify the types of the bifurcations. In addition the graph is
endowed, in a certain way, with three types of marks ($r$,
$\epsilon$ and $n$), which indicate the relation between different
bifurcations.  The loop molecule of a singular point $x$ of a
bifurcation diagram is the marked molecule that describes the
Liouville foliation in the preimage of any sufficiently small
admissible closed curve surrounding the point $x$. Loop molecules
can provide information about different molecules of admissible
curves (for example, sometimes molecules of curves can be ``glued''
together from parts of loop molecules). Examples of loop molecules
are given in Tables \ref{Tab:Cirle_Mol_Old} and
\ref{Tab:Cirle_Mol_New}.

\section{Main results} \label{s3}

We now state the main results. First, we describe the results for
the case $b \ne 0$ and then for the case $b=0$. Let us start with
the description of the bifurcation diagrams of the momentum mapping.

\begin{lemma} \label{L:Curve_Images_B_not_0_Kappa_not_0}
Let $b \ne 0$ and $\varkappa \ne 0$. Then for any non-singular orbit
$M_{a, b}$ (that is, for any orbit such that $a^2 - 4 \varkappa b^2
\ne 0$) the bifurcation diagram $\Sigma_{h, k}$ for the integrable
Hamiltonian system with Hamiltonian \eqref{Eq:Hamiltonian} and
integral \eqref{Eq:First_Integral} is contained in the union of the
following three families of curves on the plane
$\mathbb{R}^{2}(h,k)$:
\begin{enumerate}
\item The line $k=0$; \item The parametric curve
\begin{equation} \label{Eq:Parametric Curve} h(z)= \frac{b^2 c_1^2}{z^2}+2 z, \quad k(z)= \left(4 a
c_1^2-\frac{4 b^2 c_1^2}{z}+\frac{b^4 c_1^4}{z^4}\right)-2 \varkappa
c_1^2 h(z) +\varkappa^2 c_1^4, \end{equation} where $z \in
\mathbb{R} - \{0\}$.
\item The union of two parabolas \begin{equation} \label{Eq:Left Parabola} k =
\left(h-\varkappa c_1^2 -\frac{a}{\varkappa } +\frac{\sqrt{a^2-4
\varkappa b^2 }}{\varkappa }\right)^2 \end{equation} and
\begin{equation} \label{Eq:Right Parabola} k = \left(h-\varkappa c_1^2
-\frac{a}{\varkappa } - \frac{\sqrt{a^2-4 \varkappa b^2 }}{\varkappa
}\right)^2. \end{equation} \end{enumerate}
\end{lemma}

The proof of Lemma \ref{L:Curve_Images_B_not_0_Kappa_not_0} is given
in Section \ref{SubS:CritRank1}.

In order to construct the bifurcation diagrams of the momentum
mapping it remains to throw away several parts of the curves
described in Lemma~\ref{L:Curve_Images_B_not_0_Kappa_not_0}. The
precise description of the bifurcation diagrams is given in the
following theorem.

\begin{figure}[!htb]
\minipage{0.48\textwidth}
\includegraphics[width=\textwidth]{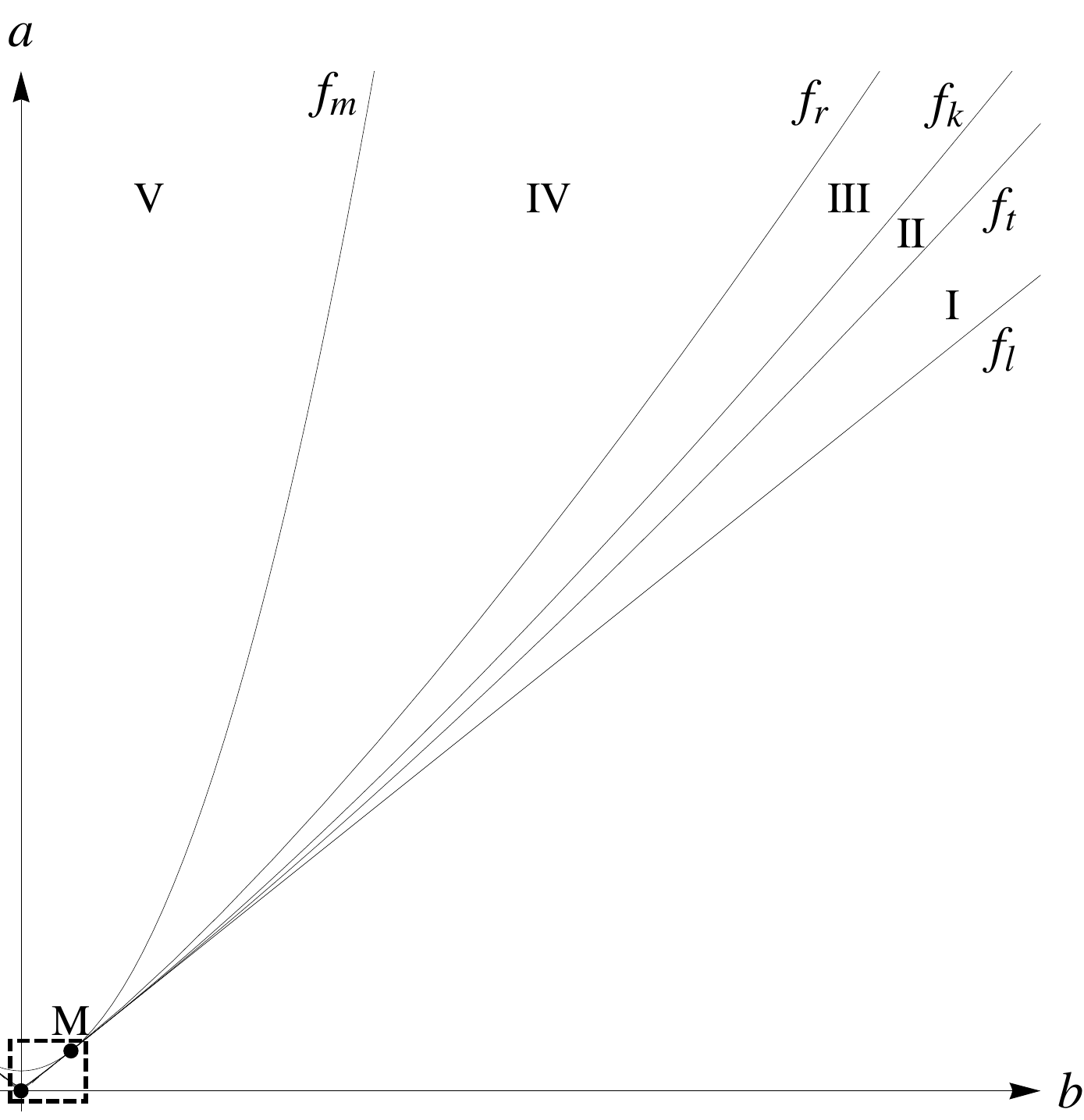}
   \caption{Partition of the set of parameters} \label{Fig:Areas_Big}
\endminipage
\hspace{0.04\textwidth}
\minipage{0.48\textwidth}
\includegraphics[width=\textwidth]{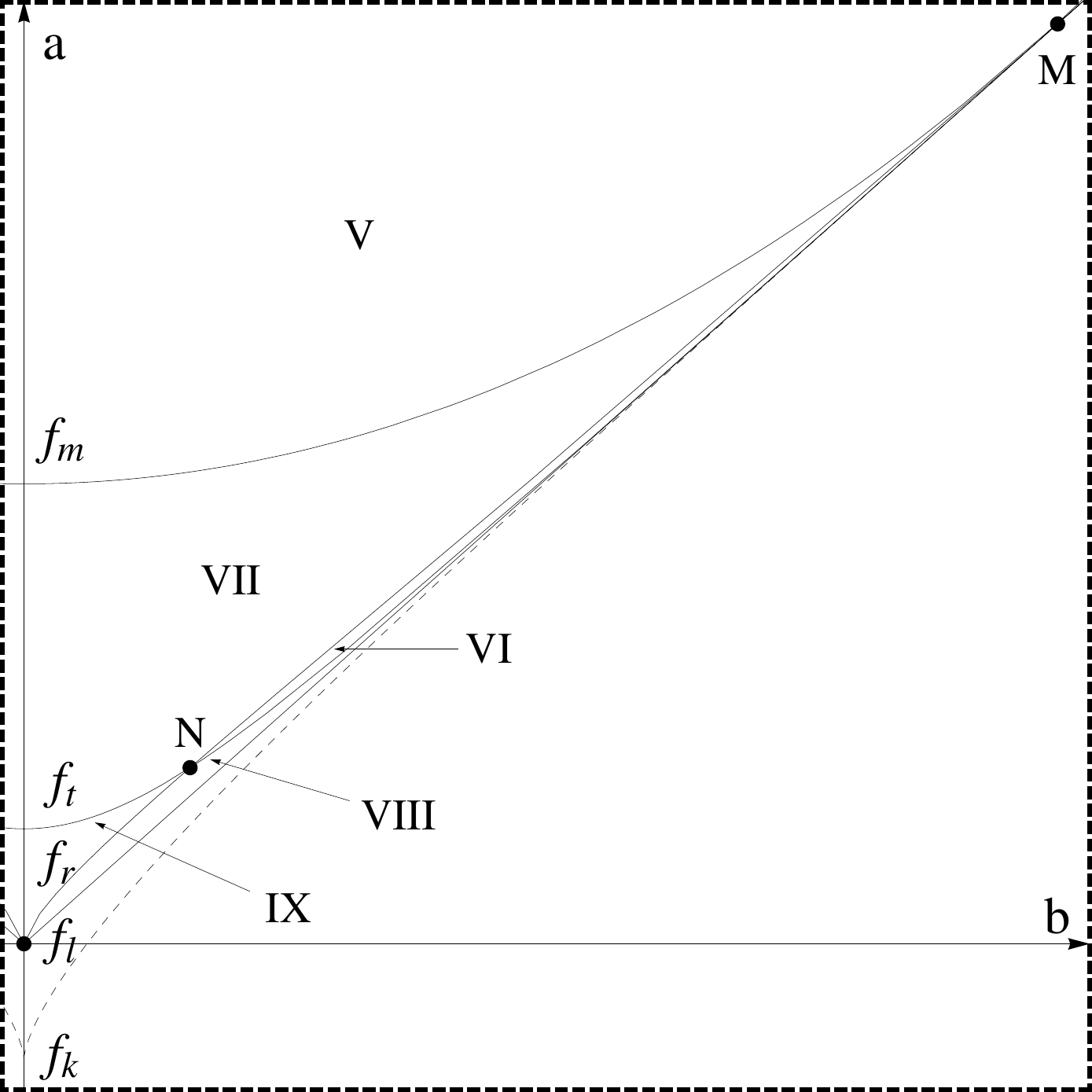}
   \caption{An enlarged fragment of Fig. \ref{Fig:Areas_Big}} \label{Fig:Areas_Small}
\endminipage
\end{figure}

\begin{theorem} \label{T:Bif_Diag}
Let $\varkappa > 0$ and $b > 0$. Then the functions $f_k, f_r, f_m,
f_t$ and $f_l$ given by the formulas
\begin{gather}
\label{Eq: Function_F_K} f_k(b) = \frac{3 b^{4/3}+6 \varkappa
b^{2/3} c_1^{4/3} - \varkappa ^2c_1^{8/3} }{4 c_1^{2/3}} \\
\label{Eq:Function_F_r} f_r(b) = \frac{b^{4/3}}{c_1^{2/3}}+
\varkappa b^{2/3}c_1^{2/3}
 \\ \label{Eq:Function_F_m} f_m(b)= \frac{b^2}{\varkappa c_1^2} +
 \varkappa^2c_1^2 \\ \label{Eq:Triple_Intersect} f_t(b) = \left(\frac{\varkappa c_1^2 + t^2}{2 c_1} \right)^2 + \varkappa t^2,
\qquad \text{where} \quad b =t \left(\frac{\varkappa c_1^2 + t^2}{2
c_1} \right) \\ \label{Eq:Function_F_l} f_l (b) = 2 \sqrt{\varkappa}
|b|
\end{gather} divide the area $\{b>0, a> 2 \sqrt{\varkappa} b\} \subset
\mathbb{R}^2(a,b)$ into $9$ areas (see Fig. \ref{Fig:Areas_Big} and
\ref{Fig:Areas_Small}). In Fig. \ref{Fig:Koval_U_1_A} --
\ref{Fig:Koval_D_2_B} the corresponding bifurcation diagrams of the
momentum mapping for the integrable Hamiltonian system with
Hamiltonian \eqref{Eq:Hamiltonian} and integral
\eqref{Eq:First_Integral} on the orbit $M_{a, b}$ of the Lie algebra
$\textrm{so}(4)$ are shown for each of these areas. More precisely,
in each case it is specified from which parts of the line $k=0$, the
curve \eqref{Eq:Parametric Curve} and two parabolas \eqref{Eq:Left
Parabola} and \eqref{Eq:Right Parabola} the bifurcation diagram of
the momentum mapping is composed.

\begin{figure}[!htb]
\minipage{0.48\textwidth}
    \includegraphics[width=\linewidth]{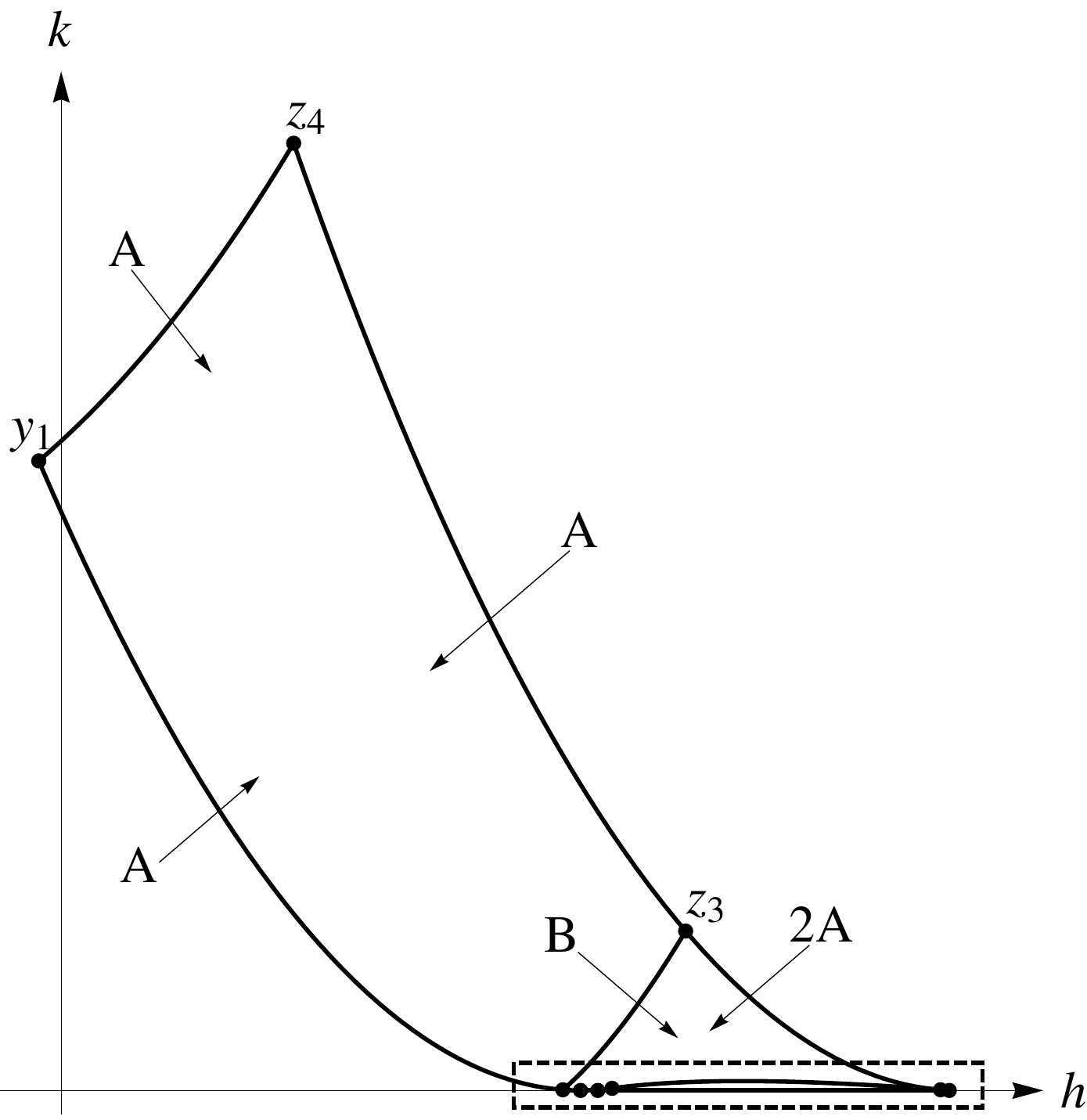}
   \caption{Area \textrm{I}: $\varkappa^3 c_1^4 < b^2, \quad f_l(b) <a <  f_t(b)$}
   \label{Fig:Koval_U_1_A}
\endminipage
\hspace{0.04\textwidth}
\minipage{0.48\textwidth}
    \includegraphics[width=\linewidth]{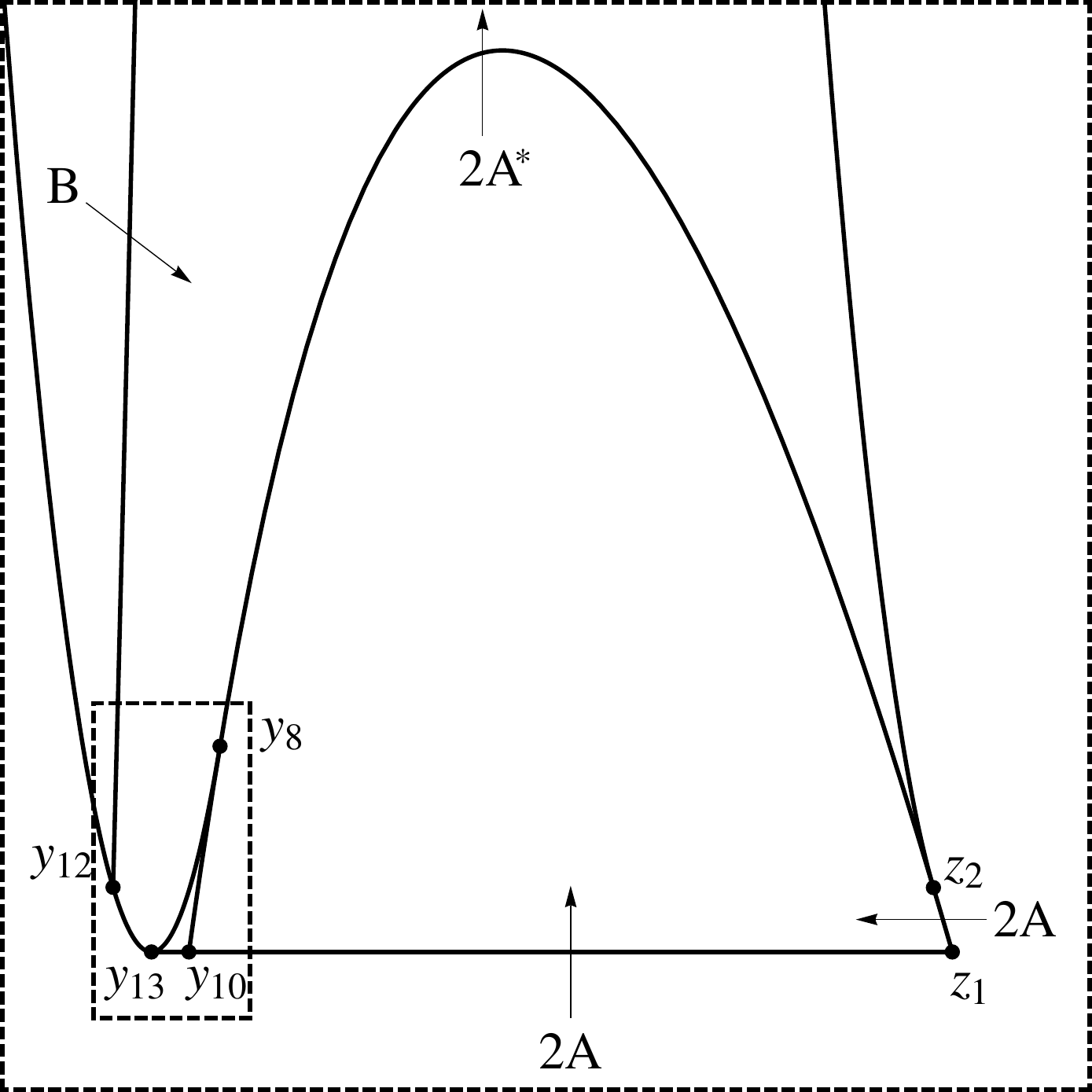}
   \caption{Area \textrm{I}: an enlarged fragment of Fig. \ref{Fig:Koval_U_1_A}}
   \label{Fig:Koval_U_1_B}
\endminipage
\end{figure}

\begin{figure}[!htb]
\minipage{0.48\textwidth}
    \includegraphics[width=\linewidth]{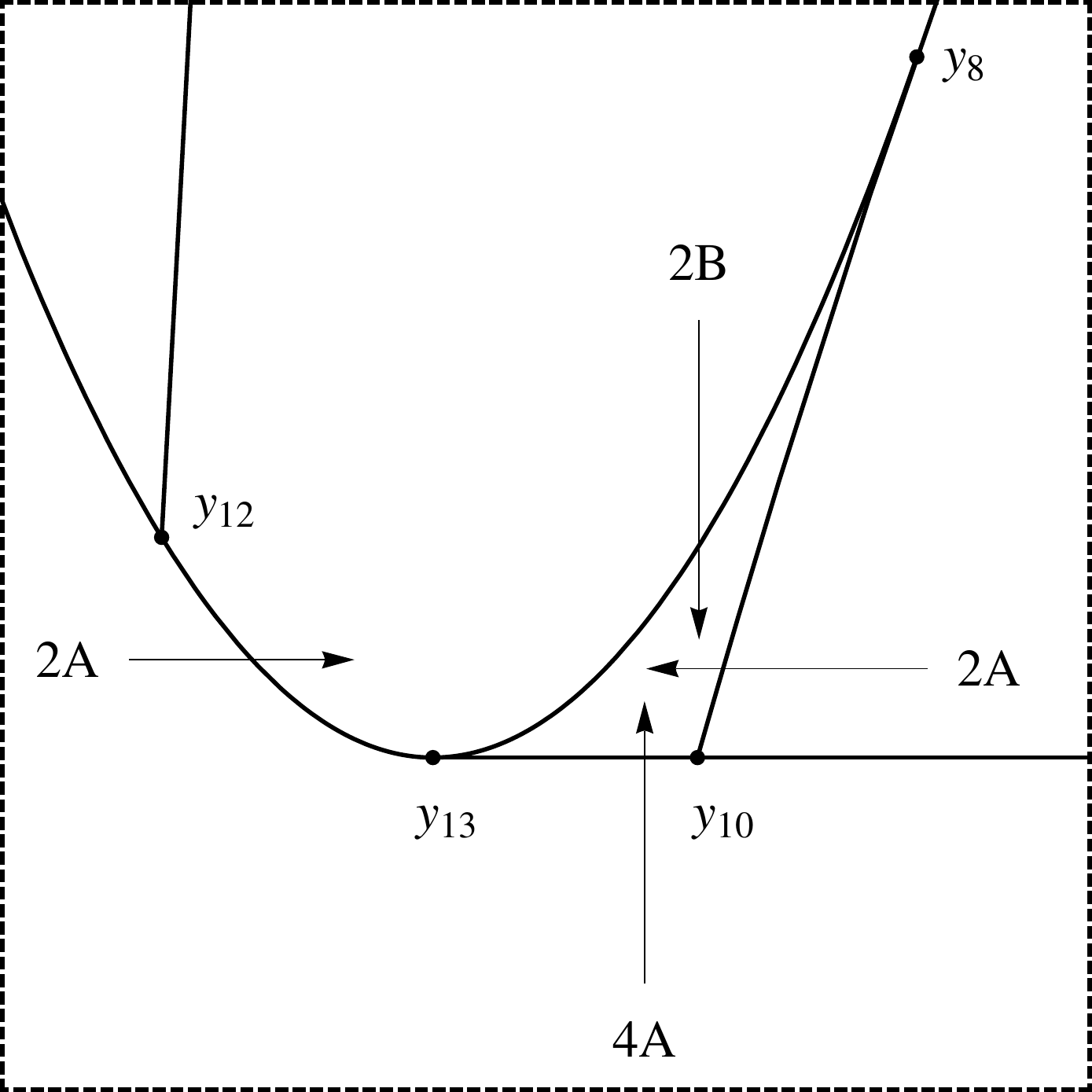}
   \caption{Area \textrm{I}: an enlarged fragment of Fig. \ref{Fig:Koval_U_1_B} }
   \label{Fig:Koval_U_1_C}
\endminipage
\hspace{0.04\textwidth}
\minipage{0.48\textwidth}
    \includegraphics[width=\linewidth]{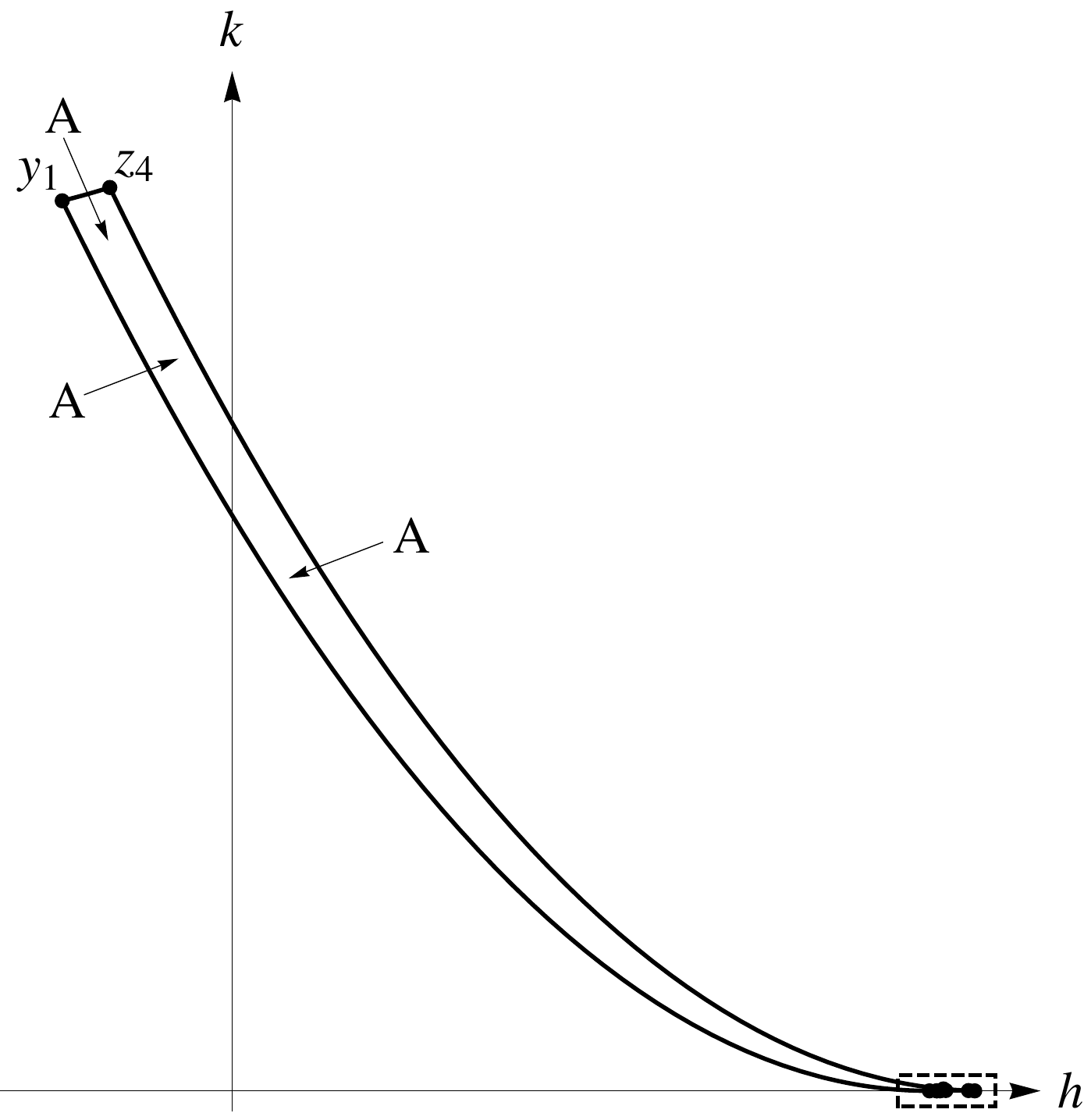}
   \caption{Area \textrm{II}: $\varkappa^3 c_1^4 < b^2, \quad f_t(b) <a< f_k(b)$}
   \label{Fig:Koval_U_2_A}
\endminipage
\end{figure}

\begin{figure}[!htb]
\minipage{0.48\textwidth}
    \includegraphics[width=\linewidth]{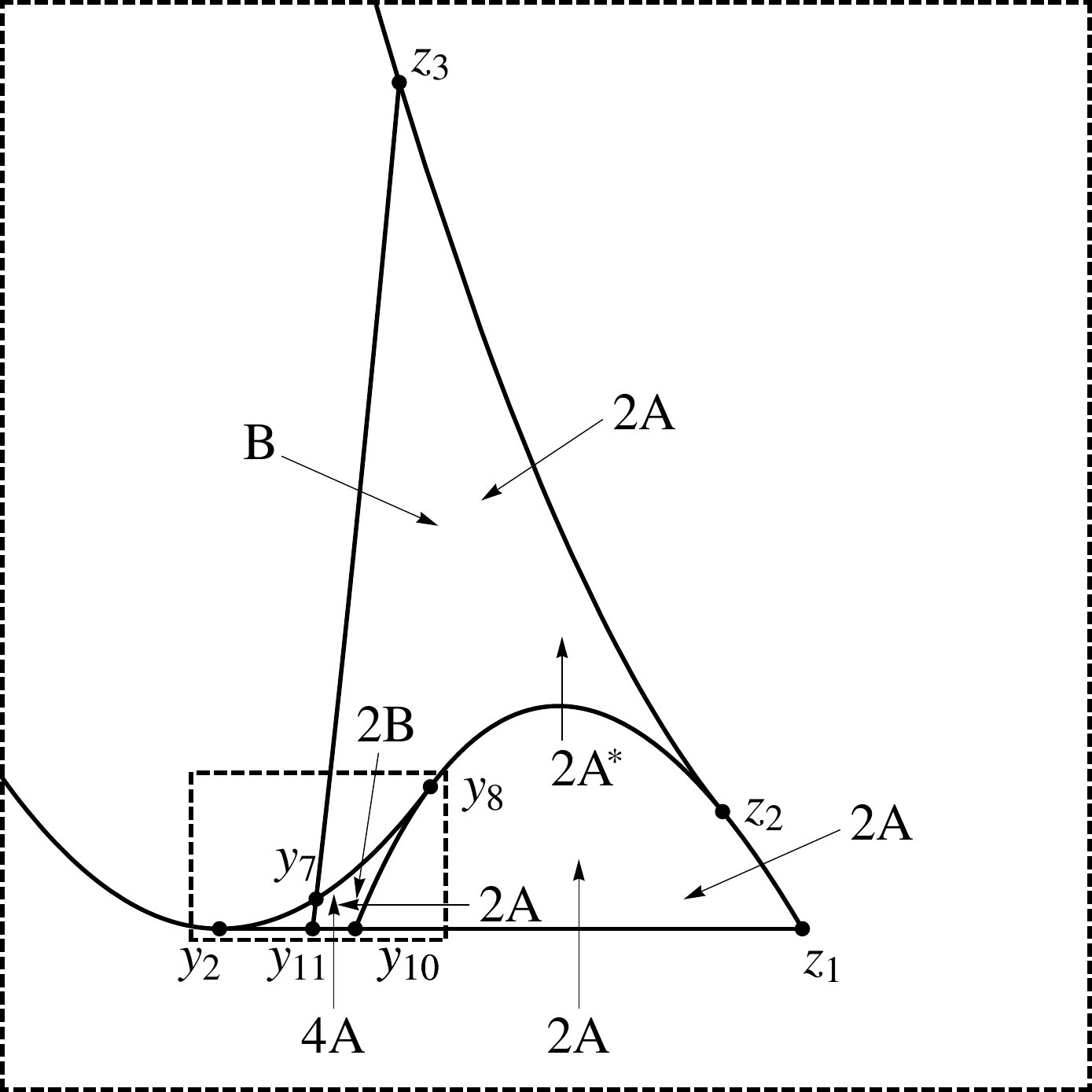}
   \caption{Area \textrm{II}: an enlarged fragment of Fig. \ref{Fig:Koval_U_2_A}}
   \label{Fig:Koval_U_2_B}
\endminipage
\hspace{0.04\textwidth}
\minipage{0.48\textwidth}
    \includegraphics[width=\linewidth]{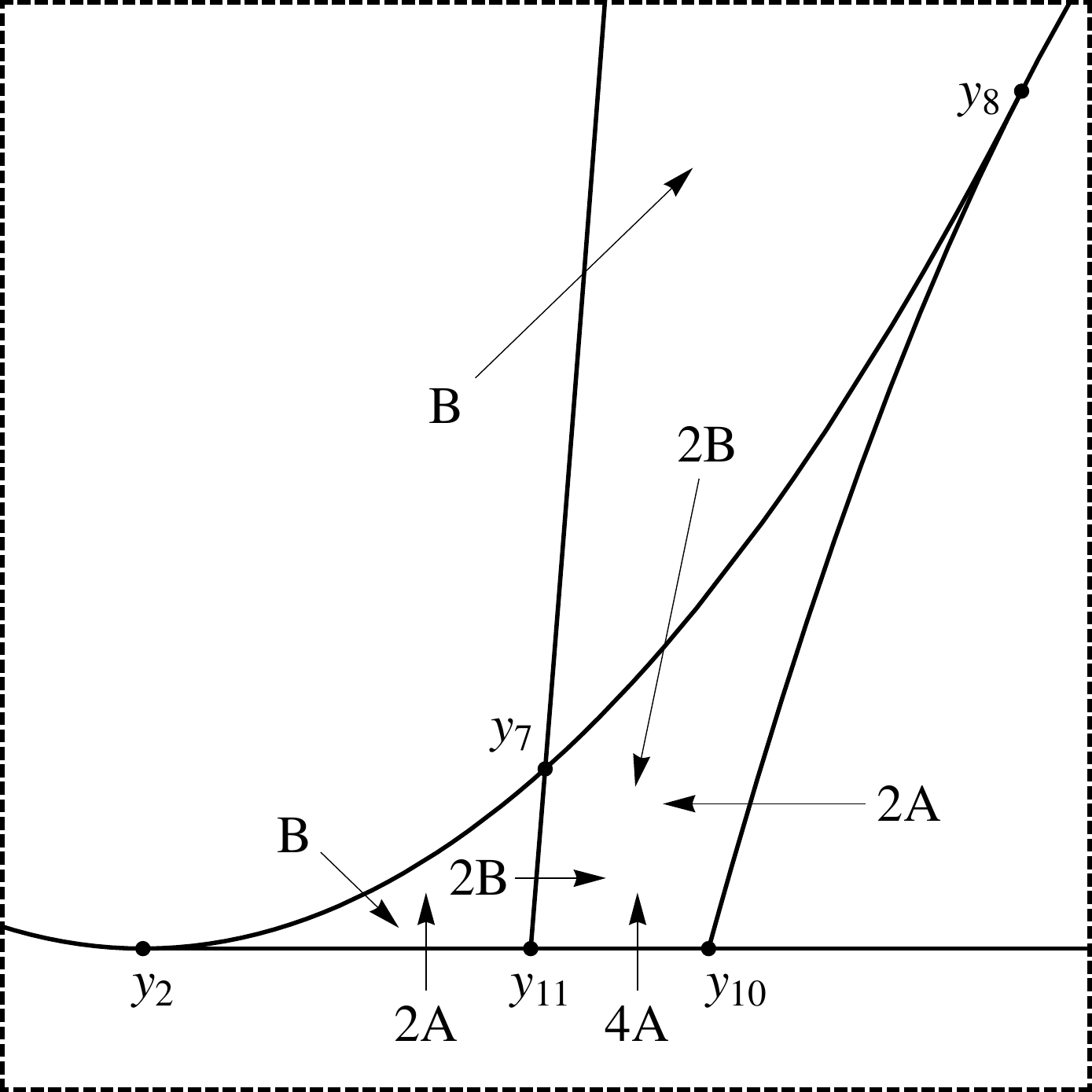}
   \caption{Area \textrm{II}: an enlarged fragment of Fig. \ref{Fig:Koval_U_2_B}}
   \label{Fig:Koval_U_2_C}
\endminipage
\end{figure}

\begin{figure}[!htb]
\minipage{0.48\textwidth}
    \includegraphics[width=\linewidth]{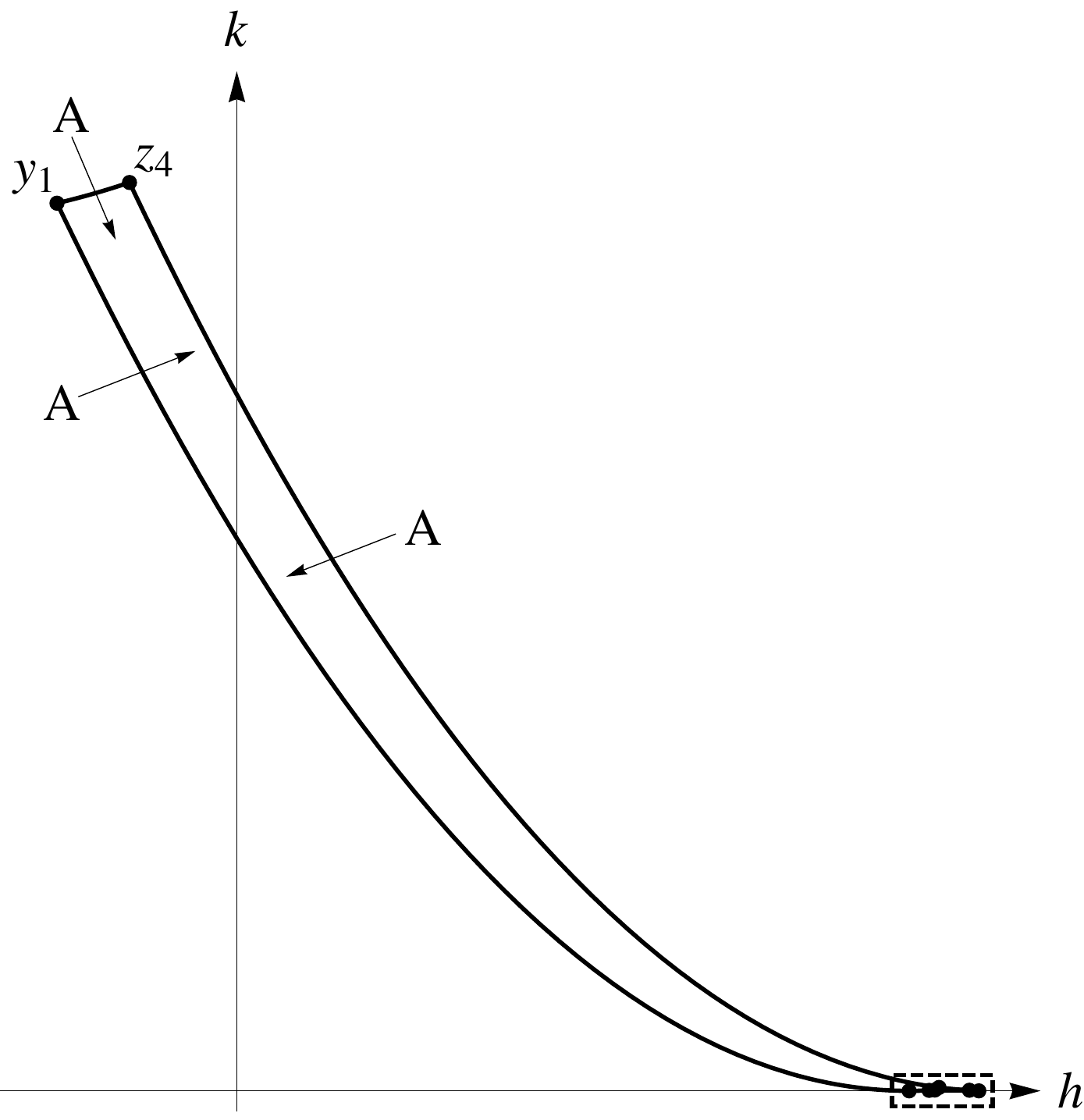}
   \caption{Area \textrm{III}: $\varkappa^3 c_1^4 < b^2, \quad f_k(b) <a< f_r(b)$}
   \label{Fig:Koval_U_3_A}
\endminipage
\hspace{0.04\textwidth}
\minipage{0.48\textwidth}
    \includegraphics[width=\linewidth]{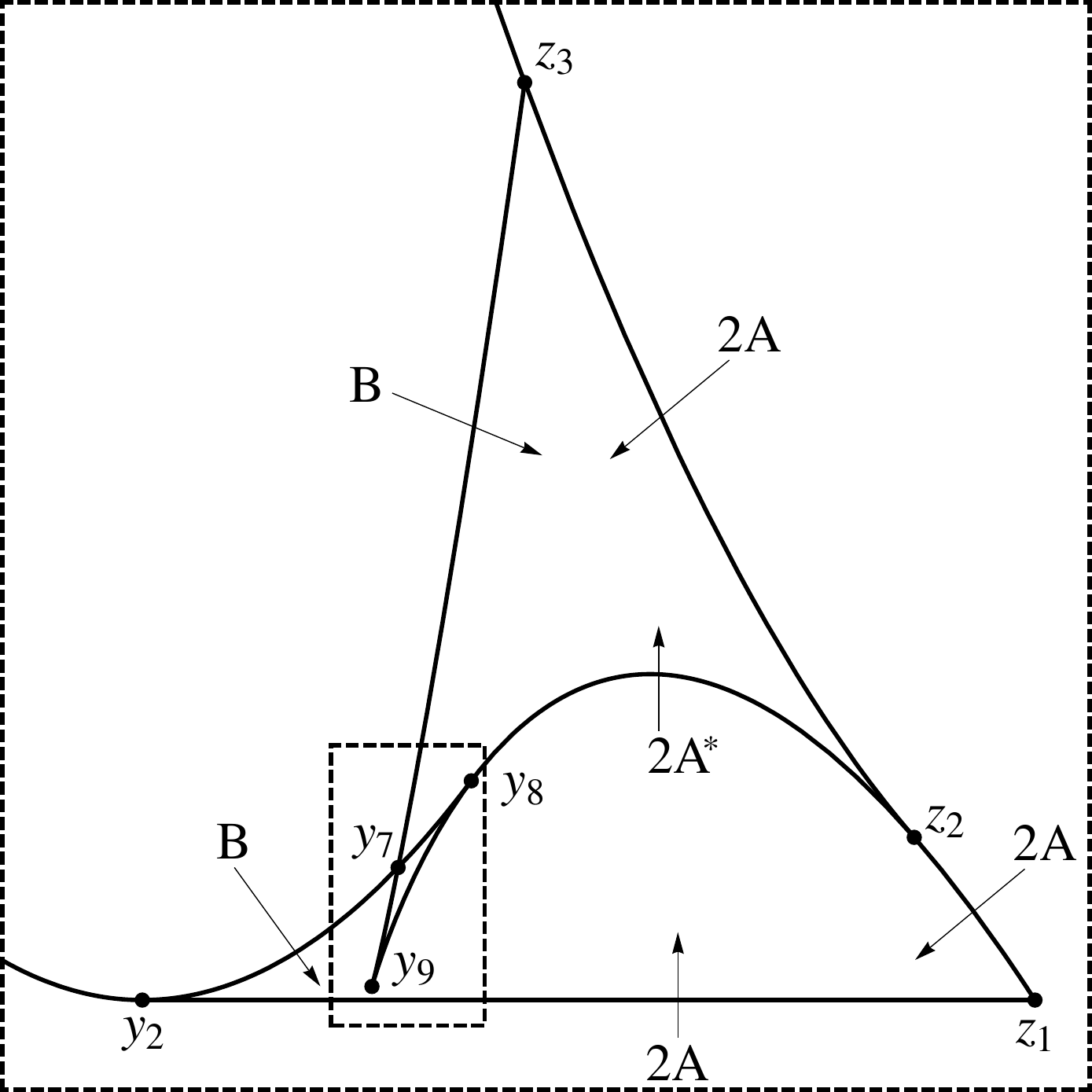}
   \caption{Area \textrm{III}: an enlarged fragment of Fig.
   \ref{Fig:Koval_U_3_A}}
   \label{Fig:Koval_U_3_B}
\endminipage
\end{figure}

\begin{figure}[!htb]
\minipage{0.48\textwidth}
    \includegraphics[width=\linewidth]{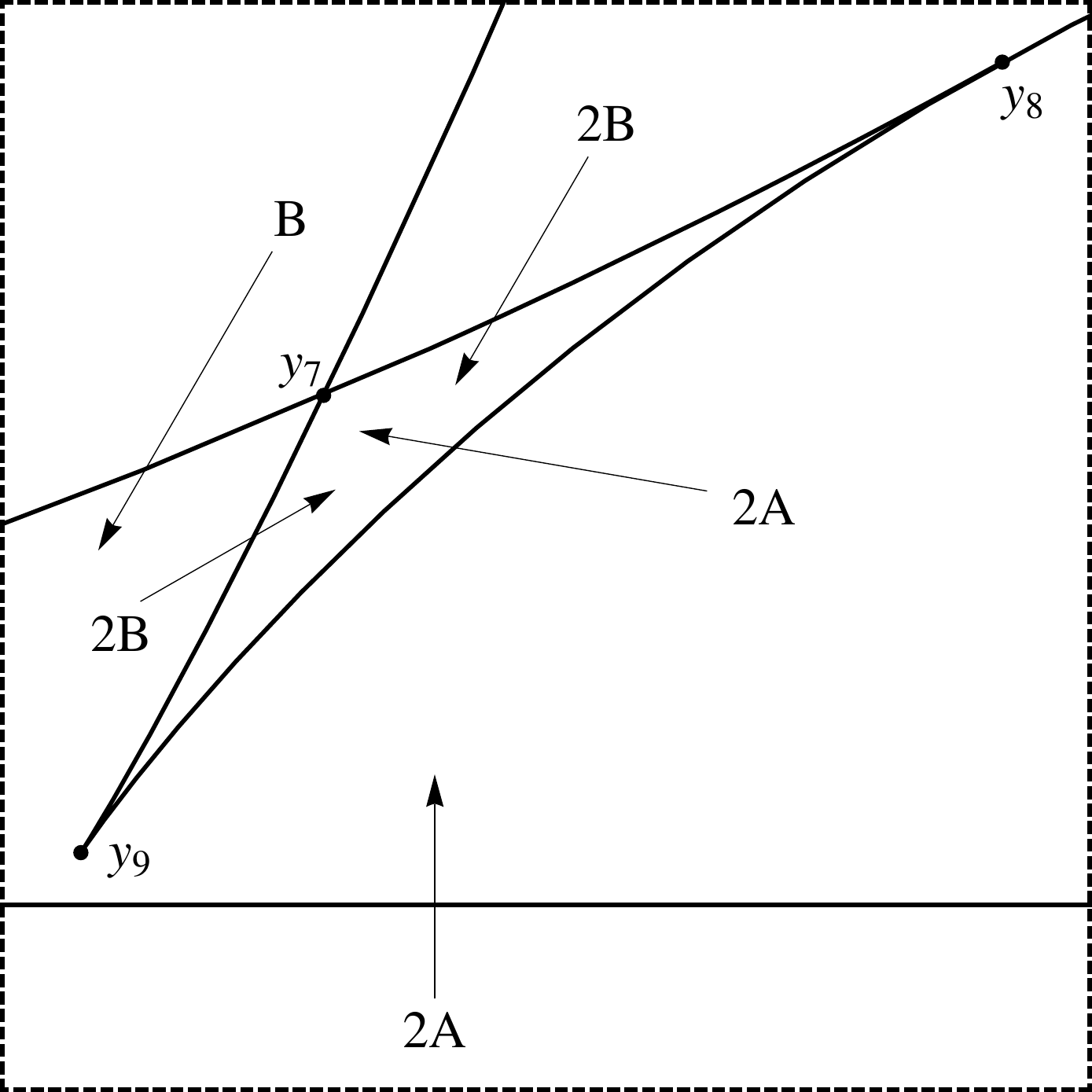}
   \caption{Area \textrm{III}: an enlarged fragment of Fig.\ref{Fig:Koval_U_3_B}}
   \label{Fig:Koval_U_3_C}
\endminipage
\hspace{0.04\textwidth}
\minipage{0.48\textwidth}
    \includegraphics[width=\linewidth]{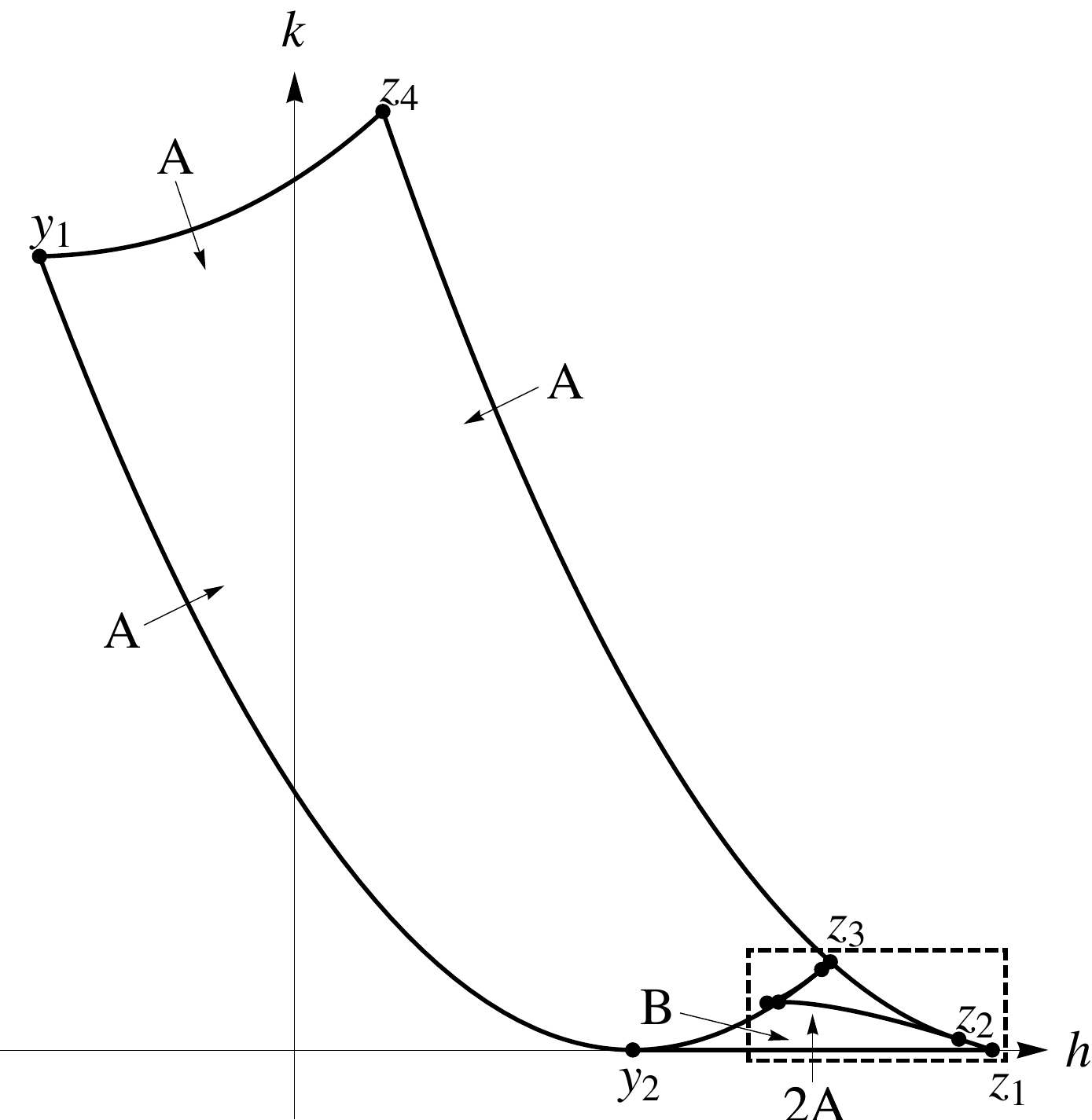}
   \caption{Area \textrm{IV}: $\varkappa^3 c_1^4 < b^2, \quad f_r(b) <a< f_m(b)$}
   \label{Fig:Koval_U_4_A}
\endminipage
\end{figure}

\begin{figure}[!htb]
\minipage{0.48\textwidth}
    \includegraphics[width=\linewidth]{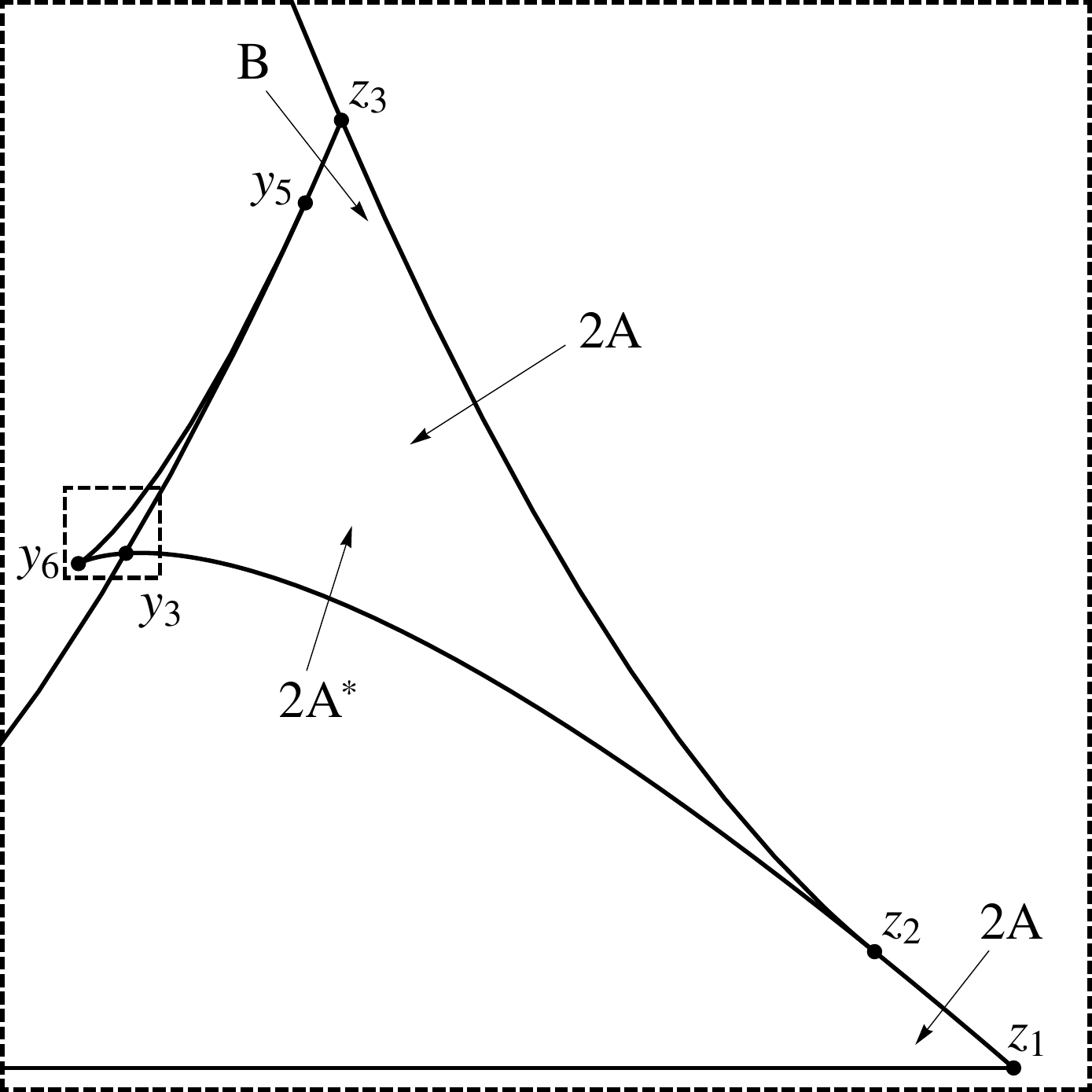}
   \caption{Area \textrm{IV}: an enlarged fragment of Fig.
   \ref{Fig:Koval_U_4_A}}
   \label{Fig:Koval_U_4_B}
\endminipage
\hspace{0.04\textwidth}
\minipage{0.48\textwidth}
    \includegraphics[width=\linewidth]{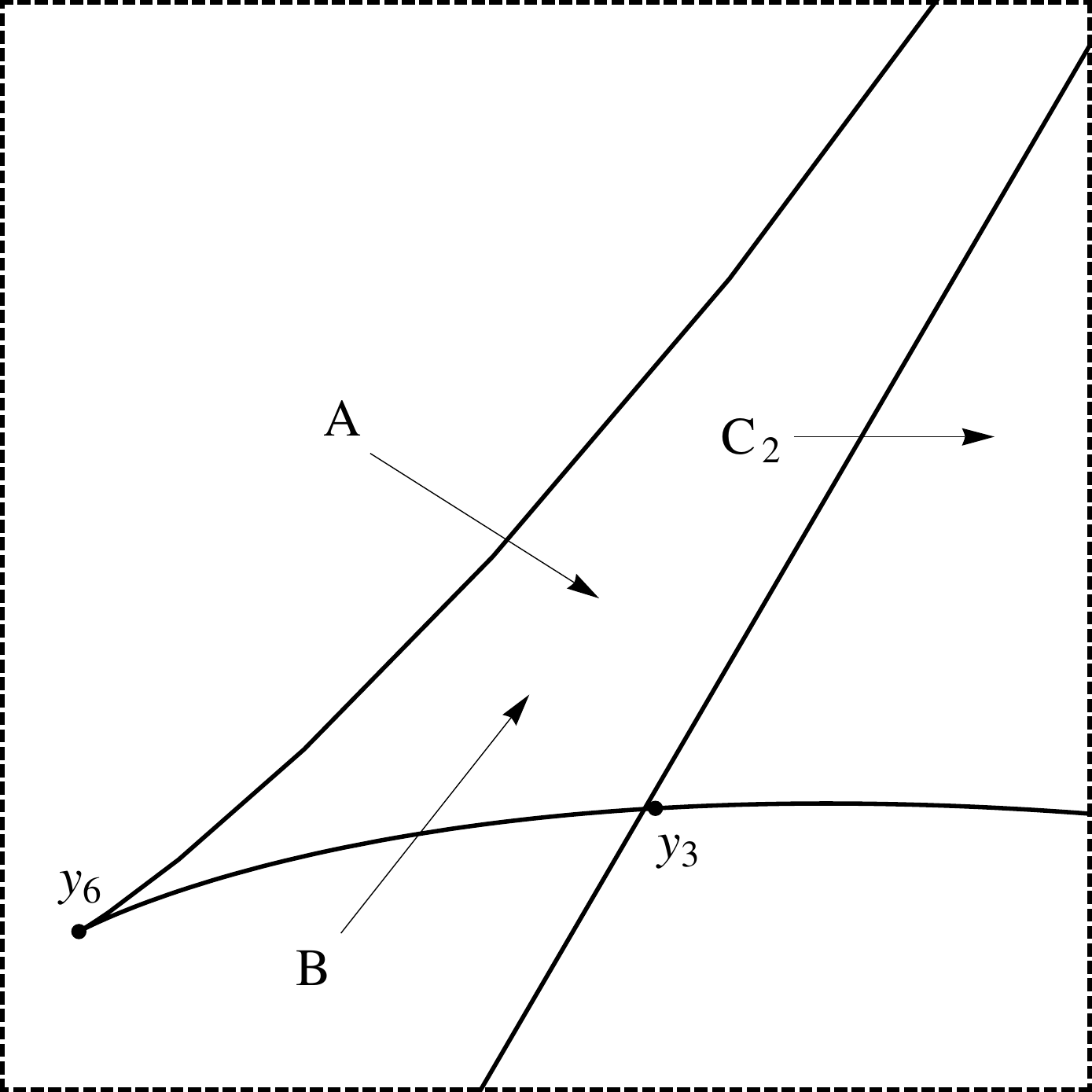}
   \caption{Area \textrm{IV}: an enlarged fragment of Fig.\ref{Fig:Koval_U_4_B}}
   \label{Fig:Koval_U_4_C}
\endminipage
\end{figure}

\begin{figure}[!htb]
\minipage{0.48\textwidth}
    \includegraphics[width=\linewidth]{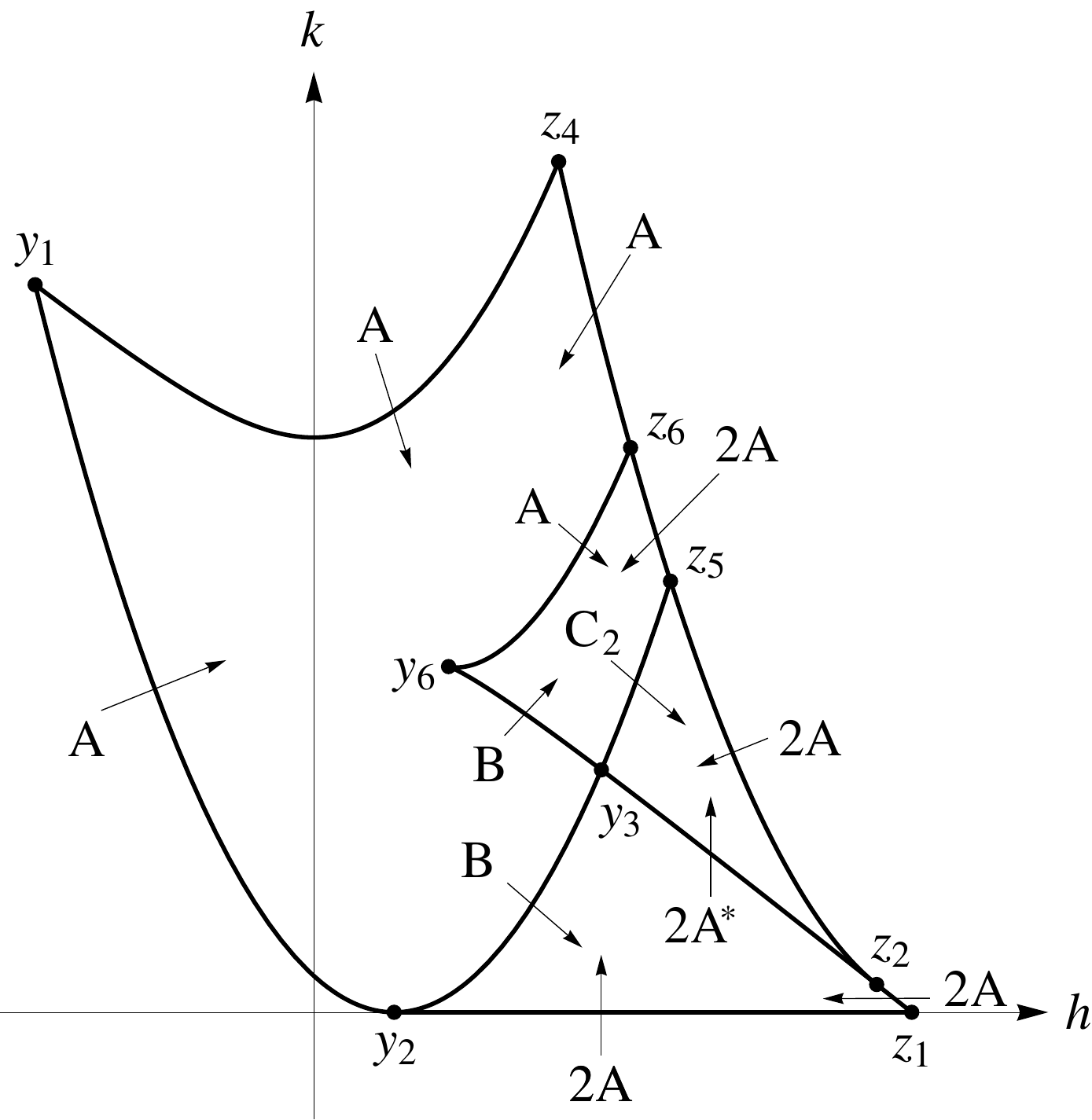}
   \caption{Area \textrm{V}: $f_m(b)<a$}
   \label{Fig:Koval_U_5_A}
\endminipage
\hspace{0.04\textwidth}
\minipage{0.48\textwidth}
    \includegraphics[width=\linewidth]{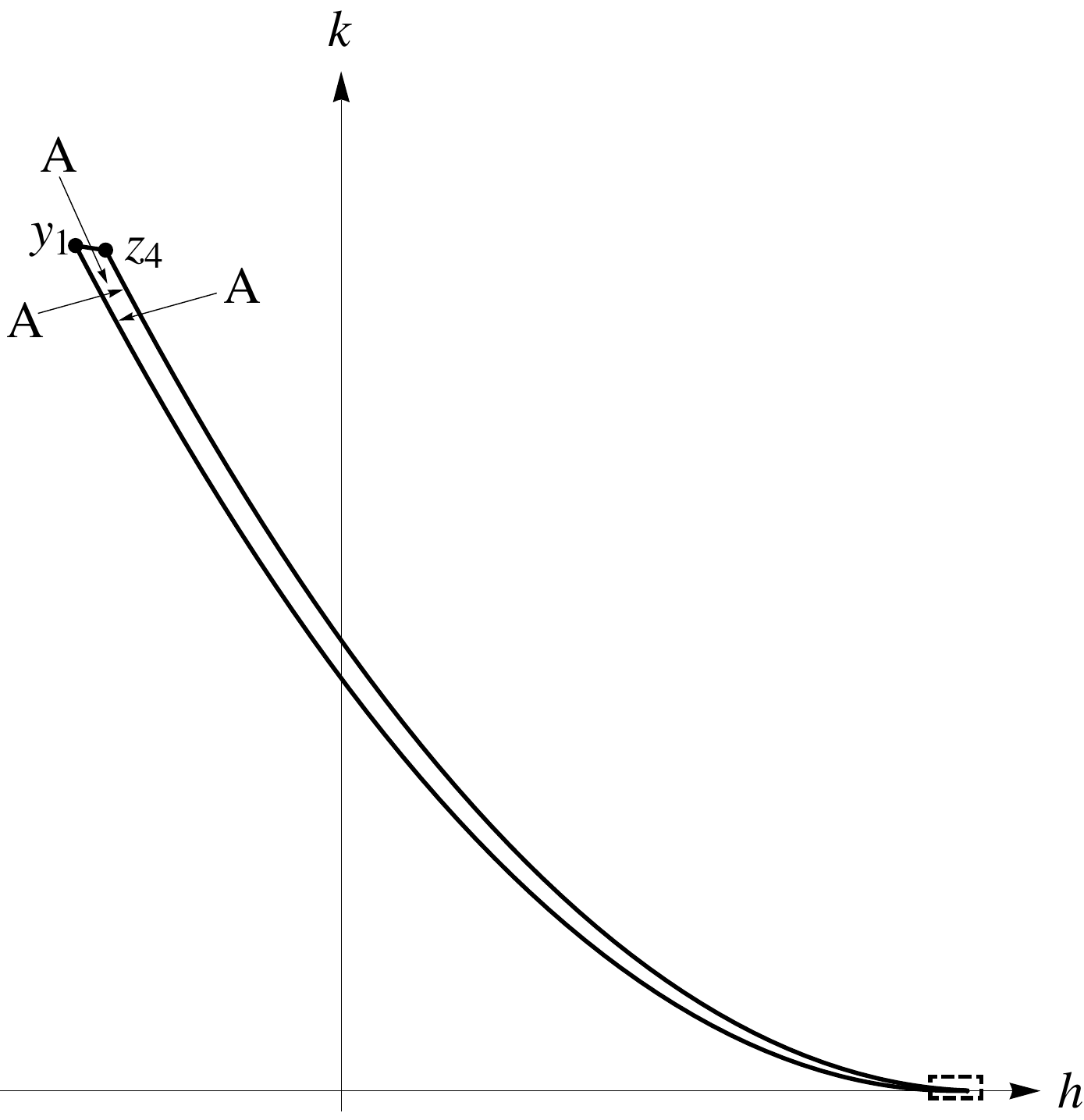}
   \caption{Area \textrm{VI}: $0<b^2 < \varkappa^3 c_1^4, \quad  f_t(b)<a<f_r(b)$}
   \label{Fig:Koval_M_1_A}
\endminipage
\end{figure}

\begin{figure}[!htb]
\minipage{0.48\textwidth}
    \includegraphics[width=\linewidth]{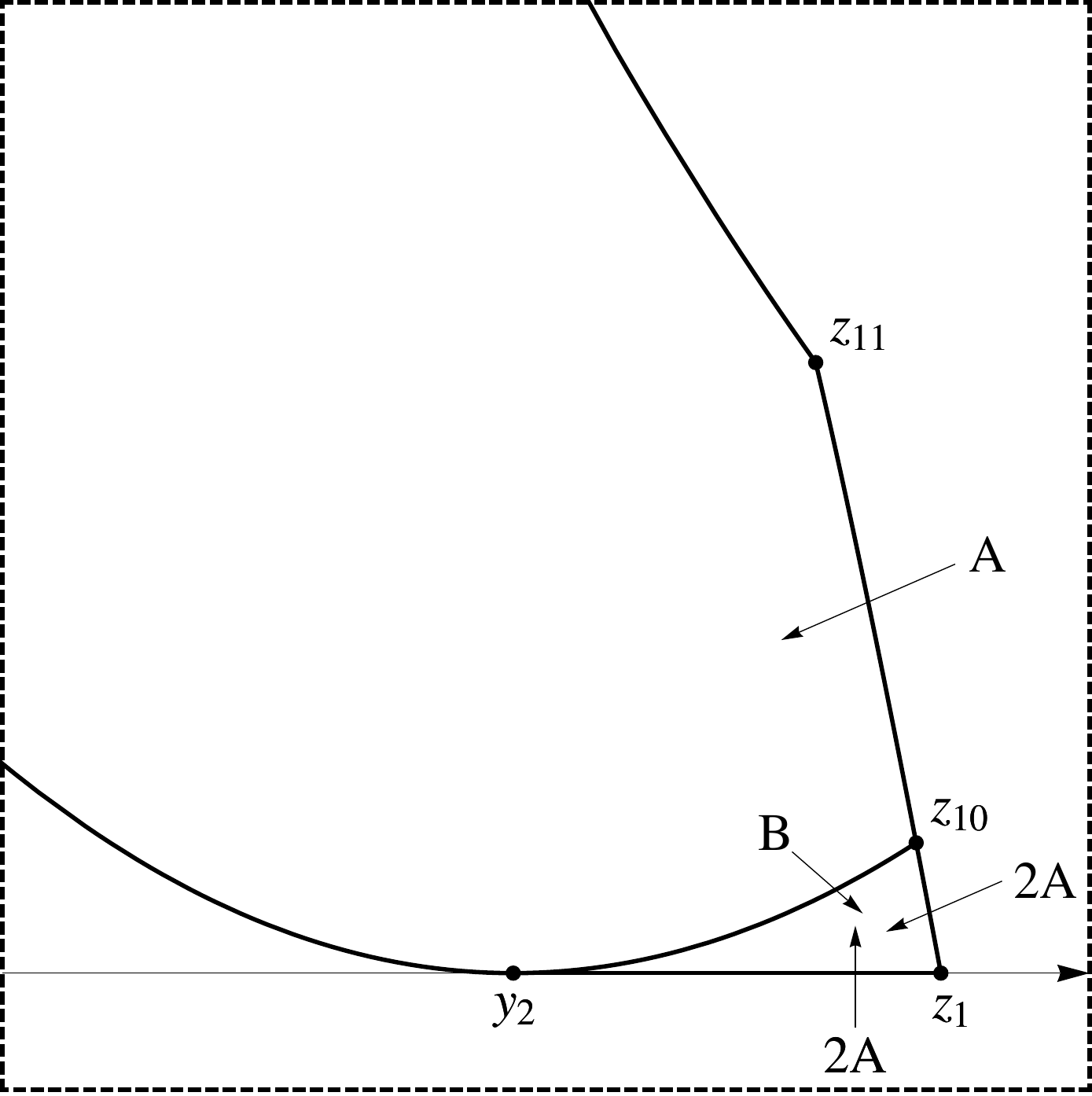}
   \caption{Area \textrm{VI}: an enlarged fragment of Fig.\ref{Fig:Koval_M_1_A}}
   \label{Fig:Koval_M_1_B}
\endminipage
\hspace{0.04\textwidth}
\minipage{0.48\textwidth}
    \includegraphics[width=\linewidth]{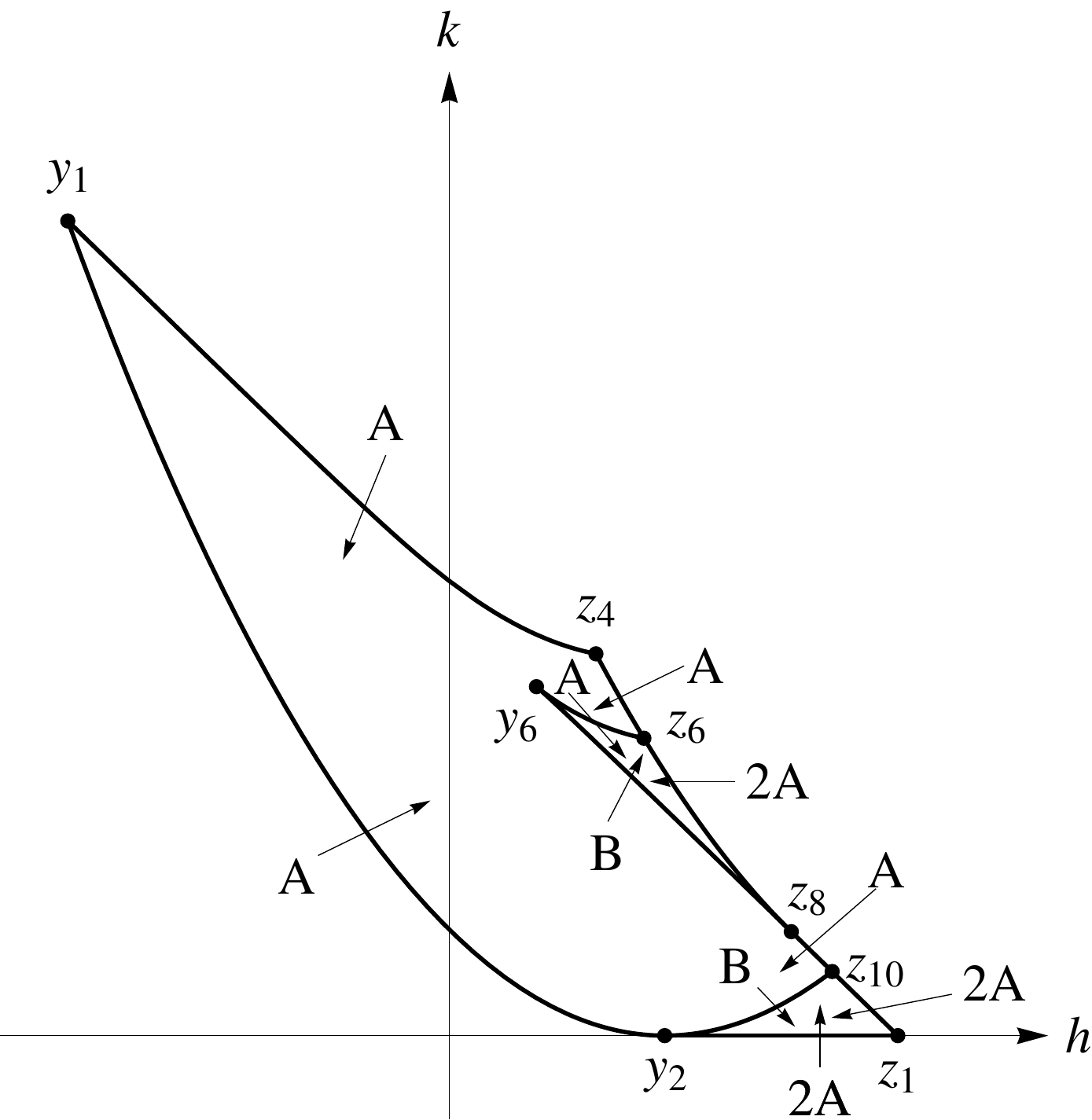}
   \caption{Area \textrm{VII}: $0<b^2 < \varkappa^3 c_1^4, \quad  \max(f_t(b),  f_r(b)) < a < f_m(b)$ }
      \label{Fig:Koval_M_2}
\endminipage
\end{figure}

\begin{figure}[!htb]
\minipage{0.48\textwidth}
    \includegraphics[width=\linewidth]{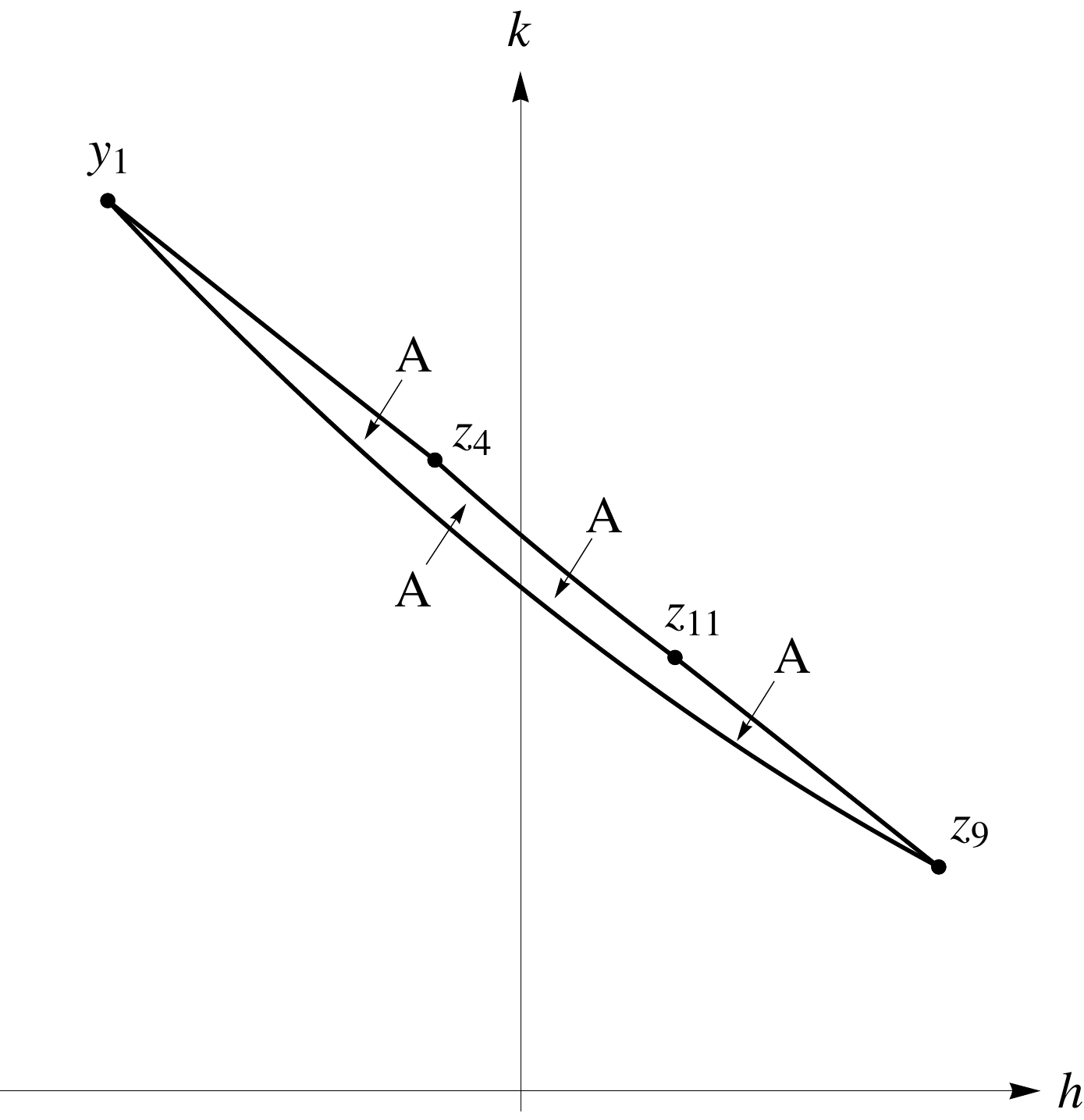}
   \caption{Area \textrm{VIII}: $0<b^2 < \varkappa^3 c_1^4, \quad  f_l(b)<a< \min(f_r(b), f_t(b))$}
      \label{Fig:Koval_D_1}
\endminipage
\hspace{0.04\textwidth}
\minipage{0.48\textwidth}
    \includegraphics[width=\linewidth]{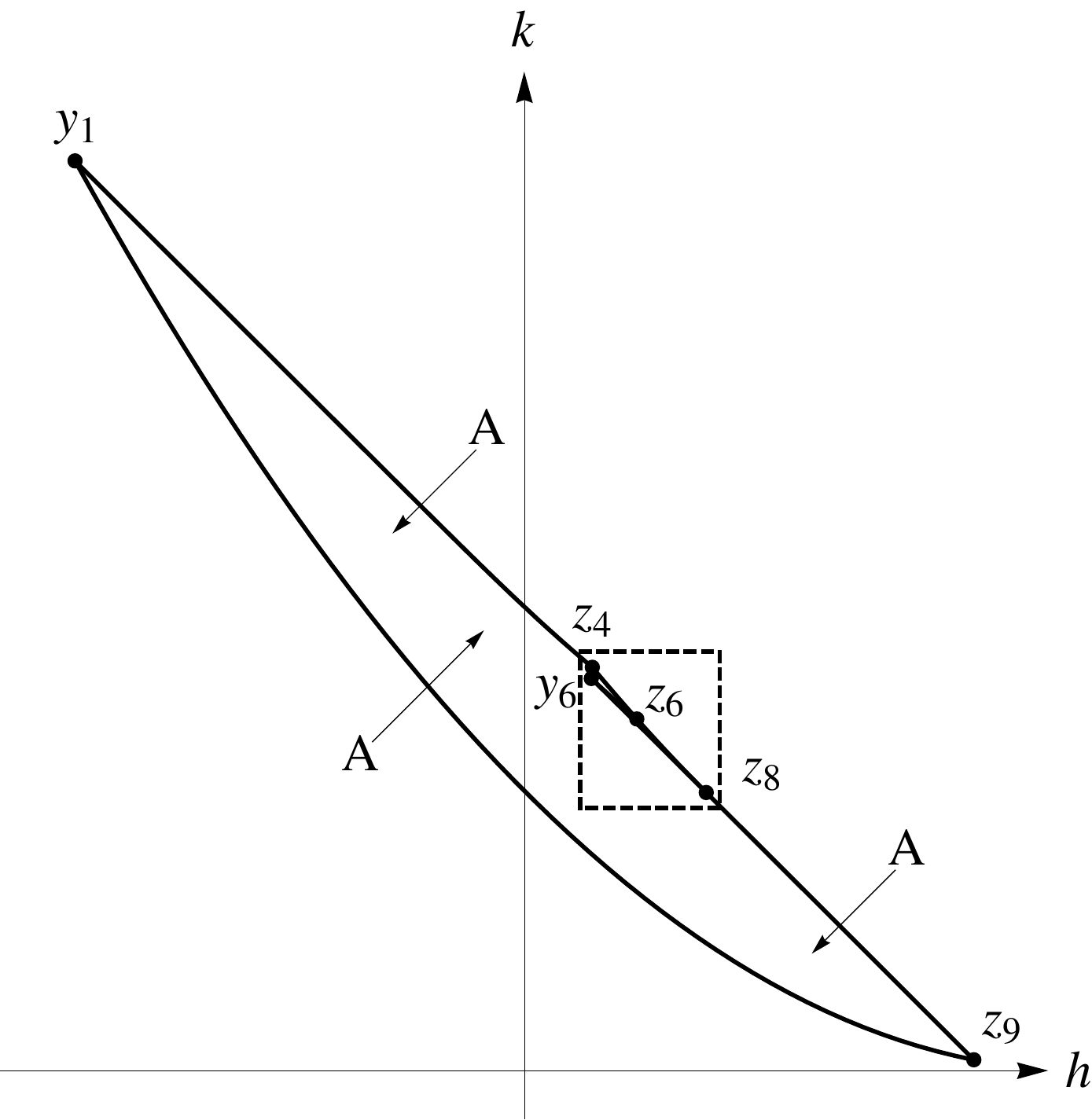}
   \caption{Area \textrm{IX}: $0<b^2 < \varkappa^3 c_1^4, \quad  f_r(b)<a<f_t(b)$}
      \label{Fig:Koval_D_2_A}
\endminipage
\end{figure}

\begin{figure}[h!]
    \centering
    \includegraphics[width=0.48\textwidth]{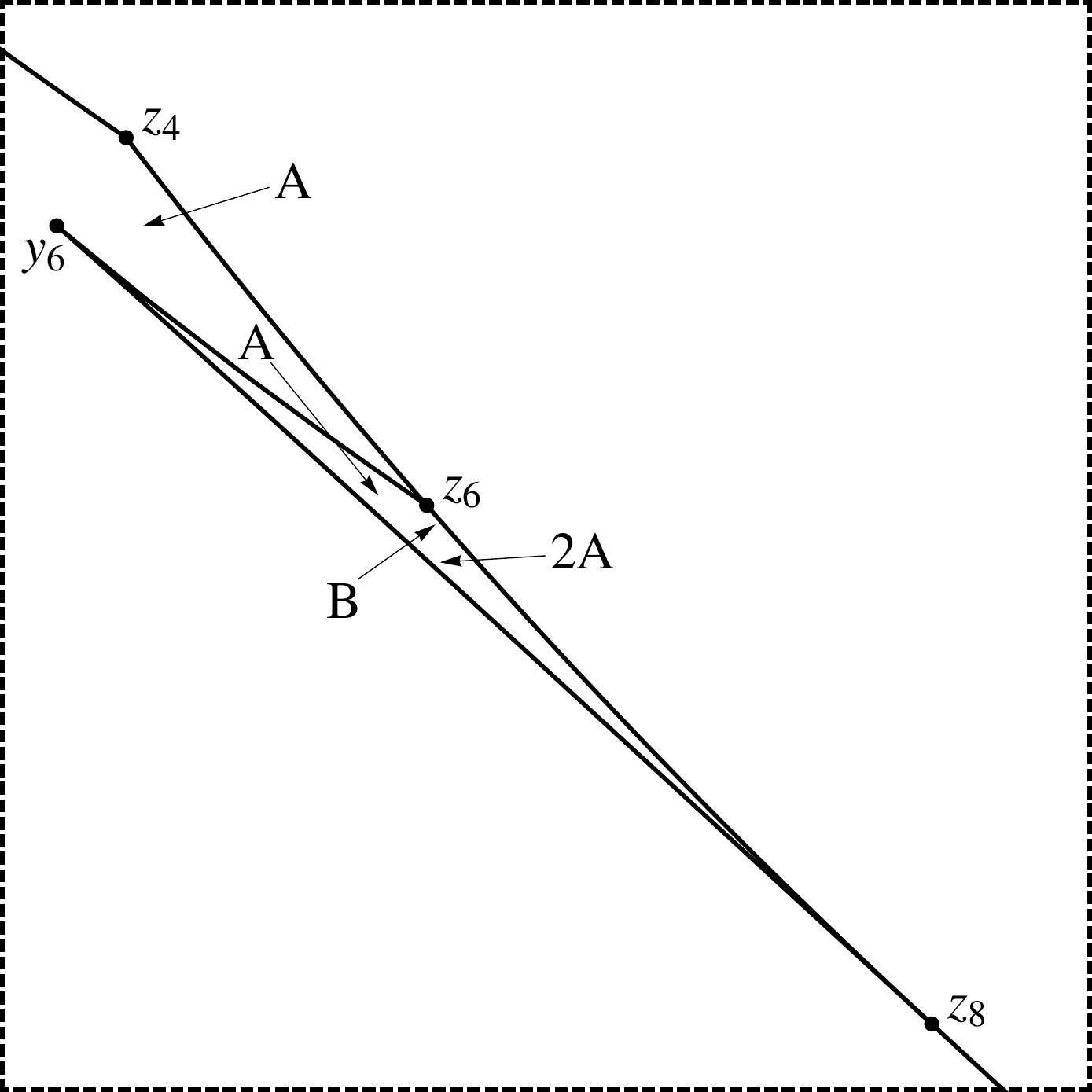}
   \caption{Area \textrm{IX}: an enlarged fragment of Fig.
   \ref{Fig:Koval_D_2_A}}
      \label{Fig:Koval_D_2_B}
\end{figure}

\end{theorem}

Recall that the bifurcation diagram for an orbit $M_{a, b}$ with
$b<0$ coincides with the bifurcation diagram for the orbit $M_{a,
-b}$ (see Remark~\ref{R:minusB}).

\begin{remark}
In Fig. \ref{Fig:Koval_U_1_A} -- \ref{Fig:Koval_D_2_B} the arcs $y_8
z_2, z_2 z_1, y_5 z_3, z_8 z_9, z_8 z_{11}$ and $z_9 z_{11}$ belong
to the parametric curve \eqref{Eq:Parametric Curve}. The rest of the
arcs of the bifurcation diagrams distribute between the curves in an
obvious way.
\end{remark}

\begin{remark}

In this paper $9$ areas of the plane $\mathbb{R}^2 (a, b)$ are
numbered with Roman numerals \textrm{I}--\textrm{IX} as it is shown
in Fig. \ref{Fig:Areas_Big} and \ref{Fig:Areas_Small} (the relative
positioning of the graphs of functions $f_k, f_r, f_m, f_t$ and
$f_l$ is also described in Assertion
\ref{A:Functions_Interposition}). We continue the numbering on the
areas of the line $b=0$ as follows:

\begin{itemize}

\item area \textrm{X}: $\{ b=0, \quad  \varkappa^2 c_1^2 < a \}$;

\item area \textrm{XI}: $\{ b=0, \quad  \frac{\varkappa^2 c_1^2}{4} < a < \varkappa^2
c_1^2 \}$;

\item area \textrm{XII}: $\{ b=0, \quad  0<a<  \frac{\varkappa^2 c_1^2}{4} \}$;

\end{itemize}

\end{remark}

A detailed description of the relative positioning of the curves
that contain the bifurcation diagrams of the momentum mapping are
given in Section \ref{SubS:Bif_Diag_B_Not_Zero}. The proof of
Theorem \ref{T:Bif_Diag} is given in Section
\ref{SubS:Proof_Main_Theorems}. In fact, in order to prove Theorem
\ref{T:Bif_Diag} it suffices to know the types of critical points of
rank $0$, which are described in the following lemma (for the
definition of nondegenerate critical points of rank $0$ and their
types see, for example, \cite{BolsinovFomenko99}).

\begin{lemma} \label{L:RankZeroType}
Let $\varkappa>0$ and $b>0$. Then the image of critical points of
rank $0$ is contained in the union of the following three families
of points.
\begin{enumerate}

\item \label{I:CritRank0Image_2Parab} The point of intersection of the parabolas \eqref{Eq:Left Parabola} and
\eqref{Eq:Right Parabola} (the point $z_5$ in
Fig.~\ref{Fig:Koval_U_5_A}). It has coordinates \[ h= \varkappa
c_1^2 + \frac{a}{\varkappa}, \quad k= \frac{a^2 - 4 \varkappa
b^2}{\varkappa^2}.\] If $a> f_m(b)$, where the function $f_m(b)$ is
given by the formula \eqref{Eq:Function_F_m}, then there are two
critical points of rank $0$ in the preimage of the point on the
orbit $M_{a, b}$. If $a=f_m(b)$, then there is one critical point of
rank $0$ in the preimage and if $a<f_m(b)$, then there is no
critical points of rank $0$ in the preimage. If $a>f_m(b)$, then all
critical points from this series are nondegenerate critical points
of saddle-center type.

\item \label{I:CritRank0Image_ParamParab}
The point of intersection of the parabolas \eqref{Eq:Left Parabola},
\eqref{Eq:Right Parabola} and the curve $\eqref{Eq:Parametric
Curve}$. The corresponding values of the parameter $z$ of the curve
$\eqref{Eq:Parametric Curve}$ are given by the equation \[z^2 =
\frac{a \pm \sqrt{a^2-4 \varkappa b^2 }}{2}c_1^2.\]

For each non-singular orbit $M_{a, b}$ there is exactly one critical
point of rank $0$ in the preimage of each intersection point of the
curve \eqref{Eq:Parametric Curve} with the parabolas. Both critical
points such that $z<0$ (that is the points $y_1$ and $z_4$ in Fig.
\ref{Fig:Koval_U_1_A}--\ref{Fig:Koval_Z_3_B}) have center-center
type.

The types of the remaining points are given in Table
\ref{Tab:Types_ParamParab}. Here $z_{+l}$ and $z_{+r}$ are the
remaining points of intersection of the curve \eqref{Eq:Parametric
Curve} with the left parabola \eqref{Eq:Left Parabola} and the right
parabola \eqref{Eq:Right Parabola} respectively and the functions
$f_r(b)$, $f_m(b)$ and $f_t(b)$ are given by the formulas
\eqref{Eq:Function_F_r}, \eqref{Eq:Function_F_m} and
\eqref{Eq:Triple_Intersect} respectively. (In Fig.
\ref{Fig:Koval_U_1_A}--\ref{Fig:Koval_D_2_B} the point $z_{+r}$ is
denoted by $z_3$, $z_6$ or $z_{11}$ depending on the type of the
point. The point $z_{+l}$ is denoted by $y_3$, $y_7$, $y_{12}$,
$z_9$ or $z_{10}$.)

\begin{table}[h]
\centering
\begin{tabular}{| l | c || c | c |}
\hline
& & $0< b^2 < \varkappa^3 c_1^4$ & $b^2> \varkappa^3 c_1^4$ \\
\hline
$a>f_m(b)$ & $z_{+r}$ & center-center & center-center \\
& $z_{+l}$ & saddle-saddle & saddle-saddle \\
\hline
$f_m(b)>a>f_t(b)$& $z_{+r}$ & center-center & center-saddle \\
$a \ne f_r(b)$ & $z_{+l}$ & center-saddle & saddle-saddle \\
\hline
$f_t(b)>a>f_l(b)$& $z_{+r}$ & center-center & center-saddle \\
$a \ne f_r(b)$ & $z_{+l}$ & center-center & center-saddle \\
\hline
\end{tabular}\caption{Types of intersections of the curve \eqref{Eq:Parametric Curve} and the parabolas \eqref{Eq:Left Parabola} and \eqref{Eq:Right Parabola}.}
  \label{Tab:Types_ParamParab}
\end{table}

\item \label{I:CritRank0Image_ParamK0} The points of intersection of the curve \eqref{Eq:Parametric Curve} and the line
$k=0$ (the points $y_{10}$, $y_{11}$ and $z_1$ in Fig.
\ref{Fig:Koval_U_1_A}--\ref{Fig:Koval_D_2_B}). In the preimage of
each point of intersection lying to the right of the point
\[h= \varkappa c_1^2 + \frac{a}{\varkappa} - \frac{\sqrt{a^2 - 4 \varkappa b^2}}{\varkappa}
\] (that is, to the right of the point of intersection of the parabola \eqref{Eq:Right Parabola}
and the line $k=0$), there are exactly two critical points of rank
$0$ and the types of these two points coincide. The preimages of all
other points of intersection are empty. In other words,

\begin{itemize}

\item if $0< b^2< \varkappa^3 c_1^4$ and $a<f_t(b)$, where the function $f_t(b)$
is given by the formula \eqref{Eq:Triple_Intersect}, then there is
no critical points of rank $0$ from this series in the preimage;

\item if $a> f_t(b)$ and $a> f_k(b)$, where $f_k(b)$ is given by the formula
\eqref{Eq: Function_F_K}, the there are $2$ point from this series
on the orbit;

\item if $f_k(b)> a>f_t(b)$ (it is possible only if $b^2> \varkappa^3 c_1^4$), then the preimage contains $6$ points from this series.

\item if $b^2> \varkappa^3 c_1^4$ and $a< f_t(b)$, then the preimage contains $4$ points from this series.

\end{itemize}

The points corresponding to the points of intersection with the
parameter $z>z_{\textrm{cusp}} = \sqrt[3]{b^{2} c_1^{2}}$, that is
with the parameter greater than the parameter of the cusp of the
curve \eqref{Eq:Parametric Curve} (that is the points   $y_{10}$ and
$z_1$ in Fig. \ref{Fig:Koval_U_1_A}--\ref{Fig:Koval_M_2}), have
center-center type. If  $a \ne f_t(b)$, then the points
corresponding to the points of intersection with the parameter
$z<z_{\textrm{cusp}} = \sqrt[3]{b^{2} c_1^{2}}$  (that is the point
$y_{11}$ in Fig. \ref{Fig:Koval_U_2_A}--\ref{Fig:Koval_U_2_C}) have
center-saddle type.

\end{enumerate}

\end{lemma}

The proof of Lemma \ref{L:RankZeroType} is given in Section
\ref{SubS:Critical_Points_Zero_Rang}.

As it turned out, in the case under consideration the obtained
information about the types of critical points allows us not only to
construct the bifurcation diagrams of the momentum mapping but also
to determine the types of bifurcations of Liouville tori and the
loop molecules for singular points of the momentum mapping. Just as
the proof of Theorem \ref{T:Bif_Diag}, the proofs of Theorems
\ref{T:Bifurcations} and \ref{T:Molecules} are given in Section
\ref{SubS:Proof_Main_Theorems}.

Just as in the classical Kovalevskaya case there are only four types
of bifurcation of tori corresponding to the smooth regular arcs of
the bifurcation diagram. In the terminology from
\cite{BolsinovFomenko99} they correspond to the atoms $A$, $A^*$,
$B$ and $C_2$.

\begin{theorem} \label{T:Bifurcations} In Fig.
\ref{Fig:Koval_U_1_A}--\ref{Fig:Class_Koval_1} for each bifurcation
diagram of the momentum mapping the bifurcations of Liouville tori
corresponding to different arcs of the bifurcation diagrams are
specified and all singular points of the bifurcation diagrams are
marked. \end{theorem}

The singular points of the bifurcation diagrams are denoted by
$y_1$--$y_{13}$ and $z_1$--$z_{11}$. These points correspond to the
cusps, the points of intersection and tangency of the bifurcation
curves that form the bifurcation diagram of the momentum mapping
(recall that these curves are described in Lemma
\ref{L:Curve_Images_B_not_0_Kappa_not_0}).

The singular points $y_1, y_3, y_7, y_{10}$--$y_{12}$, $z_1,
z_3$--$z_7$ and $z_9$--$z_{11}$ correspond to nondegenerate
singularities of the momentum mapping $\mathcal{F}= (H, K): M^4_{a,
b} \to \mathbb{R}^2$. The points marked with the same letters
correspond to singularities of the same type. The types of these
nondegenerate singularities are given in Lemmas \ref{L:RankZeroType}
and \ref{L:RankZeroType_BZero}.

The points $y_2, y_4, y_5, y_6, y_8, y_9, y_{13}, z_2$ and $z_8$
correspond to degenerate one-dimensional orbits of the action of the
group $\mathbb{R}^2$ on $M_{a, b}$ generated by the Hamiltonian
\eqref{Eq:Hamiltonian} and the integral \eqref{Eq:First_Integral}.
(In this paper we do not determine the types of critical points of
rank $1$, but it is possible to make an assumption about their types
--- most likely they are typical singularities described in
\cite{BolsinovFomenko99}. The type of a singularity can be easily
guessed from its loop molecule.) The next statement ensures that the
singularities in the preimage of all other points of the bifurcation
diagrams are nondegenerate.

\begin{assertion} \label{A:NonDegen_Crit_Rank_1}
The considered integrable Hamiltonian systems with Hamiltonian
\eqref{Eq:Hamiltonian} and \eqref{Eq:First_Integral} on all
non-singular orbits $M_{a,b}$ of the Lie algebras $so(4)$, $e(3)$
and $so(3, 1)$ are of Bott type. In other words, for all
non-singular values of the parameters $a$ and $b$ all critical
points in the preimage of non-singular points of the bifurcation
diagrams (that is, the points that are not cusps or points of
intersection or tangency for the smooth arcs of the bifurcation
diagrams) are nondegenerate points of rank $1$.
\end{assertion}

The proof of Assertion \ref{A:NonDegen_Crit_Rank_1} is given in
Section \ref{SubS:CritRank1}. Let us emphasize that in this paper
the fact that the systems are of Bott type is proved for all
non-singular orbits of the pencil $so(4)-e(3)-so(3, 1)$.

\begin{remark}
It is not hard to understand the structure of the bifurcation
diagrams for the remaining non-singular orbits that are not shown in
Fig. \ref{Fig:Koval_U_1_A}--\ref{Fig:Class_Koval_1}. For example in
the case $a=f_k(b), \varkappa>0$ the parametric curve
\eqref{Eq:Parametric Curve} intersects the line $k=0$ at the cusp
point. There is no doubt that the critical points of rank $0$ at
which the structure of bifurcation diagrams changes are degenerate.
In section \ref{SubS:Critical_Points_Zero_Rang} during the proof of
Lemma \ref{L:Curve_Images_B_not_0_Kappa_not_0} we actually prove a
more general statement that the remaining points
--- in a neighbourhood of which the bifurcation diagram does not
change its structure when passing from one area of the plane
$\mathbb{R}^2(a, b)$ to another --- remain nondegenerate and their
types do not change.
\end{remark}

\begin{theorem} \label{T:Molecules}
The loop molecules for all singular points of the bifurcation
diagrams shown in Fig.
\ref{Fig:Koval_U_1_A}--\ref{Fig:Class_Koval_1} are listed in Tables
\ref{Tab:Cirle_Mol_Old} and \ref{Tab:Cirle_Mol_New}. The loop
molecules for singular points of the diagrams marked with the same
letters coincide.
\end{theorem}


\begin{table}[ht]
\centering
\begin{tabular}{c}
    \includegraphics[width=0.75\textwidth]{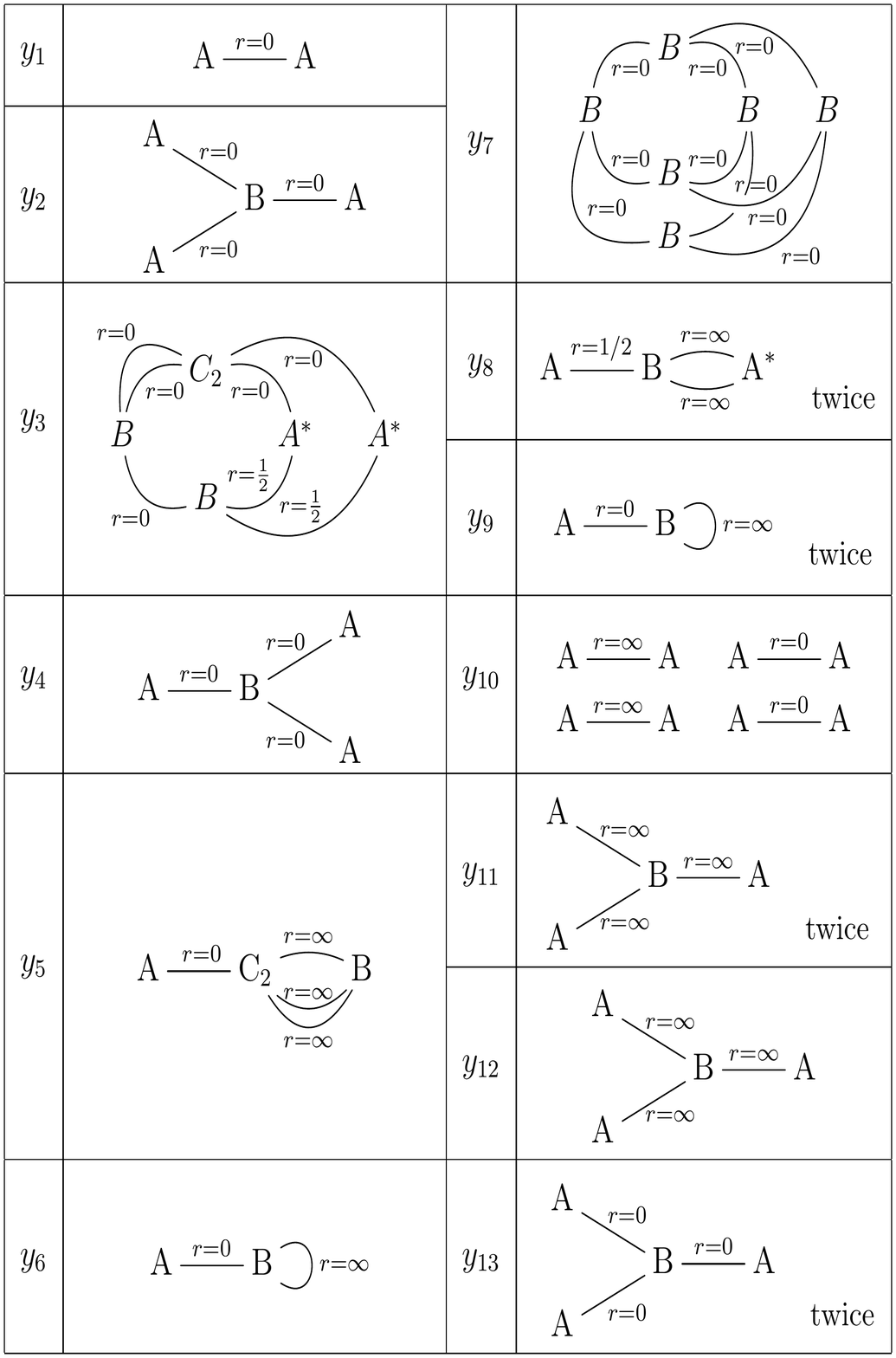}
\end{tabular}
  \caption{Loop molecules, classical Kovalevskaya case.}
\label{Tab:Cirle_Mol_Old}
\end{table}

\begin{table}[ht]
\centering
\begin{tabular}{c}
    \includegraphics[width=0.75\textwidth, height=0.5\textheight]{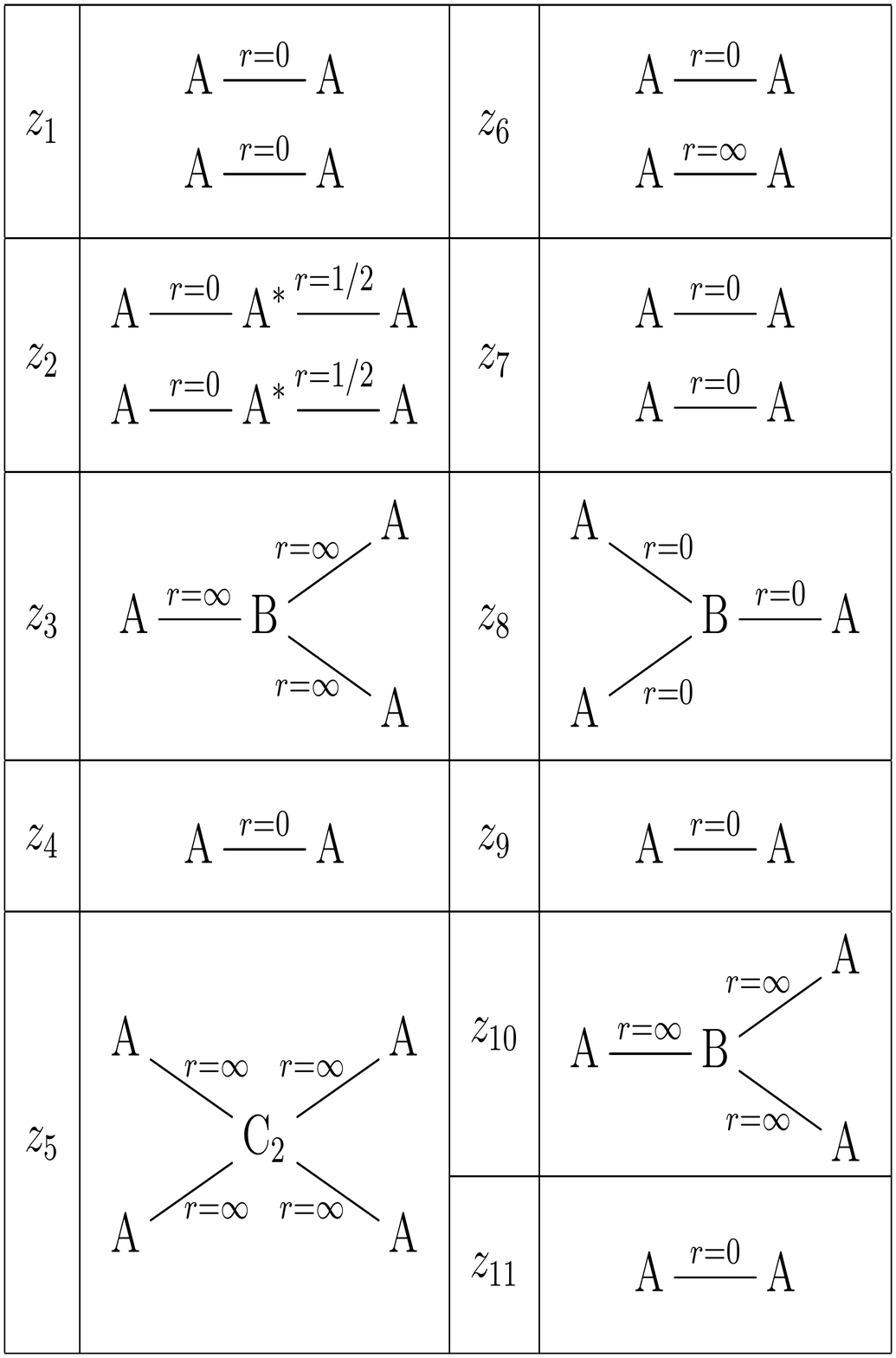}
\end{tabular}
  \caption{New loop molecules.}
\label{Tab:Cirle_Mol_New}
\end{table}

\begin{remark}
For the points on the boundary of bifurcation diagrams the loop
molecules in Tables \ref{Tab:Cirle_Mol_Old} and
\ref{Tab:Cirle_Mol_New} are shown counterclockwise.  Although in
this case the ambiguity may arise only for the point $z_2$ . The
loop molecule for this point must consists of two identical
molecules both having the same form as for the degenerate
singularity called elliptic period-doubling bifurcation. (For more
about degenerate singularities, see, for example,
\cite{BolsinovFomenko99}.)\end{remark}

Let us emphasize that in Theorems \ref{T:Bifurcations} and
\ref{T:Molecules} we consider not only the case $\varkappa>0, b\ne
0$ but also the cases $\varkappa > 0, b =0$ and $\varkappa =0$. Note
that the for $\varkappa =0$ the obtained results completely coincide
with the known results for the classical Kovalevskaya case (see, for
example, \cite{BolsinovFomenko99}).

\begin{remark}
There is an inaccuracy in the book \cite{BolsinovFomenko99} in the
list of loop molecules for the Kovalevskaya integrable case: the
molecules for the points $y_8$ and $y_9$  should be repeated twice.
\end{remark}

\subsection{Case $\varkappa>0, b=0$} \label{SubS:Bif_Diag_B_Zero}

We now describe the results in the case when the second integral
$b=0$. The following lemma is proved in Section
\ref{SubS:CritRank1}, as well as Lemma
\ref{L:Curve_Images_B_not_0_Kappa_not_0}.

\begin{lemma} \label{L:Curve_Images_B_Zero_Kappa_not_0}
Let $\varkappa \ne 0$ and $b=0$. Then for any non-singular orbit
$M_{a, 0}$ (that is, for orbits such that $a \ne 0$) the bifurcation
diagram $\Sigma_{h, k}$ for the integrable Hamiltonian system with
Hamiltonian \eqref{Eq:Hamiltonian} and integral
\eqref{Eq:First_Integral} is contained in the union of the following
three families of curves on the plane $\mathbb{R}^2(h,k)$:
\begin{enumerate}
\item The line $k=0$; \item The union of the parabola
\begin{equation}\label{Eq:Up Parabola B Zero} k=(h-\varkappa c_1^2)^2 +
4ac_1^2\end{equation} and the tangent line to this parabola at the
point $h=0$
\begin{equation}\label{Eq:Tangent Line B Zero} k= - 2 \varkappa c_1^2 h +
(4ac_1^2 + \varkappa^2 c_1^4);\end{equation} \item The union of two
parabolas \begin{equation} \label{Eq:Left Parabola B Zero} k
=\left(h-\varkappa c_1^2 \right)^2
\end{equation} and
\begin{equation}\label{Eq:Right Parabola B Zero}  k=\left(h-\varkappa c_1^2
-\frac{2a}{\varkappa } \right)^2.
\end{equation} \end{enumerate}
\end{lemma}

Now in order to construct the bifurcation diagrams of the momentum
mapping it remains to throw away several parts of the curves
described in Lemma \ref{L:Curve_Images_B_Zero_Kappa_not_0}. A
precise description of the bifurcation diagrams is given in Theorem
\ref{T:Bif_Diag_B_Zero} (see also Fig.
\ref{Fig:Koval_Z_1}--\ref{Fig:Koval_Z_3_B}).

Let us now describe the set of critical points of rank $0$. The
following lemma is proved in Section
\ref{SubS:Critical_Points_Zero_Rang}, as well as Lemma
\ref{L:RankZeroType}.

\begin{lemma} \label{L:RankZeroType_BZero}
Let $\varkappa>0$ and $b=0$. Then the image of critical points of
rank $0$ is contained in the union of the following three families
of points:
\begin{enumerate} \item The point of intersection of the parabolas \eqref{Eq:Left Parabola B Zero}
and \eqref{Eq:Right Parabola B Zero} (the point $z_5$ in Fig.
\ref{Fig:Koval_Z_1}). This point has coordinates
\[ h= \varkappa c_1^2 + \frac{a}{\varkappa}, \quad k=
\frac{a^2}{\varkappa^2}.\] If $a>\varkappa^2c_1^2$, then there are
two critical points of rank $0$ in the preimage of the point on the
orbit $M_{a, 0}$. If $a=\varkappa^2c_1^2$, then there is one
critical point of rank $0$ in the preimage, and if
$a<\varkappa^2c_1^2$, then there is no critical points of rank $0$
in the preimage. If $a>\varkappa^2c_1^2$ then all critical points
from this series are nondegenerate critical points of saddle-center
type.

\item The point of intersection of the upper parabola \eqref{Eq:Up Parabola B Zero}
and the tangent line \eqref{Eq:Tangent Line B Zero} with the
parabolas \eqref{Eq:Left Parabola B Zero} and \eqref{Eq:Right
Parabola B Zero}.

For any $a>0$ there are two points of intersection of the line
\eqref{Eq:Tangent Line B Zero} and the left parabola \eqref{Eq:Left
Parabola B Zero} with coordinates
\[h= \pm 2 \sqrt{a}c_1, \quad k= (\pm 2 \sqrt{a}c_1 -
\varkappa^2c_1^2)^2 \] and there is exactly one critical point of
rank $0$ in the preimage of each of these points on the orbit $M_{a,
0}$.

\begin{enumerate}
\item The critical point in the preimage of the upper point of intersection (that is the point
$y_1$ in Fig. \ref{Fig:Koval_Z_1} - \ref{Fig:Koval_Z_3_B}) is a
nondegenerate critical point of center--center type.

\item  If $a>\varkappa^2 c_1^2$, then the critical point $p$ in the preimage of the lower point of intersection
(that is the point $y_3$ in Fig. \ref{Fig:Koval_Z_1}, $z_{10}$ in
Fig. \ref{Fig:Koval_Z_2} and $z_{9}$ in Fig. \ref{Fig:Koval_Z_3_A},
\ref{Fig:Koval_Z_3_B}) is a nondegenerate critical point of
saddle--saddle type. If $\frac{\varkappa^2 c_1^2}{4}<a<\varkappa^2
c_1^2$, then $p$ is a nondegenerate critical point of saddle--center
type and if $a< \frac{\varkappa^2 c_1^2}{4}$, then $p$ is a
nondegenerate critical point of center--center type.

\item There are two critical points of rank $0$ in the preimage of the
intersection point of the upper parabola \eqref{Eq:Up Parabola B
Zero} with the right parabola \eqref{Eq:Right Parabola B Zero} (the
point $z_7$ in Fig. \ref{Fig:Koval_Z_1}--\ref{Fig:Koval_Z_3_B})
\[h= \frac{a}{\varkappa}, \quad k=(\frac{a}{\varkappa}-\varkappa
c_1^2)^2 + 4ac_1^2 \] on each orbit $M_{a, 0}$ (where $a>0$) and
both points in the preimage are nondegenerate critical points of
center--center type.
\end{enumerate}

\item The point of intersection of the line $k=0$ and the tangent line
\eqref{Eq:Tangent Line B Zero} (the point $z_1$ in Fig.
\ref{Fig:Koval_Z_1} and \ref{Fig:Koval_Z_2}). This point has
coordinates \[h=\frac{\varkappa c_1^2}{2}+ \frac{2a}{\varkappa},
\quad k=0. \] If $a> \frac{\varkappa^2 c_1^2}{4}$, then there are
two critical points of rank $0$ of center--center type in the
preimage of this point on the orbit $M_{a, 0}$. If $a=
\frac{\varkappa^2 c_1^2}{4}$, then there is one point critical point
of rank $0$ in the preimage, and if $a< \frac{\varkappa^2
c_1^2}{4}$, then there is no critical points of rank $0$ in the
preimage.
\end{enumerate}

\end{lemma}

In the case when $\varkappa>0$ and $b=0$ there are three
qualitatively different types of bifurcation diagrams. This is due
to the fact that the functions $f_k, f_r, f_m, f_t$ and $f_l$ given
by the formulas \eqref{Eq: Function_F_K}, \eqref{Eq:Function_F_r},
\eqref{Eq:Function_F_m}, \eqref{Eq:Triple_Intersect} and
\eqref{Eq:Function_F_l} respectively divide the ray $\{b=0, a\geq
0\}$ into $3$ parts.

\begin{theorem} \label{T:Bif_Diag_B_Zero}
In the case when $\varkappa>0$ and $b=0$ the form of the bifurcation
diagram depends on the parameter $a$. In the following three cases
the diagrams are qualitatively different:

\begin{enumerate}

\item $0< a < \frac{\varkappa^2 c_1^2}{4}$,

\item $\frac{\varkappa^2 c_1^2}{4} < a < \varkappa^2 c_1^2$,

\item $\varkappa^2 c_1^2<a$.

\end{enumerate}

The corresponding bifurcation diagrams are shown in Fig.
\ref{Fig:Koval_Z_1}--\ref{Fig:Koval_Z_3_B}, where the formulas for
the lines and parabolas are given in Lemma
\ref{L:Curve_Images_B_Zero_Kappa_not_0}.
\end{theorem}

\begin{figure}[!htb]
\minipage{0.48\textwidth}
    \includegraphics[width=\linewidth]{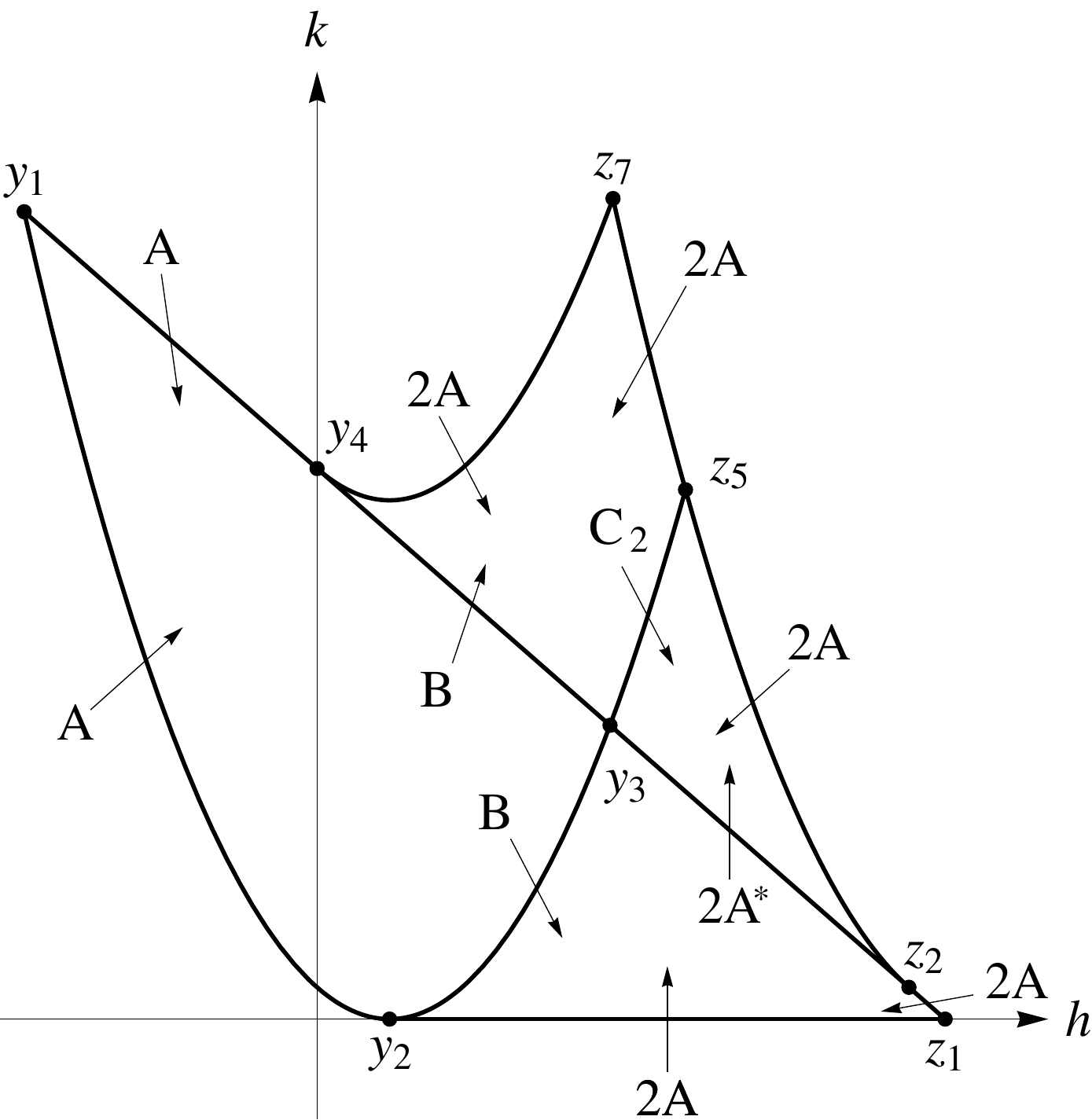}
   \caption{Area \textrm{X}:  $b=0, \quad  \varkappa^2 c_1^2 < a$}
      \label{Fig:Koval_Z_1}
\endminipage
\hspace{0.04\textwidth}
\minipage{0.48\textwidth}
    \includegraphics[width=\linewidth]{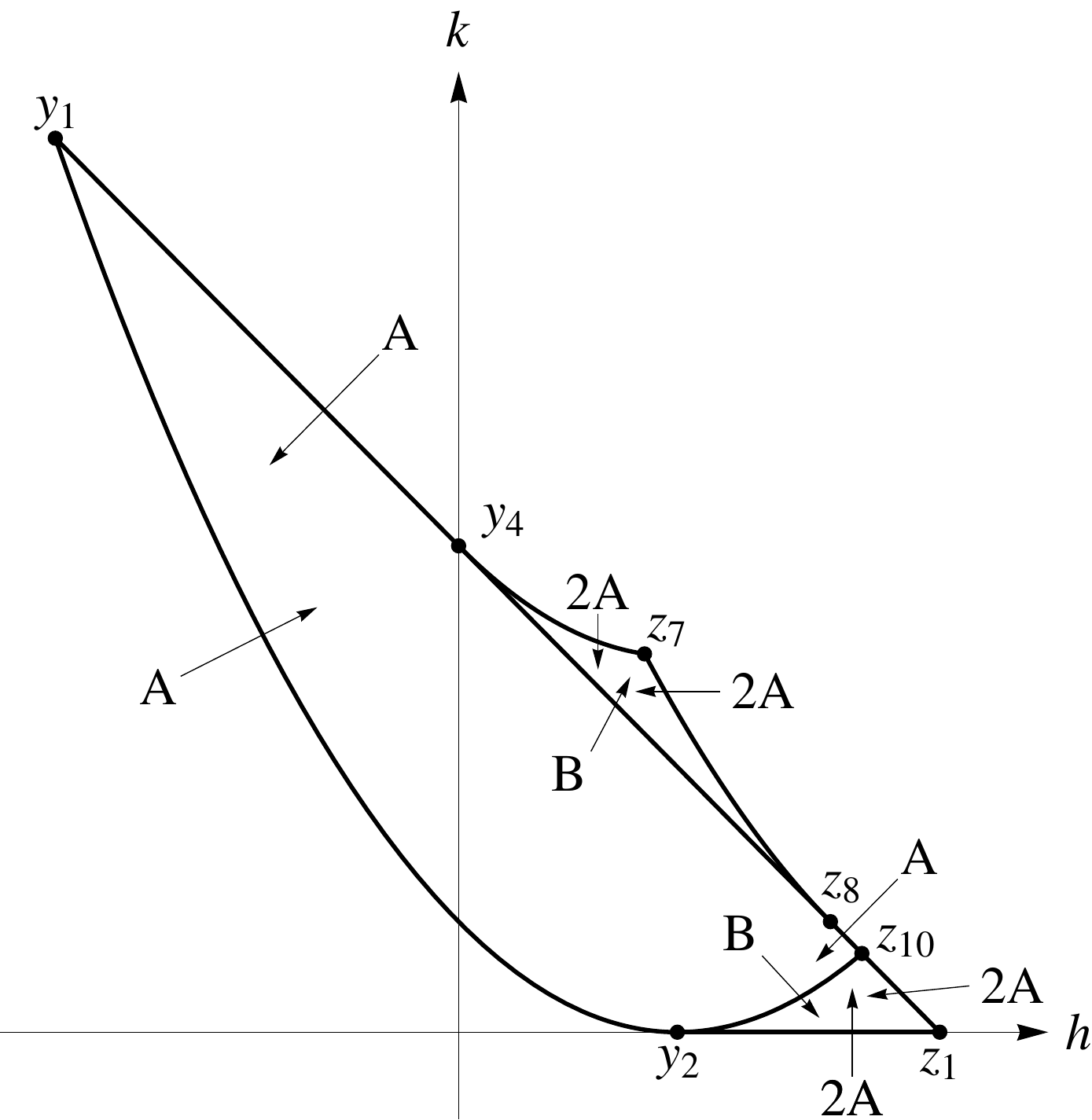}
   \caption{Area \textrm{XI}: $b=0, \quad  \frac{\varkappa^2 c_1^2}{4} < a < \varkappa^2 c_1^2$}
   \label{Fig:Koval_Z_2}
\endminipage
\end{figure}

\begin{figure}[!htb]
\minipage{0.48\textwidth}
    \includegraphics[width=\linewidth]{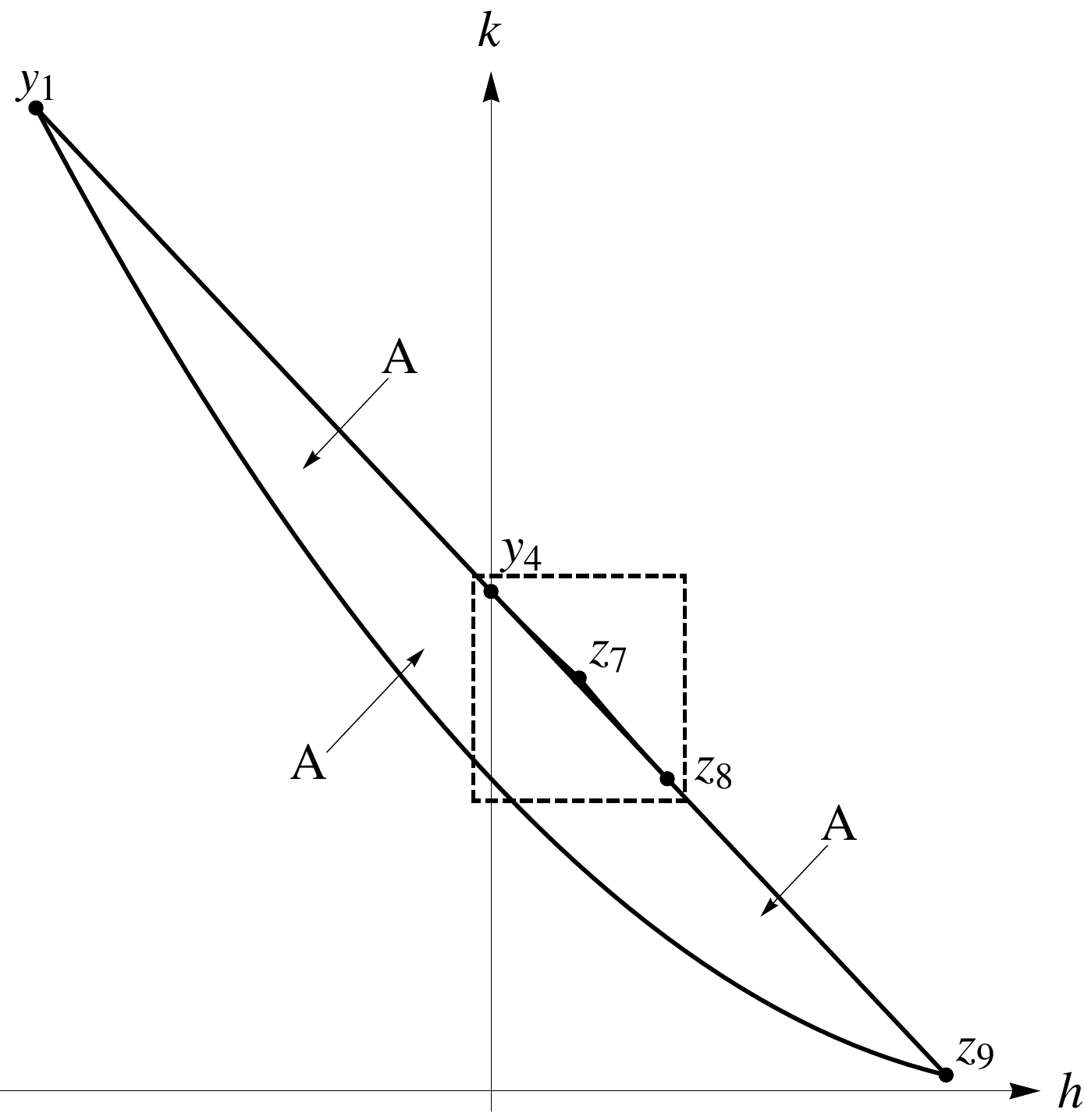}
   \caption{Area \textrm{XII}: $b=0, \quad  0<a<  \frac{\varkappa^2 c_1^2}{4}$}
   \label{Fig:Koval_Z_3_A}
\endminipage
\hspace{0.04\textwidth}
\minipage{0.48\textwidth}
    \includegraphics[width=\linewidth]{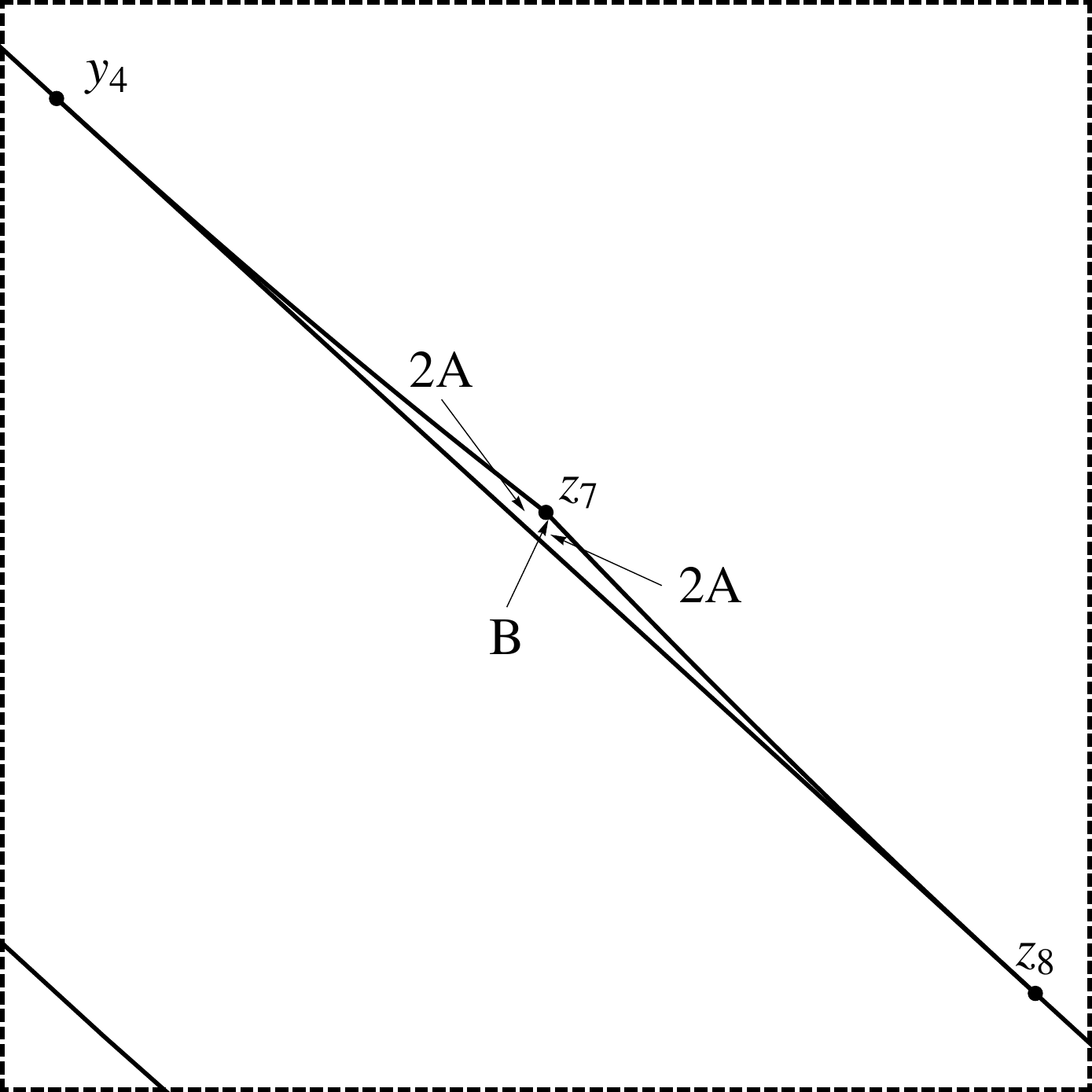}
   \caption{Area \textrm{XII}: an enlarged fragment of Fig. \ref{Fig:Koval_Z_3_A}}
      \label{Fig:Koval_Z_3_B}
\endminipage
\end{figure}

\begin{remark}
In Fig. \ref{Fig:Koval_Z_1}--\ref{Fig:Koval_Z_3_B} the arcs $z_2
z_1, z_8 z_9$ and $z_8 z_{10}$ belong to the line \eqref{Eq:Tangent
Line B Zero} and the arc  $y_4 z_7$ belongs to the upper parabola
\eqref{Eq:Up Parabola B Zero}. The rest of the arcs distribute
between the curves in an obvious way.
\end{remark}

Bifurcation diagrams for the case $b=0$ were also previously
described in the paper \cite{Kozlov12}.

The bifurcations of Liouville tori and the loop molecules for the
singular points are described earlier in Theorems
\ref{T:Bifurcations} and \ref{T:Molecules} respectively.

\section{Proof of the main statements} \label{S:Proof_Main_Statements}

\subsection{Critical points of rank $1$} \label{SubS:CritRank1}

In this section we prove Lemmas
\ref{L:Curve_Images_B_not_0_Kappa_not_0} and
\ref{L:Curve_Images_B_Zero_Kappa_not_0}, which claim that the
bifurcation diagrams are contained in the curves described in these
lemmas. For this we first describe all critical points of the
momentum mapping in Assertion \ref{A:Rank_1_Preimage} and then study
their image under the momentum mapping. Note that all these critical
points apart from the points from Assertion \ref{A:Rank_0_Preimage}
are critical points of rank $1$.

Let us emphasize that in this section we do not impose restrictions
on the parameters  $\varkappa$ and $b$ (that is $\varkappa, b \in
\mathbb{R}$, unless otherwise stated).

\begin{assertion}\label{A:Rank_1_Preimage} The set of points where the Hamiltonian vector fields corresponding
to the Hamiltonian \eqref{Eq:Hamiltonian} and the integral
\eqref{Eq:First_Integral} are linearly dependent is the union of the
following six families of points. The first three families are
three-parametric and the last three families are four-parametric.

Four-parameter families:

\begin{enumerate}

\item \label{Rank1SeriesA} $x_1 = \frac{\varkappa c_1^2  +J_1^2-J_2^2}{2 c_1}, \qquad x_2 =
\frac{J_1 J_2}{c_1}$

\item \label{Rank1SeriesB} $J_2= 0, \qquad x_3 = \frac{J_1 J_3}{c_1}$

\item \label{Rank1SeriesC} $x_1 = \varkappa c_1 + (J_1 - c_1 \frac{x_3} {J_3})\frac{x_2}{J_2}, \qquad
x_2 = J_2 \frac{(J_1 x_3 - \varkappa c_1 J_3) (J_1 J_3 - c_1 x_3) +
J_2^2 J_3 x_3}{(J_1 J_3 - c_1 x_3)^2 +J_2^2 J_3^2} ,\\$ where
$J_2J_3 \ne 0$.

\end{enumerate}

Three-parameter families:

\begin{enumerate}

\setcounter{enumi}{3}

\item \label{Rank1SeriesC1} $J_2=0, \qquad x_2=0, \qquad J_1 x_3-J_3 x_1 = 0$

\item \label{Rank1SeriesC2} $J_3 = 0, \qquad x_3 = 0, \\((x_1 - \varkappa c_1 ) J_1 + J_2 x_2) (J_2 (x_1 -\varkappa c_1) -
J_1 x_2) + c_1 x_2 (x_1 (x_1 - \varkappa c_1 ) + x_2^2) = 0$

\item \label{Rank1SeriesC3} $J_1=0, \qquad J_3=0, \qquad x_2=0$

\end{enumerate}

\end{assertion}

\begin{proof}
The points at which the Hamiltonian vector fields $X_H$ and $X_K$
are linearly dependent are exactly the points at which all 15 rank 2
minors of the matrix $\left(\begin{matrix}X_H,& X_K\end{matrix}
\right)$ composed from the coordinates of the vectors $X_H$ and
$X_K$ are equal to zero.

Note that the minor $\Delta_{13}$ corresponding to the first and the
third lines has form $$\begin{vmatrix} \{J_1, H \} & \{J_1, K \}
\\ \{J_3, H \}   & \{J_3, K \}  \end{vmatrix} = 16 c_1 \left(c_1 x_2-J_1 J_2\right) \left(J_2 J_3 \left(\varkappa c_1
- x_1 \right)+ \left(J_1 J_3-c_1 x_3\right)x_2 \right)$$

Therefore either $ x_2= \frac{J_1 J_2}{c_1}$ or $J_2 J_3
\left(\varkappa c_1 - x_1 \right)+ \left(J_1 J_3-c_1 x_3\right)x_2
=0 $.

First, we examine the case $ x_2= \frac{J_1 J_2}{c_1}$. Substituting
$x_2$ in the matrix consisting of minors we immediately obtain the
first solution \begin{equation}\label{Eq:R1Ser1A} x_1 =
\frac{\varkappa c_1^2 +J_1^2-J_2^2}{2 c_1}, \qquad x_2= \frac{J_1
J_2}{c_1}.\end{equation} The rest of the proof is by exhaustion. If
$x_3  = \frac{J_1 J_3}{c_1}$, then we obtain the following three
families of solutions:

\begin{equation}\label{Eq:R1Ser2inC} x_1 =\varkappa c_1 , \qquad x_2 =
\frac{J_1 J_2}{c_1}, \qquad x_3  = \frac{J_1 J_3}{c_1},
\end{equation}

\begin{equation}\label{Eq:R1Ser3inB} J_2 = 0, \qquad x_2 = 0, \qquad x_3 = \frac{J_1
J_3}{c_1}, \end{equation}

\begin{equation} \label{Eq:R1Ser4inC} J_1 = 0, \qquad J_3 = 0,
\qquad x_2= 0, \qquad x_3 = 0. \end{equation}

If $x_3 \ne  \frac{J_1 J_3}{c_1}$, then we get the following two
families of solutions:

\begin{equation} \label{Eq:R1Ser5inC}
J_1 x_3-J_3 x_1 = 0, \qquad J_2 = 0, \qquad x_2 = 0,
\end{equation}

\begin{equation} \label{Eq:R1Ser6inC} J_1 = 0, \qquad J_3 = 0, \qquad x_2 = 0. \end{equation}

Now suppose that $J_2 J_3 \left(\varkappa c_1 - x_1 \right)+
\left(J_1 J_3-c_1 x_3\right)x_2 =0$, $ x_2 \ne \frac{J_1 J_2}{c_1}$.
There are two variants: either $J_2 J_3 =0$ or $x_1 =
\frac{\varkappa c_1 J_2 J_3 +J_1 J_3 x_2 - c_1 x_2 x_3}{J_2 J_3}$.

If $J_2 =0$, then we get the second four-parameter family of points
of rank one:
\begin{equation} \label{Eq:R1Ser7B}
J_2= 0, \qquad x_3 = \frac{J_1 J_3}{c_1}.
\end{equation}

And if $J_2 \ne 0$, $J_3 =0$, then we obtain the following solution:
\begin{equation} \label{Eq:R1Ser8inC}
\begin{aligned}
J_3 = 0, \qquad x_3 = 0, \qquad \\ ((x_1 - \varkappa c_1 ) J_1 + J_2
x_2) (J_2 (x_1 -\varkappa c_1) - J_1 x_2) + c_1 x_2 (x_1 (x_1 -
\varkappa c_1 ) + x_2^2) = 0.
\end{aligned}
\end{equation}

Not let us consider the case $x_1 = \frac{\varkappa c_1 J_2 J_3 +J_1
J_3 x_2 - c_1 x_2 x_3}{J_2 J_3}$. It is easy to check that in this
case the minor $\Delta_{12}$ is equal to
\[-16 c_1 \left(x_2 (J_2^2 J_3^2+(J_1 J_3 - c_1 x_3)^2) - J_2
((\varkappa c_1^2  + J_1^2 + J_2^2 ) J_3 x_3 -c_1 J_1 x_3^2
-\varkappa c_1 J_1 J_3^2)\right).\]

The coefficient by $x_2$ is not equal to zero since the case $J_2
J_3 =0$ has already been analyzed. Expressing $x_2$ from this
equation we obtain the ninth solution:

\begin{equation} \label{Eq:R1Ser9C}
\begin{aligned}
x_1 = \frac{\varkappa c_1 J_2 J_3 +J_1 J_3 x_2 - c_1 x_2 x_3}{J_2
J_3},
\\ x_2 = J_2 \frac{(\varkappa c_1^2  + J_1^2 + J_2^2 ) J_3 x_3 -c_1 J_1 x_3^2 -\varkappa c_1
J_1 J_3^2}{(J_1 J_3 - c_1 x_3)^2 + J_2^2 J_3^2}.
\end{aligned}
\end{equation}

Thus we have considered all the cases. It remains to collect all the
solutions together. It is obvious that the family
\eqref{Eq:R1Ser4inC} is a particular case of the solution
\eqref{Eq:R1Ser6inC} and that the solution \eqref{Eq:R1Ser3inB} is a
special case of the solution \eqref{Eq:R1Ser7B}. It remains to note
that the family \eqref{Eq:R1Ser2inC} is contained in the families
\eqref{Eq:R1Ser7B}, \eqref{Eq:R1Ser8inC} and \eqref{Eq:R1Ser9C}.
Assertion \ref{A:Rank_1_Preimage} is proved.
\end{proof}

Now let us prove Lemma \ref{L:Curve_Images_B_not_0_Kappa_not_0}. For
this, we show that the images of the critical points from Assertion
\ref{A:Rank_1_Preimage} belong to the curves described in Lemma
\ref{L:Curve_Images_B_not_0_Kappa_not_0}.

\begin{assertion} \label{A:Rank1_Preimages_Images}
Let $\varkappa\ne 0$ and $b \ne 0$. Then the images of the families
of critical points described in Assertion \ref{A:Rank_1_Preimage}
are arranged as follows:
\begin{enumerate}

\item The images of critical points from the family \ref{Rank1SeriesA} lie on the line $k=0$.

\item The images of critical points from the family \ref{Rank1SeriesB} belong to the curve \eqref{Eq:Parametric Curve}.

\item The images of critical points from the families \ref{Rank1SeriesC},
\ref{Rank1SeriesC1}, \ref{Rank1SeriesC2} and \ref{Rank1SeriesC3} lie
on the union of two parabolas \eqref{Eq:Left Parabola} and
\eqref{Eq:Right Parabola}.

\end{enumerate}

\end{assertion}

\begin{proof}
\begin{enumerate}

\item It is explicitly checked that $k=0$.

\item The equation \eqref{Eq:Parametric Curve} can be obtained as follows.
Take the function $J_3^2+c_1 x_1$ as the parameter $z$. Note that
$J_3^2+c_1 x_1\ne 0$ since $b\ne 0$. Therefore, $J_1$  can be
expressed from the formula for $b$ and then $x_2$ can be expressed
from the formula for $a$. It remains to substitute the obtained
expressions for $J_1$ and $x_2$ in the equations for the Hamiltonian
\eqref{Eq:Hamiltonian} and the first integral
\eqref{Eq:First_Integral} and then replace $J_3^2+c_1 x_1$ by $z$.
As a result, the equations \eqref{Eq:Hamiltonian} and
\eqref{Eq:First_Integral} take the form \eqref{Eq:Parametric Curve},
as required.

\item It can be explicitly checked that for any point from the Families
\ref{Rank1SeriesC}, \ref{Rank1SeriesC1} and \ref{Rank1SeriesC2} one
of the two equations \eqref{Eq:Left Parabola} and \eqref{Eq:Right
Parabola} holds.

In this case it is easier to verify first that $k =
\left(-\frac{\lambda}{2}\right)^2$, where $\lambda$ is the
coefficient of proportionality between $X_K$ and $X_H$ (that is $X_K
+ \lambda X_H= 0$), and then to check that
\[-\frac{\lambda}{2} = h-\varkappa c_1^2 -\frac{a}{\varkappa } \pm
\frac{\sqrt{a^2-4 \varkappa b^2 }}{\varkappa }.\]

While testing this equality it is convenient to use the relation
\[ \frac{a}{b}= \frac{x_3}{J_3} + \varkappa \frac{J_3}{x_3},\]
which holds if $J_3 \ne 0$ and $x_3 \ne 0$.
\end{enumerate} Assertion \ref{A:Rank1_Preimages_Images} is proved. \end{proof}

In what follows we need the following statement about the critical
points of the family \ref{Rank1SeriesB} from Assertion
\ref{A:Rank_1_Preimage}.

\begin{assertion} \label{A:Rank1Series2_2Cirles}
Let $\varkappa\ne 0$ and $b \ne 0$. Then for $z^2> \frac{a +
\sqrt{a^2-4 \varkappa b^2 }}{2}c_1^2$ either there is no critical
points in the preimage of points of the curve $\eqref{Eq:Parametric
Curve}$ or the critical points in the preimage from two critical
circles and the symmetry $(J_3, x_3) \to (-J_3, -x_3)$ interchanges
these circles.
\end{assertion}

\begin{proof}
It is not hard to check that for a fixed parameter $z$ the critical
points form the family \ref{Rank1SeriesB} are given by the following
equations \[ J_1 = \frac{b c_1}{z}, \quad J_2 =0, \quad x_1  =
\frac{z- J_3^2}{c_1}, \quad x_3 = \frac{b}{z} J_3,
\] where the coordinates $J_3$ and $x_2$ satisfy an equation of the
form
\begin{equation} \label{Eq:Series2_EquivCond} \left(\frac{J_3^2} {c_1} +
d \right)^2 + x_2^2 = R^2\end{equation} for some constants $d$ and
$R$ depending on $\varkappa, a, b$ and $z$. Thus in order to prove
the assertion it remains to prove that $J_3 \ne 0$ for $z^2> \frac{a
+ \sqrt{a^2-4 \varkappa b^2 }}{2}c_1^2$. It is not hard to verify
explicitly: if $J_3 =0$, then the equation
\eqref{Eq:Series2_EquivCond} has the form
\[ z^4 - a c_1^2 z^2 + \varkappa b^2 c_1^4 = 0, \] which has a solution precisely when \[  \frac{a - \sqrt{a^2-4 \varkappa b^2
}}{2}c_1^2 < z^2<  \frac{a + \sqrt{a^2-4 \varkappa b^2 }}{2}c_1^2.
\] Assertion \ref{A:Rank1Series2_2Cirles} is proved. \end{proof}

Now let us prove Lemma \ref{L:Curve_Images_B_Zero_Kappa_not_0}.

\begin{assertion} \label{A:Rank1_Preimages_Images_B0} Let $\varkappa \ne 0$ and $b = 0$. Then the images of the families
of critical points described in Assertion \ref{A:Rank_1_Preimage}
are arranged as follows:

\begin{enumerate}

\item The images of critical points from the family \ref{Rank1SeriesA} lie on the line $k=0$.

\item The images of critical points from the family  \ref{Rank1SeriesB}
belong to the union of the parabola \eqref{Eq:Up Parabola B Zero}
and the tangent line \eqref{Eq:Tangent Line B Zero}.

\item Images of the critical points from the families \ref{Rank1SeriesC},
\ref{Rank1SeriesC1}, \ref{Rank1SeriesC2} and \ref{Rank1SeriesC3} lie
on the union of two parabolas \eqref{Eq:Left Parabola B Zero} and
\eqref{Eq:Right Parabola B Zero}.

\end{enumerate}

\end{assertion}

\begin{proof}

The proof of this statement is almost identical to the proof of
Assertion \ref{A:Rank1_Preimages_Images} except for the following.
First, the case $J_3^2+c_1 x_1=0$ should be considered while
determining the image of the family \ref{Rank1SeriesB}. It is not
hard to check explicitly that in this case the image lies on the
parabola \eqref{Eq:Up Parabola B Zero}. Second, we have to consider
the family \ref{Rank1SeriesC3} and show that its image lies on the
left parabola  \eqref{Eq:Left Parabola B Zero}.
\end{proof}

We now prove Assertion \ref{A:NonDegen_Crit_Rank_1} about the
nondegeneracy of critical points of rank $1$. In order to prove it
we use the following simple criterion of nondegeneracy for critical
points of rank $1$ (see \cite{BolsinovFomenko99}).

\begin{assertion} \label{A:NonDegen_Criterium}
Let $(M^4, \omega)$ be a symplectic manifold and $y_0 \in (M,
\omega)$ be a critical point of rank $1$ for an integrable system
with Hamiltonian $H$ and integral $K$. Denote by $F= \alpha H +
\beta K$ a nontrivial linear combination for which the point $y_0$
is a critical point and by $A_F$ the linearization of the
corresponding Hamiltonian vector field $X_F$ at this point $y_0$.
The point $y_0$ is nondegenerate if and only if the operator $A_F$
has a non-zero eigenvalue.
\end{assertion}

In the case under consideration the system is defined on a Poisson
manifold. Hence it is convenient to use the following statement.

\begin{assertion} \label{A:HamVectorLinear_Poisson}
Suppose that in local coordinates $(p^1, \dots, p^k, q^1, \dots, q^k,\\
z^1, \dots, z^m)$ in a neighbourhood of a point $x_0$ a Poisson
bracket has the form $\sum_{i=1}^k \frac{\partial}{\partial p^i}
\wedge \frac{\partial}{\partial q^i}$. Suppose also that $x_0$ is a
critical point for a Hamiltonian vector field $X_F$ with Hamiltonian
$F$. Then the linearization $A_{F}$ of the Hamiltonian vector field
$X_F$ has the form
\begin{gather*}(\frac{\partial^2 F}{
\partial q^i \partial p^j} \frac{\partial}{\partial p^i} \otimes dp^j + \frac{\partial^2 F}{
\partial q^i \partial q^j} \frac{\partial}{\partial p^i} \otimes
dq^j + \frac{\partial^2 F}{
\partial q^i \partial z^j} \frac{\partial}{\partial p^i} \otimes
dz^j)
-\\
(\frac{\partial^2 F}{ \partial p^i \partial p^j}
\frac{\partial}{\partial q^i} \otimes dp^j + \frac{\partial^2 F}{
\partial p^i \partial q^j} \frac{\partial}{\partial q^i} \otimes
dq^j + \frac{\partial^2 F}{
\partial p^i \partial z^j} \frac{\partial}{\partial q^i} \otimes
dz^j).
\end{gather*}

\end{assertion}

Thus if $\hat{F}$ is the restriction of the function $F$ to a
symplectic leaf, then the spectrum of the linearization of the
vector field $X_F$ can be obtained from the spectrum of the
linearization of the vector field $X_{\hat{F}}$ by adding zeros in
the amount equal to the codimension of the symplectic leaf.
Therefore, as well as in the symplectic case, in order to check the
nondegeneracy of points of rank $1$ it is sufficient to verify that
the spectrum of the corresponding operator does not consist solely
of zeros.

This can be verified explicitly. To simplify the verification in the
following assertion we specify the coefficient of proportionality
$\lambda$ between the Hamiltonian vector fields corresponding to the
Hamiltonian \eqref{Eq:Hamiltonian} and the integral
\eqref{Eq:First_Integral} as well as describe the spectrum of the
linearization of the corresponding Hamiltonian vector field $X_{K+
\lambda H}$ for all critical points of rank $1$ of the integrable
Hamiltonian system under consideration (that is, for all points from
Assertion \ref{A:Rank_1_Preimage} except for the points of rank $0$
from Assertion \ref{A:Rank_0_Preimage}).

\begin{assertion}\label{A:Rank1_Coeff_Spectr}
For each critical point of rank $1$ from Assertion
\ref{A:Rank_1_Preimage} we specify $\lambda$ such that $X_K +
\lambda X_H= 0$ at this point and $\mu$ such that the spectrum of
the operator $A_{K+ \lambda H}$ consists of four zero and $\pm \mu$.

Four-parameter families:

\begin{enumerate}

\item Family \ref{Rank1SeriesA}. The coefficient of proportionality $\lambda =
0$, that is $X_K = (0, 0, 0, 0, 0, 0)$. The nontrivial eigenvalue:
\[\mu= 8 i |\left(\varkappa c_1^2  + J_1^2+J_2^2\right) J_3 -2 c_1
J_1 x_3|.\]

\item Family \ref{Rank1SeriesB}. The coefficient of proportionality: \[\lambda = 2 \left(\varkappa c_1^2
-J_1^2\right),\] the square of the eigenvalue:
\[ \mu^2 = 64 c_1 \left(J_1^2-J_3^2 - c_1 x_1 \right) \left(
\left(J_1^2-c_1 x_1\right) \left(x_1-\varkappa c_1 \right)-c_1 x_2^2
\right). \]

\item Family \ref{Rank1SeriesC}. The coefficient of proportionality: \[\lambda = 2 \left(\varkappa c_1^2
+J_1^2+J_2^2- 2 c_1 J_1 \frac{ x_3}{J_3}\right),\] the square of the
eigenvalue: \[\mu^2= -32\lambda \left(\left(J_1 J_3-c_1
x_3\right){}^2+\left(\left(J_1^2+J_2^2\right) -c_1 \frac{J_1
x_3}{J_3}\right){}^2+J_2^2 \left(\varkappa c_1^2
+J_3^2\right)\right).\]

\end{enumerate}

Three-parameter families:

\begin{enumerate}

\setcounter{enumi}{3}

\item Family \ref{Rank1SeriesC1}. The coefficient of proportionality: \[\lambda = 2 \left(\varkappa c_1^2
+J_1^2-2 c_1 x_1 \right),\] the square of the eigenvalue:
\[\mu^2=-32 \lambda  \left(\left(J_1^2-c_1 x_1\right){}^2+\left(J_1
J_3-c_1 x_3\right){}^2\right)\]

\item Family \ref{Rank1SeriesC2}. In this item we assume that $x_2 \ne 0$ since all point from the family
\ref{Rank1SeriesC2} that satisfy the condition $x_2=0$ either have
rank $0$ or belong to the family \ref{Rank1SeriesC3}. The
coefficient of proportionality:
\[\lambda =  2(\varkappa c_1^2 -J_1^2 + J_2^2)+ \frac{4}{x_2} J_1
J_2 \left( x_1 - \varkappa c_1 \right).  \] If in addition $x_1 \ne
\frac{\varkappa c_1^2 +J_1^2-J_2^2}{2 c_1}$, then the square of the
eigenvalue is equal to: \[\mu^2 = \frac{ 16 \lambda^2 J_2 \gamma}
{x_2 \left(\varkappa c_1^2 +J_1^2-J_2^2 -2 c_1 x_1\right)},\] where
\[ \gamma = \left(c_1 J_1 x_2^2 -J_1 \left(x_1- \varkappa c_1
\right) \left(J_1^2-J_2^2 -c_1 x_1\right) - J_2 x_2 \left(\varkappa
c_1^2 +J_1^2-J_2^2 -2 c_1 x_1\right)\right).
\]

If $x_1 = \frac{\varkappa c_1^2  +J_1^2-J_2^2}{2 c_1}$, then either
$x_2 = \frac{J_1 J_2}{c_1}$ or $x_2 = \pm \frac{\varkappa c_1^2
-J_1^2+J_2^2}{2 c_1}$. In the first case $\mu=0$, in the second case
\[ \mu^2 = -32 J_2^2 \lambda \left(\varkappa c_1^2 +\left(J_1 \mp
J_2\right){}^2\right).\]

\item Family \ref{Rank1SeriesC3}. The coefficient of proportionality:  \[\lambda = 2 \left(\varkappa c_1^2
-J_2^2 -2 c_1 x_1 \right),\] the square of the eigenvalue:
\[\mu^2=-32 \lambda  c_1^2 \left(\varkappa J_2^2  +x_1^2+x_3^2\right).\]

\end{enumerate}

\end{assertion}

\begin{proof}[Assertion \ref{A:NonDegen_Crit_Rank_1}]
We use Assertion \ref{A:NonDegen_Criterium} to prove the
nondegeneracy of points of rank $1$. The coefficients of
proportionality and the spectrum of the corresponding operators are
described in Assertion \ref{A:Rank1_Coeff_Spectr}. After thats the
nondegeneracy is proved by exhaustion. Assertion
\ref{A:NonDegen_Crit_Rank_1} is proved.
\end{proof}

\subsection{Types of bifurcation diagrams. (Case $b\ne 0$)} \label{SubS:Bif_Diag_B_Not_Zero}

In this section we show that the curves from Lemma
\ref{L:Curve_Images_B_not_0_Kappa_not_0} are positioned relative to
each other as it is shown in Fig.
\ref{Fig:Koval_U_1_A}--\ref{Fig:Koval_D_2_B} (for the corresponding
values of the parameters $a$ and $b$). Thereby we actually describe
all possible bifurcation diagrams of the momentum mapping. We are
interested in the singular points of bifurcation diagrams, that is
in the cusps, in the points of intersection and tangency of these
curves in the first place.

First of all, it is obvious that for any values of the parameters
$a$ and $b$ the parabolas \eqref{Eq:Left Parabola} and
\eqref{Eq:Right Parabola} intersect the line  $k=0$ at the points
\[h= \varkappa c_1^2 +\frac{a}{\varkappa } -\frac{\sqrt{a^2-4
\varkappa b^2 }}{\varkappa }, \qquad  \text{ and } \qquad h=
\varkappa c_1^2 +\frac{a}{\varkappa } +\frac{\sqrt{a^2-4 \varkappa
b^2 }}{\varkappa } \] respectively and intersect each other at the
point \begin{equation} h= \varkappa c_1^2 +\frac{a}{\varkappa } ,
\qquad k= \frac{a^2-4 \varkappa b^2 }{\varkappa^2 }.\end{equation}
Thus, it remains to describe the relative position of the curve
\eqref{Eq:Parametric Curve} with respect to the line $k=0$ and the
parabolas \eqref{Eq:Left Parabola}, \eqref{Eq:Right Parabola} .

In this section we first describe this curve \eqref{Eq:Parametric
Curve} (see Assertion \ref{A:Parametric_Curve_Description}), then we
determine the number of its points of intersection with the
parabolas described above (see Assertion
\ref{A:Paramtric_And_Parabolas}) and the line $k=0$ (see Assertion
\ref{A:Parametric_Curev_Number_of_Zeroes}). After that the rest of
the section is devoted to the study of the mutual interposition of
the found ``singular'' points. The final result can be formulated
as.

\begin{lemma} \label{L:Bif_Diag_Weak}
The functions $f_k, f_r, f_m, f_t$ and $f_l$ given by the formulas
\eqref{Eq: Function_F_K}, \eqref{Eq:Function_F_r},
\eqref{Eq:Function_F_m}, \eqref{Eq:Triple_Intersect} and
\eqref{Eq:Function_F_l} respectively divide the area $\{b>0, a> 2
\sqrt{\varkappa} b\}$ into $9$ sub-areas. For each of these
sub-areas the cusps, the point of intersection and tangency of the
line $k=0$, the parabolas \eqref{Eq:Left Parabola}, \eqref{Eq:Right
Parabola} and the curve \eqref{Eq:Parametric Curve} are located on
this four curves as it is shown in Fig.
\ref{Fig:Koval_U_1_A}--\ref{Fig:Koval_D_2_B}.
\end{lemma}

We begin with a description of the curve \eqref{Eq:Parametric
Curve}.

\begin{assertion}\label{A:Parametric_Curve_Description}
Let $\varkappa \ne 0$ . Then for any $a, b \in \mathbb{R}$, $b \ne
0$, the curve \eqref{Eq:Parametric Curve} has one cusp point
\begin{equation} \label{Eq:Parametric_Point_Return}
z_{\textrm{cusp}} = \sqrt[3]{b^{2} c_1^{2}}\end{equation} and two
points of local extrema
\begin{equation} \label{Eq:Parametric_Point_Extr} z_{\textrm{+ext}}
= \frac{|b|}{\sqrt{\varkappa}} \qquad \text{ and } \qquad
z_{\textrm{-ext}} = -\frac{|b|}{\sqrt{\varkappa }}.\end{equation}

The point $z_{\textrm{-ext}}$ is a local minimum for any values of
the parameters $a$ and $b$. If $b>\varkappa^{3/2} c_1^2$, then the
point $z_{\textrm{+ext}}$ is a local maximum. If $b<\varkappa^{3/2}
c_1^2$, then the point $z_{\textrm{+ext}}$ is a local minimum. (If
$b=\varkappa^{3/2} c_1^2$, then the point $z_{\textrm{+ext}}$
coincides with the cusp $z_{\textrm{cusp}}$.) In other words, the
function $k(z)$ monotonically increases between the points
$z_{\textrm{+ext}}$ and $z_{\textrm{cusp}}$ and monotonically
decreases on the remaining parts of the ray $z>0$.

The graph of the corresponding function is convex upward for
$z<z_{\textrm{cusp}}$ and convex downward for $z>
z_{\textrm{cusp}}$.

As $z \to \pm \infty$ the curve \eqref{Eq:Parametric Curve}
asymptotically tends to the line \[k= - 2\varkappa c_1^2 h+ (4 a
c_1^2 \varkappa^2 c_1^4) .\] Moreover, as $z\to \pm 0$ both
functions $h(z)$ and $k(z)$ simultaneously tend to $+ \infty$ and
besides \[\frac{k(z)}{h^2(z)} \underset{z\rightarrow \pm
0}\longrightarrow 1.\] \end{assertion}

We now describe the points of intersection of the curve
\eqref{Eq:Parametric Curve} with the other curves: with the
parabolas \eqref{Eq:Left Parabola} and \eqref{Eq:Right Parabola} and
the line $k=0$. We start with the points of intersection with the
parabolas. The proof of the following assertion is by direct
computation.

\begin{assertion} \label{A:Paramtric_And_Parabolas}
Let $\varkappa \ne 0$ and $b \ne 0$. Then the curve
\eqref{Eq:Parametric Curve} and the left parabola \eqref{Eq:Left
Parabola} have two points of intersection and one point of tangency.
For the points of intersection the corresponding values of the
parameter $z_{+l}$ and $z_{-l}$ are given by the relation
\begin{equation} \label{Eq:Parametric_Point_Left_Intersect} z^2 =
\frac{a + \sqrt{a^2-4 \varkappa b^2 }}{2}c_1^2.\end{equation} The
tangency point corresponds to the value of the parameter
\begin{equation} \label{Eq:Parametric_Point_Left_Tangent} z_{lt} =
\frac{a - \sqrt{a^2-4 \varkappa b^2}}{2 \varkappa }.\end{equation}

Analogously, the curve \eqref{Eq:Parametric Curve} and the right
parabola \eqref{Eq:Right Parabola} have two points of intersection
with the corresponding values of the parameter $z_{+r}$ and $z_{-r}$
given by the relation \begin{equation}
\label{Eq:Parametric_Point_Right_Intersect} z^2 = \frac{a -
\sqrt{a^2-4 \varkappa b^2 }}{2}c_1^2 \end{equation} and one point of
tangency corresponding to the value of the parameter
\begin{equation} \label{Eq:Parametric_Point_Right_Tangent} z_{rt} = \frac{a + \sqrt{a^2-4 \varkappa b^2}}{2
\varkappa }.\end{equation}

\end{assertion}

\begin{remark}
In Assertion \ref{A:Paramtric_And_Parabolas} (and further in the
text) we assume that $z_{+l}>0$ and $z_{+r}>0$. As a consequence
$z_{-l}<0$ and $z_{-r}<0$.
\end{remark}

Now let us find the number of intersection points of the line $k=0$
and the curve \eqref{Eq:Parametric Curve}.

\begin{assertion} \label{A:Parametric_Curev_Number_of_Zeroes} Suppose that $\varkappa > 0$. Then the number of intersection points of the line $k=0$
and the curve  \eqref{Eq:Parametric Curve} depends on the values of
the parameters $a$ and $b$ as follows. Consider the function
\[ f_k(b) = \frac{3 b^{4/3}+6 \varkappa b^{2/3} c_1^{4/3} -
\varkappa ^2c_1^{8/3} }{4 c_1^{2/3}}. \]

\begin{enumerate}
\item Assume that $b> \varkappa^{3/2} c_1^2$. If $a> f_k(b)$, then
the line $k=0$ and the curve \eqref{Eq:Parametric Curve} have
exactly $3$ points of intersection. If $a< f_k(b)$, then there is
only $1$ point of intersection and if $a= f_k(b)$, then there are
$2$ points of intersection.

\item Assume that $0< b < \varkappa^{3/2} c_1^2$. If $a< f_k(b)$,
then the line $k=0$ and the curve \eqref{Eq:Parametric Curve} have
exactly $3$ points of intersection. If $a> f_k(b)$,  then there is
only $1$ point of intersection and if $a= f_k(b)$, then there are
$2$ points of intersection.

\item If $b = \varkappa^{3/2} c_1^2$, then the line $k=0$ and the curve \eqref{Eq:Parametric Curve} have exactly $1$
point of intersection for any value of the parameter $a$.

\end{enumerate}

\end{assertion}

The results of Assertion \ref{A:Parametric_Curev_Number_of_Zeroes}
are collected together in table \ref{Tab:Number_of_Zeroes}.

\begin{table}[h]
\centering
\begin{tabular}{| l || c | c | c |}
  \hline
  & $0<b^2<\varkappa^3 c_1^4$ & $b^2=\varkappa^3 c_1^4$ & $b^2> \varkappa^3 c_1^4$ \\
  \hline \hline
  $a> f_k(b)$ & 1 & 1 & 3 \\
  \hline
  $a= f_k(b)$ & 2 & 1 & 2\\
  \hline
  $a<f_k(b)$ & 3 & 1 & 1\\
\hline
\end{tabular}\caption{Number of intersection points of the curve \eqref{Eq:Parametric Curve} and the line $k=0$.}
  \label{Tab:Number_of_Zeroes}
\end{table}

\begin{proof}
It is clear that there is no points of intersection if $z<0$ since
\[k(z) = 4 a c_1^2 -4 \varkappa c_1^2 z -\frac{4 b^2
c_1^2}{z}+\left(\frac{b^2 c_1^2}{z^2}- \varkappa c_1^2 \right)^2>
0.\] It follows from Assertion \ref{A:Parametric_Curve_Description}
that the function $k(z)$ has two local extrema if $z>0$:
\[z_{\textrm{+ext}} = \frac{b}{\sqrt{\varkappa }} \qquad \text{and}
\qquad z_{\textrm{cusp}} = \sqrt[3]{b^{2} c_1^{2}}.\] It remains to
examine the location of these extrema and whether the value of the
function $k(z)$ at these points is greater than zero. Assertion
\ref{A:Parametric_Curev_Number_of_Zeroes} is proved.
\end{proof}

Thus we found the points $z_{\textrm{cusp}}, z_{\pm \textrm{ext}},
z_{\pm l}, z_{\pm r}, z_{lt}$ and $z_{rt}$ on the curve
\eqref{Eq:Parametric Curve} given by the formulas
\eqref{Eq:Parametric_Point_Return},
\eqref{Eq:Parametric_Point_Extr},
\eqref{Eq:Parametric_Point_Left_Intersect},
\eqref{Eq:Parametric_Point_Left_Tangent},
\eqref{Eq:Parametric_Point_Right_Intersect} and
\eqref{Eq:Parametric_Point_Right_Tangent} respectively. We now
determine the order of these points on the $z$-axis. It is obvious
that \[z_{-l} < z_{\textrm{-ext}} < z_{-r} <0, \] and that the value
of the parameter $z$ for the other described points is greater than
zero. The next assertion is proved by direct calculation.

\begin{assertion} \label{A:Order_of_Points}
Let $\varkappa>0$. Then, depending on the values of the parameters
$a$ and $b$, the points $z_{\textrm{cusp}}, z_{+\textrm{ext}},
z_{+l}, z_{+r}, z_{lt}$ and $z_{rt}$ (described in Assertions
\ref{A:Parametric_Curve_Description} and
\ref{A:Paramtric_And_Parabolas}) are arranged on the ray $z>0$ so as
it is described in Tables
\ref{Tab:Points_on_Parametric_Curve_B_Small} and
\ref{Tab:Points_on_Parametric_Curve_B_Big}. Here functions $f_r(b)$
and $f_m(b)$ are given by the formulas \eqref{Eq:Function_F_r} and
\eqref{Eq:Function_F_m} respectively.
\end{assertion}

\begin{table}[h]
\centering
\begin{tabular}{| l | c  |  }
\hline
  $a> f_m(b)$ & $z_{rt}>z_{+l}>z_{\textrm{cusp}}>z_{+\textrm{ext}}>z_{+r}>z_{lt}$  \\
  \hline
  $a= f_m(b)$ & $z_{rt}=z_{+l}>z_{\textrm{cusp}}>z_{+\textrm{ext}}=z_{+r}>z_{lt}$  \\
  \hline
  $f_r(b)<a<f_m(b)$ & $z_{+l}>z_{rt}>z_{\textrm{cusp}}>z_{+r}>z_{+\textrm{ext}}>z_{lt}$ \\
\hline
  $a=f_r(b)$ & $z_{+l}>z_{rt}=z_{\textrm{cusp}}=z_{+r}>z_{+\textrm{ext}}>z_{lt}$ \\
\hline
  $a<f_r(b)$  & $z_{+l}>z_{+r}>z_{\textrm{cusp}}>z_{rt}>z_{+\textrm{ext}}>z_{lt}$
\\  \hline
\end{tabular}\caption{Interposition of points of the curve \eqref{Eq:Parametric Curve} for $0<b^2<\varkappa^3 c_1^4$.}
  \label{Tab:Points_on_Parametric_Curve_B_Small}
\end{table}

\begin{table}[h]
\centering
\begin{tabular}{| l | c |   }
  \hline
  $a> f_m(b)$ &
$z_{rt}>z_{+l}>z_{+\textrm{ext}}>z_{\textrm{cusp}}>z_{+r}>z_{lt}$ \\
  \hline
  $a= f_m(b)$ &
$z_{rt}>z_{+l}=z_{+\textrm{ext}}>z_{\textrm{cusp}}>z_{+r}=z_{lt}$ \\
  \hline
  $f_r(b)<a<f_m(b)$ &
$z_{rt}>z_{+\textrm{ext}}>z_{+l}>z_{\textrm{cusp}}>z_{lt}>z_{+r}$ \\
\hline
  $a=f_r(b)$ &
$z_{rt}>z_{+\textrm{ext}}>z_{+l}=z_{\textrm{cusp}}=z_{lt}>z_{+r}$ \\
\hline
  $a<f_r(b)$  &
$z_{rt}>z_{+\textrm{ext}}>z_{lt}>z_{\textrm{cusp}}>z_{+l}>z_{+r}$
\\  \hline
\end{tabular}\caption{Interposition of points of the curve \eqref{Eq:Parametric Curve} for  $b^2> \varkappa^3 c_1^4$ .}
  \label{Tab:Points_on_Parametric_Curve_B_Big}
\end{table}

We now describe when three curves intersect at one point.

\begin{assertion}\label{A:Triple_Intersect}
Suppose that $\varkappa >0$, $b \ne 0$ and $a^2 - 4 \varkappa b^2
>0$. Then the line $k=0$, the curve \eqref{Eq:Parametric Curve} and the
right parabola \eqref{Eq:Right Parabola} can not intersect at one
point. The line $k=0$, the curve \eqref{Eq:Parametric Curve} and the
left parabola \eqref{Eq:Left Parabola} intersect at one point if and
only if
\begin{equation} \label{Eq:Triple_IntersectB} \begin{gathered} a = \left(\frac{\varkappa c_1^2 + t^2}{2 c_1} \right)^2 + \varkappa t^2,
\qquad b =t \left(\frac{\varkappa c_1^2 + t^2}{2 c_1} \right)
\end{gathered} \end{equation} for some $t \in \mathbb{R}$.
\end{assertion}

Note that in the formula \eqref{Eq:Triple_IntersectB} the parameter
$a$ is uniquely determined by $b$, therefore the formula
\eqref{Eq:Triple_IntersectB} defines a function, which we denoted by
$a= f_t(b)$ (see the formula \eqref{Eq:Triple_Intersect}).

\begin{proof}[of Assertion \ref{A:Triple_Intersect}]
It is not hard to check that under the conditions of the assertion
the points of rank $1$ that belong to the families
\ref{Rank1SeriesA}, \ref{Rank1SeriesB} and also to one of the
families \ref{Rank1SeriesC}, \ref{Rank1SeriesC1},
\ref{Rank1SeriesC2} and \ref{Rank1SeriesC3} are the points \[x_1 =
\frac{\varkappa c_1^2 + J_1^2}{2c_1}, \quad J_2 =0, \quad J_3 =0,
\quad x_2 =0, \quad x_3 = 0.\] For these points the equality
\eqref{Eq:Triple_IntersectB} holds with $t= J_1$. Therefore, if the
equality \eqref{Eq:Triple_IntersectB} holds, then the curves
intersect at one point.

We now show that if three curves intersect at one point, then the
equality \eqref{Eq:Triple_IntersectB} holds. It is not hard see that
the only point of the curve \eqref{Eq:Parametric Curve} that can be
a point of triple intersection is its point of intersection with the
left parabola $z_{+l}$ (it also easily follows from geometrical
considerations). Since by assumption the curves intersect at one
point the equality
\[\varkappa c_1^2 +\frac{a}{\varkappa } -\frac{\sqrt{a^2-4 \varkappa
b^2 }}{\varkappa } = \frac{b^2 c_1^2}{z_{+l}^2} + 2 z_{+l} \] holds.
Substituting the equality \eqref{Eq:Parametric_Point_Left_Intersect}
we get \begin{equation}\label{Eq:Triple_Z_Form} (z- \varkappa
c_1^2)^2 = \sqrt{a^2 -4 \varkappa b^2} c_1^2. \end{equation} If we
put \[z= \frac{t^2 + \varkappa c_1^2}{2},\] where $\text{sgn}(t) =
\text{sgn}(b)$, then we get
\begin{equation} \label{Eq:Triple_ab_comb_Form} \sqrt{a^2 -4
\varkappa b^2} = \frac{(t^2 - \varkappa c_1^2)^2}{4c_1^2}.
\end{equation} Substituting \eqref{Eq:Triple_Z_Form} and \eqref{Eq:Triple_ab_comb_Form}
in the formula \eqref{Eq:Parametric_Point_Left_Intersect} we obtain
the required expression for $a$ from the formula
\eqref{Eq:Triple_IntersectB}. This immediately implies the formula
\eqref{Eq:Triple_IntersectB} for $b$. Assertion
\ref{A:Triple_Intersect} is proved.
\end{proof}

It follows from Assertion \ref{A:Parametric_Curve_Description},
\ref{A:Paramtric_And_Parabolas},
\ref{A:Parametric_Curev_Number_of_Zeroes}, \ref{A:Order_of_Points}
and \ref{A:Triple_Intersect} that the interposition of the curves
from Lemma \ref{L:Curve_Images_B_not_0_Kappa_not_0} qualitatively
differs depending on position of the parameters $a$ and $b$ with
respect to the functions $f_k, f_r, f_m, $ and $f_t$, given by the
formulas \eqref{Eq: Function_F_K}, \eqref{Eq:Function_F_r},
\eqref{Eq:Function_F_m} and \eqref{Eq:Triple_Intersect}
respectively.

Since in the case $\varkappa>0$ the values of the parameters $a$ and
$b$ always satisfy the inequality $a^2 - 4 \varkappa b^2 \geq 0$ we
also consider the function
\[ f_l (b) = 2 \sqrt{\varkappa} |b|. \]

We now describe how the graphs of the functions $f_k, f_r, f_m, f_t$
and $f_l$ are positioned relative to each other. The next assertion
is proved by direct calculation.

\begin{assertion} \label{A:Functions_Interposition}
Let $\varkappa > 0$. Denote by $\alpha_0$ the only root of the
equation $x^3 + x^2 + x -1 =0$. The graphs of functions $f_k, f_r,
f_m, f_t$ and $f_l$ given by the formulas \eqref{Eq: Function_F_K},
\eqref{Eq:Function_F_r}, \eqref{Eq:Function_F_m},
\eqref{Eq:Triple_Intersect} and \eqref{Eq:Function_F_l} respectively
are symmetrical about the axis $b=0$ and in the case $b>0$  they are
positioned as it is shown in Fig. \ref{Fig:Areas_Big} and
\ref{Fig:Areas_Small}. In other words, for $b>0$ all the graphs
intersect at the point
\[M = (\varkappa^{3/2} c_1^2, \, \, 2 \varkappa^2 c_1^2),\] and the
graphs of functions $f_r$ and $f_t$ intersect at the point
\[N = (\alpha_0^3 \varkappa^{3/2} c_1^2, \, \, f_r(\alpha_0^3
\varkappa^{3/2} c_1^2)).\] Further, if $0< b^2 < \alpha_0^6
\varkappa^3 c_1^4$, then
\[f_k < f_l < f_r < f_t < f_{m}.\] If $\alpha_0^6 \varkappa^3
c_1^4< b^2< \varkappa^3 c_1^4$, then
\[f_k < f_l < f_t< f_r < f_{m}.\] And if $b^2 > \varkappa^3 c_1^4$,
then \[f_l < f_t < f_k < f_r < f_{m}.\] \end{assertion}

\begin{proof}[of Lemma \ref{L:Bif_Diag_Weak}] Lemma \ref{L:Bif_Diag_Weak} easily follows from
earlier proved Assertions \ref{A:Parametric_Curve_Description},
\ref{A:Paramtric_And_Parabolas},
\ref{A:Parametric_Curev_Number_of_Zeroes}, \ref{A:Order_of_Points},
\ref{A:Triple_Intersect} and simple geometric considerations.
\end{proof}

\subsection{Critical points of rank 0} \label{SubS:Critical_Points_Zero_Rang}

In this section we prove Lemmas \ref{L:RankZeroType} and
\ref{L:RankZeroType_BZero} about types of critical points of rank
$0$. At first we describe all critical points of the momentum
mapping of rank $0$ and then we find their types and images under
the momentum mapping. Recall that on non-singular orbits (that is,
on orbits $M_{a, b}$ such that $a^2 - 4 \varkappa b^2 >0$)
non-singular points of rank $0$ are precisely the points at which
both Hamiltonian vector field $X_H$ and $X_K$ vanish. Let us
emphasize that the following statement holds for any value of the
parameter $\varkappa \in \mathbb{R}$.

\begin{assertion}\label{A:Rank_0_Preimage}
The set of points where both Hamiltonian vector fields $X_H$ and
$X_K$ with Hamiltonians \eqref{Eq:Hamiltonian} and
\eqref{Eq:First_Integral} respectively vanish is the union of the
following (two-parameter) families of points in $\mathbb{R}^6(J,
x)$:

\begin{enumerate}

\item $(J_1, J_2, 0, \varkappa c_1, 0, 0),$

\item $(J_1, 0, 0, x_1, 0, 0),$
\item $(J_1, 0, J_3, \frac{\varkappa c_1^2 +J_1^2 }{2c_1}, 0, \frac{J_1
J_3}{c_1}).$

\end{enumerate}

\end{assertion}

\begin{proof} The vector field $X_H$ has the following coordinates: \begin{gather*} \{J_1, H\} = -2 J_
2 J_ 3, \qquad \{J_2, H\} = 2 J_ 1 J_ 3-2 c_ 1 x_ 3 \\ \{J_3,
H\} = 2 c_ 1 x_ 2, \qquad \{x_1, H\} = 2 J_ 2 x_ 3-4 J_ 3 x_2 \\
\{x_2, H\} =  4 J_ 3 x_ 1-2 J_ 1 x_ 3 -2 \varkappa c_ 1 J_ 3, \qquad
\{x_3, H\} =  2 J_ 1 x_ 2 -2 J_ 2 x_ 1 +2 \varkappa c_ 1 J_ 2
\end{gather*}
It is easy to see that $x_2 = 0$ and that either $J_2 =0$ or $J_3=0,
x_3=0$. The rest of the proof is by exhaustion.
\end{proof}

\begin{proof}[of Lemmas \ref{L:RankZeroType} and \ref{L:RankZeroType_BZero}]

It can be proved by direct calculation that the images of the points
of rank $0$ lie on these curves. It is possible to explicitly find
all the points from the \ref{I:CritRank0Image_2Parab} and
\ref{I:CritRank0Image_ParamParab} series on each orbit $M_{a, b}$.
We now prove the statement about the number of points for the
\ref{I:CritRank0Image_ParamK0} series on the orbit $M_{a, b}$. The
case $b=0$ is trivial, hence we assume that $b \ne 0$. It is not
hard to verify that the set of points (or rather the corresponding
values of the parameters $J_1$ and $J_3$) is given by the following
system of equations \begin{equation}
\label{Eq:CritRank0Image_ParamK0}
\begin{gathered}
J_1^5  -2 \varkappa c_1^2 J_1^3  -4 b c_1 J_1^2 +  \left(4 a c_1^2 +
\varkappa ^2 c_1^4 \right) J_1 -4 \varkappa b c_1^3   =0 \\ J_3^2 =
\frac{2 b c_1- \varkappa c_1^2 J_1  -J_1^3}{2 J_1}.
\end{gathered} \end{equation} It is clear that there are exactly two
points in the preimage for each point in the image (the coordinates
$J_1$ for these points coincide and the coordinates $J_3$ are
opposite). In order to find the exact number of solutions let us
first divide the set of parameters $(a, b)$ into areas for which
this number is constant and then solve this problem for each of the
areas. Let us slightly simplify the equation
\eqref{Eq:CritRank0Image_ParamK0} by putting \[\hat{b} =
\frac{b}{\varkappa^{3/2}c_1^2}, \qquad \hat{a} =\frac{a}{\varkappa^2
c_1^2}, \qquad s = \frac{J_1}{\varkappa^{1/2}c_1}.\] Then the number
of points in the image is equal to the number of solutions of the
equation
\begin{equation} \label{Eq:ParamK0_Simple5} s^5 - 2\hat{b} s^3 - 8
\hat{b} s^2 + (4\hat{a} +1) s - 4\hat{b}=0\end{equation} on the
segment
\begin{equation} \label{Eq:ParamK0_SimpleCond} 0\leq
 s+ s^3 \leq 2\hat{b}.\end{equation} Note that the function $s+ s^3$ is monotonically increasing, so the
inequality \eqref{Eq:ParamK0_SimpleCond} defines a segment. Note
that if we vary the parameters $a$ and $b$, then the number of
solutions can change only in the following cases:
\begin{enumerate}
\item The equation \eqref{Eq:ParamK0_Simple5} has multiple roots.

\item One of the endpoints of the segment
\eqref{Eq:ParamK0_SimpleCond} is a solution of the equation
\eqref{Eq:ParamK0_Simple5}. (It is easy to see that the point $0$
can not be a root of the equation \eqref{Eq:ParamK0_Simple5},
therefore we only need to check the point $s+s^3 = 2 \hat{b}$).
\end{enumerate}

It can be verified that the point  $s+s^3 = 2 \hat{b}$ is a solution
of the equation \eqref{Eq:ParamK0_Simple5} if and only if $a =
f_t(b)$ (recall that the function $f_t(b)$ is defined by the formula
\eqref{Eq:Triple_Intersect} and that $a =f_t(b)$ if and only if
three curves of the bifurcation diagram intersect at one point, see
Assertion \ref{A:Triple_Intersect}). It can also be verified that in
a sufficiently small neighbourhood of any point of the curve
$a=f_t(b)$ (except, maybe, for the points at which the equation
\eqref{Eq:ParamK0_Simple5} has multiple roots) the points in the
area $a>f_t(b)$ have one more point in the preimage than the points
in the area $a<f_t(b)$.

Further, it is easy to verify that the equation
\eqref{Eq:ParamK0_Simple5} has multiple roots in the following
cases: either $a = \pm 2 \sqrt{\varkappa} b$ or $a = f_k(b)$, where
$f_k(b)$ is given by the formula \eqref{Eq: Function_F_K}. In the
case $a = 2 \sqrt{\varkappa} |b|$ the only multiple root is $s=1$ of
multiplicity $2$. It follows that for $b^2< \varkappa^3 c_1^4$ and
$2 \sqrt{\varkappa} |b|< a<f_t(b)$ there is no point from the
\ref{I:CritRank0Image_ParamK0} series in the preimage and for $b^2>
\varkappa^3 c_1^4$ and $2 \sqrt{\varkappa} |b|< a<f_t(b)$ there are
exactly two point from the \ref{I:CritRank0Image_ParamK0} series in
the preimage. Since the number of points in the preimage increases
by $1$ when passing through the curve $a=f_t(b)$ by increasing the
parameter $a$ (for a fixed $b$) we conclude that in the area $a>
f_t(b), a> f_k(b)$ there is exactly one point in the preimage and
there are three points in the preimage for the area $f_k(b)>
a>f_t(b)$.

Thus the statement about the number of points on each orbit is
proved. It remains to prove about the statement about their types.
It is not hard to do using the following criteria (for more details
see \cite{BolsinovFomenko99}).

\begin{assertion} \label{A:NonDegen_Criterium_Rank_0}
Consider a symplectic manifold $(M^4, \omega)$ and suppose that $x_0
\in (M, \omega)$ is a critical point of rank $0$ for an integrable
Hamiltonian system with Hamiltonian $H$ and integral $K$. Then the
point $x_0$ is nondegenerate if and only if the linearizations $A_H$
and $A_K$ of the Hamiltonian vector fields $X_H$ and $X_K$ at the
point $x_0$ satisfy the following properties: \begin{enumerate}
\item the operators $A_H$ and $A_K$ are linearly independent,

\item there exists a linear combination $\lambda A_H + \mu A_K$ such that all its eigenvalues are different and not equal to $0$.

\end{enumerate}

Moreover, if the point $x_0$ is nondegenerate, then its type is
completely determined by the spectrum of any linear combination
$\lambda A_H + \mu A_K$ that has no zero eigenvalues. More precisely
the type of the point depends on the type of the spectrum as
follows.

\begin{itemize}

\item If the spectrum of a linear combination $\lambda A_H + \mu A_K$ has the form $\alpha, - \alpha, \beta,
-\beta$, where $\alpha, \beta \in \mathbb{R}-\{0\}$, then the point
$x_0$ is a critical point of saddle-saddle type.

\item If the spectrum has the form $i\alpha, - i\alpha, i\beta,
-i\beta$, where $\alpha, \beta \in \mathbb{R}-\{0\}$, then the point
$x_0$ is a critical point of center-center type.

\item If the spectrum has the form $i\alpha, - i\alpha, \beta,
-\beta$, where $\alpha, \beta \in \mathbb{R}-\{0\}$, then the point
$x_0$ is a critical point of center-saddle type.

\item If the spectrum has the form $\alpha + i\beta, \alpha - i\beta, -\alpha + i\beta,
-\alpha -i\beta$, where $\alpha, \beta \in \mathbb{R}-\{0\}$, then
the point $x_0$ is a critical point of focus-focus type.

\end{itemize}

\end{assertion}

We compute the spectrum of the linearization of the Hamiltonian
vector field $X_H$ with Hamiltonian \eqref{Eq:Hamiltonian} using
Assertion \ref{A:HamVectorLinear_Poisson}.

\begin{assertion} \label{A:Spectr_Ham_Rank0}

For all three series of critical points of rank $0$ from Assertion
\ref{A:Rank_0_Preimage} the spectrum of the linearization of the
Hamiltonian vector field $X_H$ with Hamiltonian
\eqref{Eq:Hamiltonian} contains a zero eigenvalue with multiplicity
$2$. The spectrum also contains the following elements:

\begin{enumerate}

\item For the \ref{I:CritRank0Image_2Parab} series of critical points of rank
$0$ the spectrum also contains the eigenvalues $\pm \sqrt{2}
\sqrt{\alpha \pm \sqrt{\alpha^2 + 4 \varkappa \beta^2}}$, where
\begin{gather*} \alpha = \varkappa c_1^2 - J_1^2 - J_2^2 = (2
\varkappa c_1^2 - \frac{a}{\varkappa} ) \\ \beta = c_1 J_2 = c_1^2 a
- \varkappa^2 c_1^4 - \frac{b^2}{\varkappa}\end{gather*}

\item For the \ref{I:CritRank0Image_ParamParab} series the spectrum also contains the eigenvalues \[\pm 2\sqrt{c_1 (x_1 -
\varkappa c_1)} \quad \text{and} \quad \pm 2 \sqrt{-J_1^2 +2 c_1 x_1
- \varkappa c_1^2}\]

\item For the \ref{I:CritRank0Image_ParamK0} series the spectrum also contains the eigenvalues \[\pm 4 i J_3 \quad \text{and} \quad \pm
\sqrt{2} \sqrt{J_1^2 - 2 J_3^2 - \varkappa c_1^2}\]

\end{enumerate}

\end{assertion}

We further note that in order to prove the nondegeneracy of the
points it suffices to verify that all the $4$ eigenvalues of the
operator $A_H$ are not equal to $0$ and that there exists $\lambda
\in \mathbb{R}$ such that the spectrum of the operator $A_F$, where
$F=K +\lambda H$, has exactly $2$ non-zero eigenvalues. Indeed, then
the restrictions of the operators $A_H$ and $A_K$ are linearly
independent since otherwise the spectrum of any linear combination
is obtained from the spectrum of the operator $A_H$ by a
multiplication by a constant (it follows from Assertion
\ref{A:HamVectorLinear_Poisson}). Moreover, there is no need to
verify that all the eigenvalues of the operator $A_H$ are distinct
since in this case the spectrum of some linear combination $H + \mu
F$ has $4$ different non-zero eigenvalues.

For the series \ref{I:CritRank0Image_ParamParab} and
\ref{I:CritRank0Image_ParamK0} the required coefficient of
proportionality  $\lambda$ has already been found in Assertion
\ref{A:Rank1_Coeff_Spectr} (we can put $\lambda = 2 \left(\varkappa
c_1^2 -J_1^2\right)$ and $\lambda = 0$ for the series
\ref{I:CritRank0Image_ParamParab} and \ref{I:CritRank0Image_ParamK0}
respectively). For the series \ref{I:CritRank0Image_2Parab} we can
use the fact that $k= \frac{\lambda^2}{4}$ for the critical points
in the preimage of points of the parabolas \eqref{Eq:Left Parabola}
and \eqref{Eq:Right Parabola} and formally put $\lambda =
\sqrt{\frac{k}{2}}$. Then for the function $F= K +\sqrt{\frac{k}{2}}
H$ the spectrum of the operator $A_F$ consists only of $4$ zeroes
and \[ \pm 4 \sqrt{2}\sqrt{\frac{4 \varkappa b^2-a^2 }{\varkappa^2}
\left( \frac{a^2 -2 \varkappa c_1^2}{\varkappa} + \sqrt{\frac{a^2 -4
\varkappa b^2}{\varkappa^2}} \right)}. \]

It is now not hard to check the nondegeneracy of points from Lemmas
\ref{L:RankZeroType} and \ref{L:RankZeroType_BZero}. (Note that the
degenerate points appear only at the boundaries of the areas of the
plane $\mathbb{R}^2(a, b)$ and their degenerations are related to
the restructurization of the bifurcation diagrams in a neighbourhood
of these points).

The types of critical points can be easily found using Assertions
\ref{A:NonDegen_Criterium_Rank_0} and \ref{A:Spectr_Ham_Rank0}.
Lemmas \ref{L:RankZeroType} and \ref{L:RankZeroType_BZero} are
proved.
\end{proof}

\subsection{Proof of Theorems \ref{T:Bif_Diag},
\ref{T:Bifurcations} and \ref{T:Molecules}}
\label{SubS:Proof_Main_Theorems}

Theorems \ref{T:Bif_Diag}, \ref{T:Bifurcations} and
\ref{T:Molecules} easily follow from earlier-proved Lemmas
\ref{L:Curve_Images_B_not_0_Kappa_not_0}, \ref{L:RankZeroType},
\ref{L:Curve_Images_B_Zero_Kappa_not_0}, \ref{L:RankZeroType_BZero}
and the following simple geometric considerations.
\begin{enumerate}
\item First, we use the fact that the orbits of the coadjoint
representation $M_{a, b}$ of the Lie algebra $\textrm{so}(4)$ are
compact. In particular, since the image of a compact set is compact,
it allows us to discard all unlimited domains during the
construction of bifurcation diagrams.

\item Second, we use some well known results about nondegenerate
singularities. All the required statements are described in detail
in the book \cite{BolsinovFomenko99} so we only briefly recall them.

\begin{enumerate}

\item There exists only one, up to Liouville equivalence, singular point
of center-center type. The bifurcation diagram in a neighbourhood of
a center-center singularity is the union of two curves emanating
from this point. The loop molecule of the singularity has the form
$A - A$ and the mark $r=0$.

\item Any singularity of center-saddle type is Liouville equivalent to a direct
product of a saddle atom and the elliptic atom $A$. In a
neighbourhood of an image of a center--saddle point the bifurcation
diagram is the union of a curve passing through this point and
another curve emanating from this point. The loop molecule is
obtained from the corresponding saddle atom by adding the atom $A$
at the end of each edge, all marks $r = \infty$. (For example the
loop molecule of the point $y_{12}$ in Table \ref{Tab:Cirle_Mol_Old}
has such form.)

\item There exists exactly $4$ singularities of saddle--saddle type of complexity
$1$ (that is, containing exactly one singular point on the leaf).
These singularities are completely determined by their loop
molecules. (The required two loop molecules for the saddle--saddle
singularities are given in Table \ref{Tab:Cirle_Mol_Old} for the
points $y_3$ and $y_7$.)

\end{enumerate}

\item Third, since we consider a two-parameter family of orbits $M_{a,
b}$ we can use the fact that some invariants of the system
continuously depend on the parameters $a$ and $b$ (for example a
small perturbation of the parameters does not change the number of
tori in the preimage of regular points from the ``same area'').

\end{enumerate}

We also use stability of bifurcations of types $A$, $B$ and $A^*$.

In this section we do not consider the classical Kovalevskaya case
($\varkappa =0$) because it was considered earlier in the paper
\cite{Harlamov88} and was described in detail in the book
\cite{BolsinovFomenko99}. Also, we do not consider the case $b=0$ in
much detail: the proof of the statements in this case is similar to
the proof in the case $b \ne 0$.

\begin{proof}[of Theorems \ref{T:Bif_Diag} and \ref{T:Bifurcations} ]

Previously it was shown that the required bifurcation diagrams of
the momentum mapping is contained in the union of curves described
in Lemma \ref{L:Curve_Images_B_not_0_Kappa_not_0}. Therefore, to
prove the theorems is remains to throw away several parts of the
described curves and determine the bifurcations for the remaining
arcs. Let us show how it can be done using the previously obtained
information about the types of critical points of rank $0$ (see
Lemmas \ref{L:RankZeroType} and \ref{L:RankZeroType_BZero}).

We prove Theorems \ref{T:Bif_Diag} and \ref{T:Bifurcations} for the
values of the parameters $a$ and $b$ from the area $I$, that is in
the case $|b|> \varkappa^{3/2} c_1^2$, $2 \sqrt{\varkappa} b < a<
f_t(b)$ (where the function $f_t(b)$ is given by the formula
\eqref{Eq:Triple_Intersect}). The remaining cases are treated
similarly.

In this case the bifurcation diagram of the momentum mapping should
have the form shown in Fig. \ref{Fig:Koval_U_1_A},
\ref{Fig:Koval_U_1_B}, \ref{Fig:Koval_U_1_C}. Denote by $P$, $Q$ and
$R$ the point of intersection of the curve \eqref{Eq:Parametric
Curve} with the line $k=0$, the point of intersection of the
parabolas \eqref{Eq:Left Parabola} and \eqref{Eq:Right Parabola} and
the point of intersection of the right parabola \eqref{Eq:Right
Parabola} with the line $k=0$ respectively.

First of all, we discard all unlimited domains (since the orbits of
$so(4)$ are compact) and all areas lying below the line $k=0$
(obviously, $k \geq 0$). Then notice that the ``curvilinear
triangles'' $y_{12}y_{13}P$ and $z_1 z_2 R$ do not belong to the
image of the momentum mapping. Indeed, the point $z_1$ is the
rightmost point in the image of the coadjoint orbits because all the
points in the preimage of the rightmost point on the line $k=0$ must
be of center--center type and there is no images of critical points
of rank $0$ to the right of the point $z_1$. Similarly, if the point
$P$ belonged to the image of the momentum mapping, then it would be
the leftmost and the lowest point in some of its neighbourhood.
Therefore, there would have to be points of rank $0$ in its
preimage, which is false.

Further, since we know the types of all points of rank $0$ we can
easily determine the bifurcations corresponding to all arcs that
contain an image of a critical points of rank $0$. In this case the
only ambiguity occurs for the point $y_{10}$. This point is the
image of two points of center-center type thus a priori there are
$3$ variants: for the area lying to the left of the point $y_{10}$
there are either $4$, or $2$ or $0$ tori in the preimage. Let us
show that only the first case is possible. If we increase the
parameter $a$, then in the case $f_t(b) < a < f_k(b)$ there appears
the point $y_7$ of saddle--saddle type, therefore there have to be
$4$ tori in the preimage of a point in the neighbouring camera. For
the reasons of continuity in the case $a< f_t(b)$ there should also
be $4$ tori for the camera to the left of the point $y_{10}$.

It remains to determine whether the ``curvilinear triangle'' $y_8
z_{2} Q$ and the arc $y_8 y_{13}$ belong to the bifurcation diagram
and what bifurcations correspond to these arcs. The curve $y_8 Q$
does not belong to the bifurcation diagram, and the curve $z_2 Q$
belongs since in this case there is no critical points of rank $0$
in the preimage of the point $q$. Here we use the following simple
assertion.

\begin{assertion} \label{A:Trans_Diagram_Rank_0}
Let $(M^4, \omega)$ be a  compact symplectic manifold, $H, K$ be two
commuting (with respect to the Poisson bracket) functions on $M^4$,
which are independent almost everywhere. Suppose that in a
neighbourhood of a point $x \in \mathbb{R}^2$ the bifurcation
diagram has the same structure as for the critical points of rank
$0$ of center--center type (that is, two arcs from the boundary of
the image intersect transversely) or of center--saddle type (that
is, an arc transversely intersects a smooth arc from the boundary of
the image of the momentum mapping). The there is a critical point of
rank $0$ in the preimage of the point $x$.
\end{assertion}

It remains to show that the arcs $y_{13} y_8$ and $y_8 z_2$ belong
to the bifurcation diagram and that the corresponding bifurcations
have types $2B$ and $2 A^*$ respectively. Let us start with the arc
$y_8 z_2$. We already know (see Assertion
\ref{A:Rank1Series2_2Cirles}) that the preimage of the points of
this arc is either empty or consists of two critical circle and in
the latter case the symmetry $(J_3, x_3) \to (-J_3, -x_3)$
interchanges these circles. First of all we show that if the
preimage consists of two critical circles, then the bifurcations
corresponding to each of them have type $A^*$. Since there are two
circles and they transform two tori into two and the system has a
symmetry, the only possible bifurcation apart from $2A^*$ is the
bifurcation $C_2$. In order to show that the latter case can not
occur let us consider the isoenergetic surfaces $H=\textrm{const}$
for the values of energy $H$ close to $h_0 = \frac{b^2
c_1^2}{z_{rt}^2} + 2 z_{rt}$, where $z_{rt}$ is given by the formula
\eqref{Eq:Parametric_Point_Right_Tangent} (this value of energy
corresponds to the point of tangency of the right parabola
\eqref{Eq:Right Parabola} and the parametric curve
\eqref{Eq:Parametric Curve}). Since there is no points where the
Hamiltonian vector field $X_H$ vanishes for the values of $H$ close
to $h_0$ the type of the isoenergetic surface does not change in a
neighbourhood of $h_0$. However, when the energy parameter increases
$H> h_0$ the isoenergetic surface is obviously disconnected:  its
rough molecule consists of two copies of $A - A$. Therefore, it is
disconnected for the lower values of $H$ as well. However, if the
bifurcation for the curve $y_8 z_2$ had type $C_2$, then the
isoenergetic surface would have been connected. Thus the only
bifurcation that can correspond to the curve $y_8 z_2$ is $2A^*$.

Now let us show that the bifurcations for the curve $y_8 z_2$ do
really exist. First of all we notice that for sufficiently large
values of the parameter $a$ (more precisely, for $a>f_m(b)$) these
bifurcations exist (and they are of type  $2A^*$). It suffices to
show that the loop molecule of the point  $y_3$ has the form shown
in Table \ref{Tab:Cirle_Mol_Old}. It is true because there are $3$
tori in the preimage of the points to the ``left'' of the point
$y_3$ and there are two tori in the preimage of the points to the
``right''. It easily follows from the analysis of the types of
critical points in the case $f_r(b)< a< f_m(b)$ and in the case $a>
f_m(b)$  it follows from the reasons of continuity. Note that for
$a> f_m(b)$  the bifurcation for the arc $y_3 z_2$ has type $2A^*$
since the critical point of rank $0$ in the preimage of the point
$z_3$ has center--saddle type and hence the bifurcation for the
curve $y_3 z_3$ has to be orientable.

From the reasons of continuity it follows that the bifurcations for
the curve $y_8 z_2$ exist in case $a< f_m(b)$. Let us describe the
last transition in more detail. On one hand, the bifurcations of the
type $A^*$ are stable, hence they survive under a small perturbation
of parameters $a$ and $b$. Therefore, the set of points $a, b$ for
which the bifurcation for the curve $y_8 z_2$ exists and has type
$2A^*$ is an open subset. On the other hand, the set of points where
the Hamiltonian vector fields $X_H$ and $X_K$ are linearly dependent
is a closed subset of $\mathbb{R}^7 = \mathbb{R}^7(\mathbf{J},
\mathbf{x}, \varkappa)$. Therefore, since the orbits of
$\textrm{so}(4)$ compact, the image of all critical points with
$\varkappa> 0$ by the mapping $(H, K, a, b, \varkappa):
\mathbb{R}^7(\mathbf{J}, \mathbf{x}, \varkappa) \to \mathbb{R}^5$ is
a closed set. Therefore, the set of points $a, b$ for which the
bifurcations for the curve $y_8 z_2$ exist is a closed subset. It
follows that these bifurcations exist and have type $2A^*$ for any
value of the parameter $a>2 \sqrt{\varkappa} b$ (and $b^2>
\varkappa^3 c_1^4$).

Finally we prove that the bifurcation corresponding to the arc $y_8
y_{13}$ has type $2B$. The arc  $y_8 y_{13}$ belongs to the
bifurcation diagram because there are $4$ tori over the ``bottom''
region and $2$ tori over the ``top'' region. The number of tori in
the areas can be easily found by a careful examination of the types
of critical points of rank $0$. For example, $2$ tori lie over the
``top'' region since the point $y_{12}$ is the image of a single
point of center-saddle type and hence the bifurcation for the curve
$y_{12} z_5$ has type $B$ (that is, one torus transforms into two).

We further note that the bifurcation corresponding to the arc $y_8
y_{13}$ consists of two identical parts. Moreover, the following
statement holds, which follows easily from Assertion
\ref{A:Rank1Series2_2Cirles} and the fact that there is no images of
the points of rank $0$ in a neighbourhood of the point $y_8$.

\begin{assertion} \label{A:Disconn_Deg_Point} The preimage of a sufficiently small neighbourhood of the singular
point $y_8$ consists of two connected components and the symmetry
$\sigma_3: (J_3, x_3) \to (-J_3, -x_3)$ interchanges these
components.\end{assertion}

Thus there are two identical bifurcations corresponding to the arc
$y_8 y_{13}$, which transform two tori into four. It follows from
considerations of continuity as previously for the bifurcations
$2A^*$ that both these bifurcations have type $B$: for $a>f_t(b)$
the bifurcation corresponding to the arc $y_7 y_8$ has type $2B$
(the loop molecule of the saddle--saddle point in the preimage of
the point $y_7$ is uniquely determined by the fact that there are
$4$ tori in the preimage of points in one of the neighbouring
cameras).

Thus we determined all the bifurcations. Theorems \ref{T:Bif_Diag}
and \ref{T:Bifurcations} are proved.
\end{proof}

\begin{proof}[of Theorem \ref{T:Molecules}]

The structure of loop molecules for nondegenerate singularities of
rank $0$ is well known and is described in detail in the book
\cite{BolsinovFomenko99}. It remains to prove Theorem
\ref{T:Molecules} for the images of degenerate critical points of
rank $1$. In almost all cases the loop molecules (without marks) for
degenerate critical points can be uniquely determined using the
obtained information about the bifurcations and the number of tori
for all areas. Ambiguity for the loop molecules of points $y_8$ and
$y_9$ can be easy solved using the fact that the loop molecules of
these points must consist of two identical parts (see Assertion
\ref{A:Disconn_Deg_Point}). Marks for the degenerate singularities
can be found using standard methods (``rule of summation of the
marks'', considerations of continuity), which are described, for
example, in \cite{BolsinovFomenko99} or \cite{Morozov06}.
\end{proof}

\section{Classical Kovalevskaya case ($\varkappa=0$)}
\label{S:Classival_Kovalevskaya}

In this section we show that the bifurcation diagrams for classical
Kovalevskaya case defined on the Lie algebra $\textrm{e}(3)$ by the
Hamiltonian
\begin{equation} H = J_1^2 + J_2^2 + 2J_3^2 + 2 x_1\end{equation}
and the integral
\begin{equation} K = (J_1^2 - J_2^2-2 x_1)^2 + (2J_1J_2 - 2 x_2)^2
\end{equation} can be obtained from the bifurcation diagrams of the
integrable Hamiltonian system with Hamiltonian
\eqref{Eq:Hamiltonian} and the first integral
\eqref{Eq:First_Integral} (where $c_1=1$) on the Lie algebra
$\textrm{so}(4)$ by passing to the limit $\varkappa \to 0$. This
limit $\varkappa \to 0$ preserves types of critical points of rank
$0$, the bifurcations of Liouville tori and the loop molecules of
singular points of the momentum mappping.

The structure of bifurcation diagrams for the classical Kovalevskaya
case is well known and is described, for example, in
\cite{Harlamov88} and \cite{BolsinovFomenko99}. However, in this
section we not only compare the answers but also show how to
construct the bifurcation diagram for the classical Kovalevskaya
case and compute some of its invariants (more precisely, in this
section we determine the bifurcations of Liouville tori) using the
obtained information about the Kovalevskaya case on the Lie algebra
$so(4)$ while performing as little as possible additional
calculations.

In this section we denote by $\Sigma (a, b, 0)$ the bifurcation
diagram of the momentum mapping for the orbit $M_{a, b}$ of the Lie
algebra for which the corresponding value of the parameter of the
pencil is equal to $\varkappa$.

\begin{lemma} \label{L:Bif_Diag_Limit} Consider arbitrary $a, b \in
\mathbb{R}$ such that $a>0$.  Then a point $x$ belongs to the
bifurcation diagram $\Sigma (a, b, 0)$ if and only if there exists a
sequence of points $x_n \in \Sigma(a_n, b_n, \varkappa_n)$ such that
$\lim_{n \to \infty} (a_n, b_n, \varkappa_n) = (a, b, 0)$.
\end{lemma}

\begin{proof}
In one direction, it follows from Assertion \ref{A:Rank_1_Preimage}
that for any critical point $z$ on the Lie algebra $e(3)$ (that is,
for a point with the parameter $\varkappa=0$) there exists a
sequence of critical points $z_n$  on the Lie algebra $so(4)$ (that
is, such that the value of the parameter $\varkappa>0$) converging
to it. Therefore, the image of the point $z$ is the limit of the
images of points $z_n$.

In the other direction, the proof is by contradiction. Suppose that
a point $x \in \mathbb{R}^2$ is regular, that is $ x \not \in \Sigma
(a, b, 0)$, but there exists a sequence $x_n \in \Sigma (a_n, b_n,
\varkappa_n)$ such that $x_n \to x$ and $ (a_n, b_n, \varkappa_n)
\to (a, b, 0)$ as $\varkappa \to 0$. In order to get a
contradiction, we choose a point $z_n$ in the preimage of each point
$x_n$ and prove that sequence $z_n$ contains a convergent
subsequence. To do this we show that the sequence $z_n$ is contained
in a compact set $A \subset \mathbb{R}^7(\mathbf{J}, \mathbf{x},
\varkappa)$.

Consider two sufficiently small closed discs $\overline{D_1}$ and
$\overline{D_2} \subset \mathbb{R}^2$ that contain points $x$ and
$(a, b)$ respectively and a small segment $[0, T] \subset
\mathbb{R}$. Then the set
\[ A = \{(J, x, \varkappa) | (H, K, f_1, f_2, \varkappa)(\mathbf{J},
\mathbf{x}, \varkappa) \subset \overline{D_1} \times \overline{D_2}
\times [0, T] \}
\] is compact. Indeed, $A \subset \mathbb{R}^7$ is a closed subset since $\overline{D_1} \times \overline{D_2} \times [0,
T]$ is a closed subset of $\mathbb{R}^5$ and the mapping $(H, K,
f_1, f_2, \varkappa): \mathbb{R}^7(\mathbf{J}, \mathbf{x},
\varkappa) \to \mathbb{R}^5$ is continuous. It remains to prove that
the set $A$ is bounded. Note that the set of numbers $(x_1, x_2,
x_3)$ is bounded since the integral $f_1 = \varkappa \textbf{J}^2 +
\textbf{x}^2$ is bounded above and below by some constants. It
remains to note that the set of numbers $(J_1, J_2, J_3)$ is also
bounded because \[J_1^2 + J_2^2 + 2J_3^2 = H - 2 c_1 x_1
\] and the right side is bounded since $(H, K) \in \overline{D_1}$
for all points from $A$. Lemma \ref{L:Bif_Diag_Limit} is proved.
\end{proof}

It is not hard to get the exact equations for the curves that
contain the bifurcation diagrams of the momentum mapping for
$\varkappa =0$. To do this it suffices to pass to the limit
$\varkappa \to 0$ in the equations for the curves
\eqref{Eq:Parametric Curve} and \eqref{Eq:Left Parabola} (the right
parabola \eqref{Eq:Right Parabola} ``shifts to the right to
infinity'' as $\varkappa \to 0$ thus it has no limit points for
$\varkappa =0$).

\begin{lemma} \label{L:Curve_Images_B_not_0_Kappa 0}
Let $\varkappa = 0$ and $b \ne 0$. Then for any non-singular orbit
$M_{a, b}$ (that is, for any orbit such that $a> 0$) the bifurcation
diagram $\Sigma_{h, k}$ for the integrable Hamiltonian system with
Hamiltonian \eqref{Eq:Hamiltonian}  and integral
\eqref{Eq:First_Integral} is contained in the union of the following
three families of curves on the plane $\mathbb{R}^2(h,k)$:
\begin{enumerate}
\item The line $k=0$; \item The parametric curve
\begin{equation} \label{Eq:Parametric Curve Kappa Zero} h(z)= \frac{b^2 c_1^2}{z^2}+2 z, \qquad k(z)= 4 a
c_1^2-\frac{4 b^2 c_1^2}{z}+\frac{b^4 c_1^4}{z^4},
\end{equation} where $z \in \mathbb{R} - \{0\}$.
\item The parabola \begin{equation} \label{Eq:Parabola Kappa Zero} k =
\left(h- \frac{2 b^2}{a}\right)^2.
\end{equation}

\end{enumerate}
\end{lemma}

We have an analogous statement for $b =0$.

\begin{lemma} \label{L:Curve_Images_Kappa 0 B_0}
Let $\varkappa = 0$ and $b=0$. Then for any non-singular orbit
$M_{a, 0}$ (that is, for any orbit such that $a > 0$) the
bifurcation diagram $\Sigma_{h, k}$  for the integrable Hamiltonian
system with Hamiltonian \eqref{Eq:Hamiltonian} and integral
\eqref{Eq:First_Integral} is contained in the union of the following
three families of curves on the plane $\mathbb{R}^2(h,k)$:
\begin{enumerate}
\item The line $k=0$; \item The union of the parabola
\begin{equation}\label{Eq:Up Parabola Kappa Zero B Zero} k=h^2 +
4ac_1^2\end{equation} and the tangent line to this parabola at the
point $h=0$
\begin{equation}\label{Eq:Tangent Line Kappa Zero B Zero} k=
4ac_1^2.\end{equation} \item The parabola
\begin{equation} \label{Eq:Parabola Kappa Zero B Zero}
k=h^2.\end{equation} \end{enumerate}
\end{lemma}

Now we determine which areas described in Theorems \ref{T:Bif_Diag}
and \ref{T:Bif_Diag_B_Zero} survive as $\varkappa \to 0$. It is not
hard to check that in the limit the curves $f_k, f_r, f_t$ and $f_l$
given by the formulas \eqref{Eq: Function_F_K},
\eqref{Eq:Function_F_r}, \eqref{Eq:Triple_Intersect} and
\eqref{Eq:Function_F_l} respectively go to the curves $a=
\frac{3}{4}\frac{b^{4/3}}{c_1^{2/3}}$, $a=
\frac{b^{4/3}}{c_1^{2/3}}$, $a = \frac{1}{2^{2/3}}
\frac{b^{4/3}}{c_1^{2/3}}$ and $b =0$ respectively. (For a fixed
$b\ne 0$ the curve $f_m(b)$  has no limit points in the area $\{a>0,
b> 0\}$ as $\varkappa \to 0$.) The found curves divide the area
$\{a> 0, b>0\}$ into $4$ sub-areas which we denote in this paper as
follows:

\begin{enumerate}

\item Area $\textrm{I}'$ is the area $\{\varkappa=0, \quad 0<b, \quad 0 <a <\frac{1}{2^{2/3}}
\frac{b^{4/3}}{c_1^{2/3}} \}$;

\item Area $\textrm{II}'$: $\{\varkappa=0, \quad 0<b,  \quad \frac{1}{2^{2/3}}
\frac{b^{4/3}}{c_1^{2/3}}<a<\frac{3}{4}\frac{b^{4/3}}{c_1^{2/3}}\}$;

\item Area $\textrm{III}'$ : $\{\varkappa=0, \quad 0<b, \quad \frac{3}{4}\frac{b^{4/3}}{c_1^{2/3}}<a<\frac{b^{4/3}}{c_1^{2/3}}\}$;

\item Area $\textrm{IV}'$: $\{\varkappa=0, \quad 0<b, \quad  \frac{b^{4/3}}{c_1^{2/3}}<a\}$.

\end{enumerate}

If $\varkappa=0$, then all the curves $f_r, f_k, f_t$ and $f_l$
intersect only at the origin, therefore in this case we should
consider only one additional area of the line $b=0$:

\begin{enumerate}

\item Area $\textrm{V}'$: $\{\varkappa=0, \quad b=0,  \quad 0<a \}$.

\end{enumerate}

\begin{remark}
If $\varkappa=0$, then without loss of generality, we can assume
that  $a=1$ (other orbits can be obtained from the case $a=1$ by a
suitable change of variables). It is not hard to check that the line
$a=1$ intersects the areas $\textrm{I}'$--$\textrm{V}'$ at the
following subsets: it intersects the first four areas at the
intervals $2<b^2$, $(4/3)^{3/2}<b^2<2$, $1<b^2<(4/3)^{3/2}$ and
$0<b^2<1$ (intervals are listed in ascending order of areas) and it
intersects the area $\textrm{V}'$ at the point $b=0$.
\end{remark}

Now it is not hard to understand how the bifurcation diagrams for
the classical Kovalevskaya case look like: roughly speaking, they
are obtained from the diagrams shown in Fig.
\ref{Fig:Koval_U_1_A}--\ref{Fig:Koval_U_4_C} (or in Fig.
\ref{Fig:Koval_Z_1} for $b=0$) by removing all the ``right arcs'',
that is all the arcs belonging to the right parabola \eqref{Eq:Right
Parabola} and the arc $z_1 z_2$. Namely, the following theorem
holds.

\begin{theorem}[\cite{Harlamov88}] \label{T:Bif_Diag_Kappa_Zero}
Let $\varkappa = 0$ and $b > 0$. The curves \[a=
\frac{b^{4/3}}{c_1^{2/3}}, \quad a=
\frac{3}{4}\frac{b^{4/3}}{c_1^{2/3}} \quad \text{and} \quad a =
\frac{1}{2^{2/3}} \frac{b^{4/3}}{c_1^{2/3}}\] divide the area
$\{a>0, b> 0\}$ into $4$ areas. In Fig. \ref{Fig:Class_Koval_4}--
\ref{Fig:Class_Koval_1} the bifurcation diagrams of the momentum
mapping for the integrable Hamiltonian system with Hamiltonian
\eqref{Eq:Hamiltonian} and integral \eqref{Eq:First_Integral} on the
orbit $M_{a, b}$ of the Lie algebra $\textrm{e}(3)$ are shown for
each of the these areas. The enlarged fragments of Fig.
\ref{Fig:Class_Koval_4}, \ref{Fig:Class_Koval_3},
\ref{Fig:Class_Koval_2} and \ref{Fig:Class_Koval_1} have the same
forms as Fig. \ref{Fig:Koval_U_4_C}, \ref{Fig:Koval_U_3_C},
\ref{Fig:Koval_U_2_C} and \ref{Fig:Koval_U_1_C} respectively. The
bifurcations diagram for the orbit $M_{a, 0}$ of the Lie algebra
$\textrm{e}(3)$ (that is, in the case $\varkappa =0, b=0, a>0$) is
shown in Fig. \ref{Fig:Class_Koval_0} .
\end{theorem}

\begin{figure}[h!]
    \centering
        \includegraphics[width=0.48\textwidth]{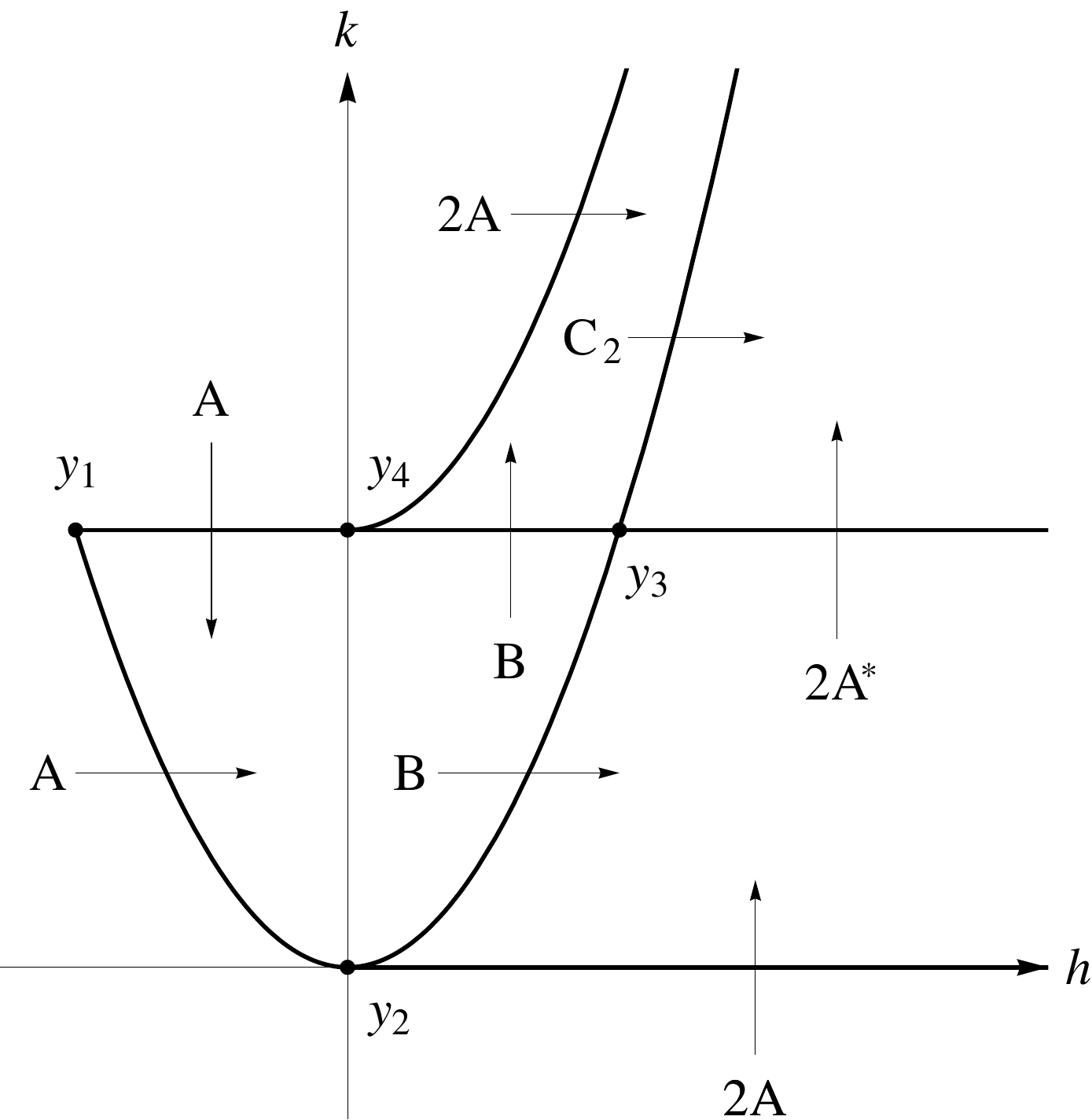}
   \caption{Classical Kovalevskaya case: $\varkappa =0, \quad b=0$}
      \label{Fig:Class_Koval_0}
\end{figure}

\begin{figure}[!htb]
\minipage{0.48\textwidth}
    \includegraphics[width=\linewidth]{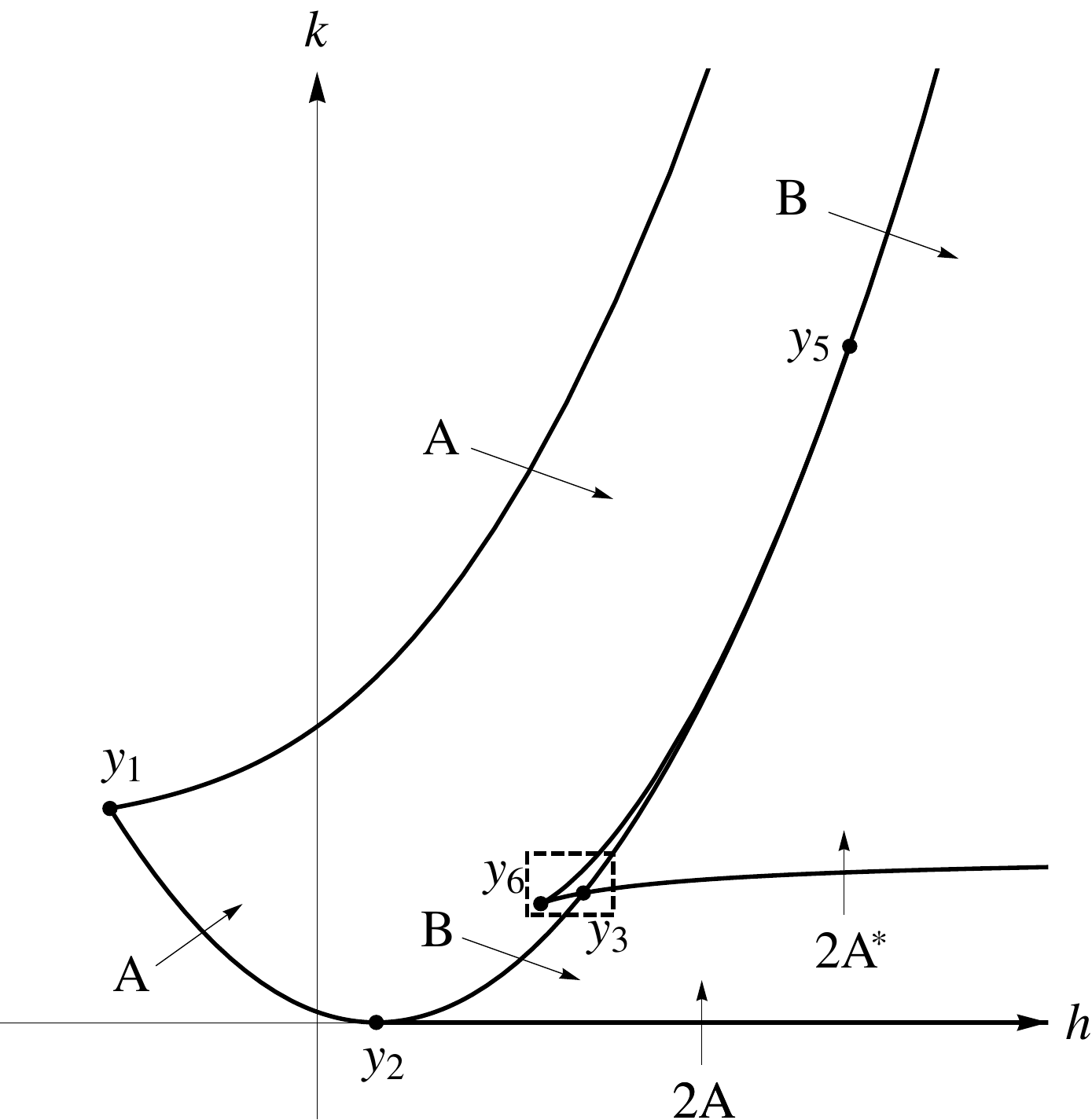}
   \caption{Classical Kovalevskaya case: $\varkappa =0, \quad a=1, \quad 0<b^2<1$}
   \label{Fig:Class_Koval_4}
\endminipage
\hspace{0.04\textwidth}
\minipage{0.48\textwidth}
    \includegraphics[width=\linewidth]{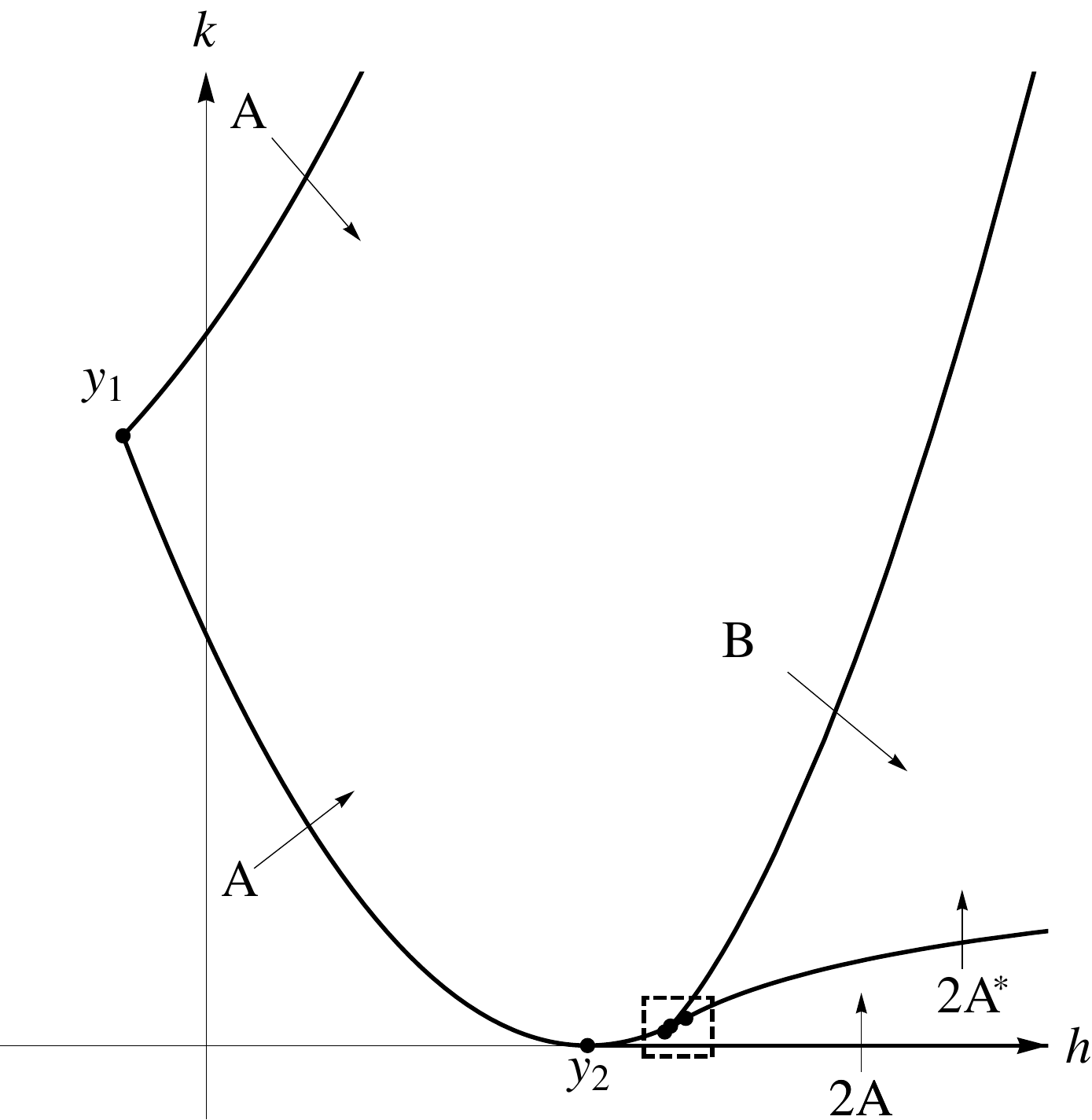}
   \caption{Classical Kovalevskaya case: $\varkappa =0, \quad a=1,  \quad 1<b^2<(4/3)^{3/2}$}
   \label{Fig:Class_Koval_3}
\endminipage
\end{figure}

\begin{figure}[!htb]
\minipage{0.48\textwidth}
    \includegraphics[width=\linewidth]{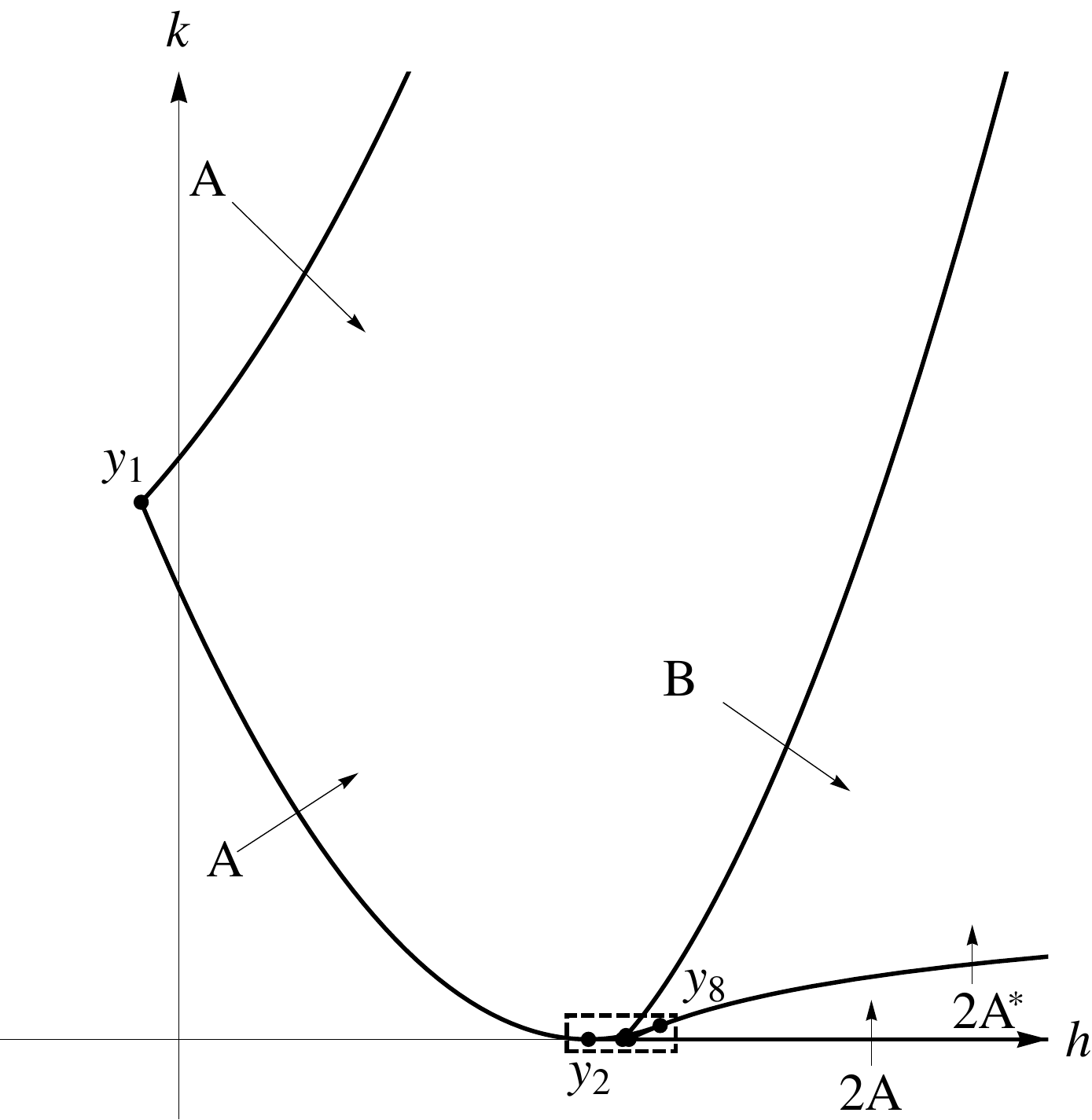}
   \caption{Classical Kovalevskaya case: $\varkappa =0, \quad a=1, \quad (4/3)^{3/2}<b^2<2$}
   \label{Fig:Class_Koval_2}
\endminipage
\hspace{0.04\textwidth}
\minipage{0.48\textwidth}
  \includegraphics[width=\linewidth]{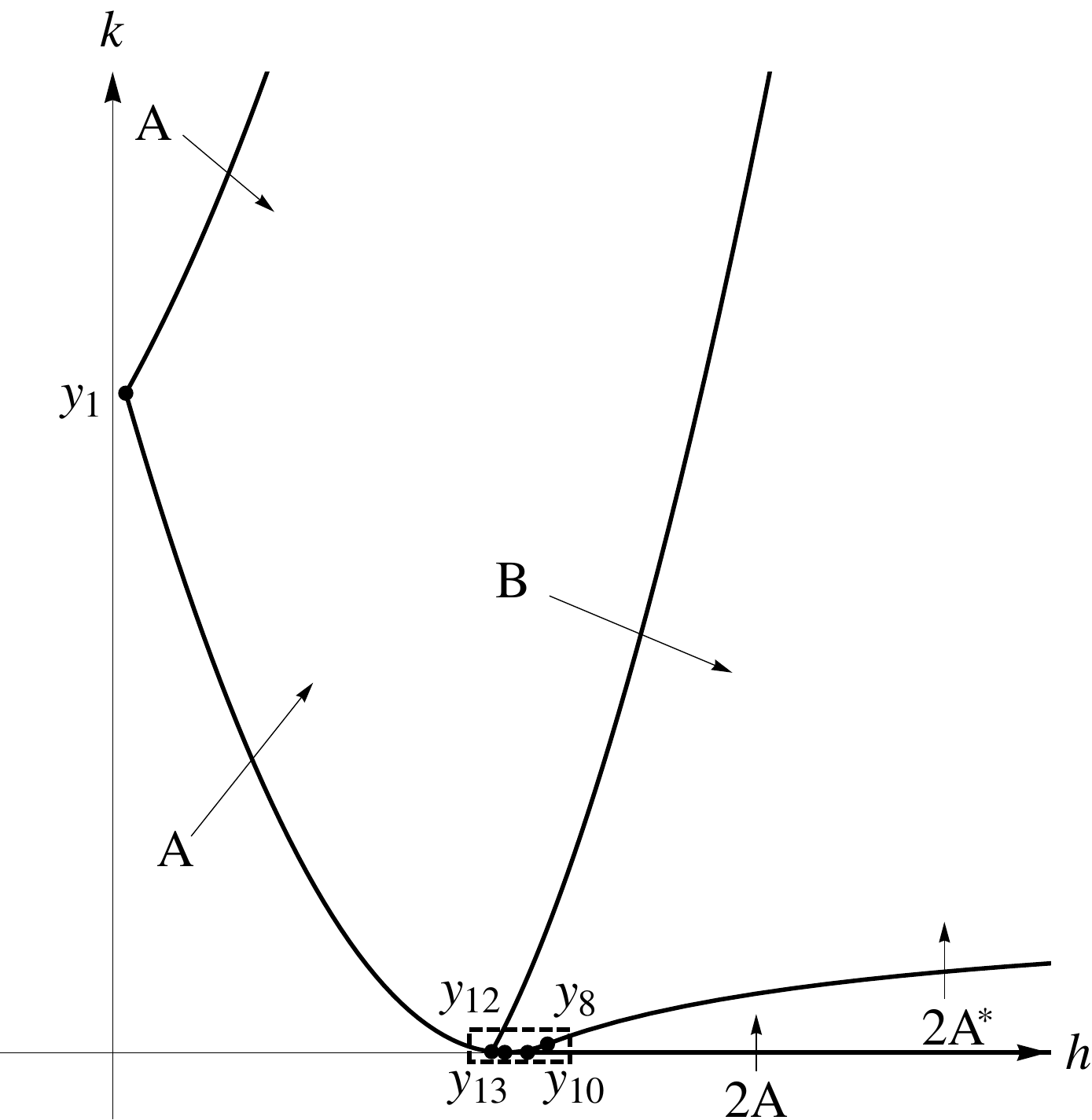}
   \caption{Classical Kovalevskaya case: $\varkappa =0, \quad a=1, \quad 2<b^2$}
   \label{Fig:Class_Koval_1}\endminipage
\end{figure}

\begin{remark}
In Fig. \ref{Fig:Class_Koval_4}--\ref{Fig:Class_Koval_1} the arcs
$y_1 y_2, y_2 y_3, y_3 y_5, y_2 y_7, y_7 y_8, y_1 y_{12}, y_{12}
y_{13}$ and $y_{13} y_8$ belong to the parabola \eqref{Eq:Parabola
Kappa Zero}. The rest of the arcs distribute between the curve
\eqref{Eq:Parametric Curve Kappa Zero} and the line $k=0$ in an
obvious way. \end{remark}

\begin{proof}[of Theorem \ref{T:Bif_Diag_Kappa_Zero}]

First of all, it is necessary to check the nondegeneracy of the
critical points in the case  $\varkappa=0$ since in general critical
points may become degenerate during a variation of parameters (for
example, for the integrable system on the Lie algebra $so(4)$ under
consideration this happens when passing from one area of parameters
$a, b$ to another). Nevertheless, all calculations done in Sections
\ref{SubS:CritRank1} and \ref{SubS:Critical_Points_Zero_Rang} for
the points of rank $1$ and $0$ respectively remain valid in the case
$\varkappa =0$, therefore it is not hard to verify that all critical
points corresponding to non-singular points of the bifurcation
diagram are nondegenerate critical points of rank $1$ and that all
critical points of rank $0$ are nondegenerate and have the same type
as the corresponding critical points of rank $0$ for the Lie algebra
$so(4)$.

Furthermore, it follows from the reasons of continuity that the
preimage of each regular point in the image of the momentum mapping
for the Lie algebra $e(3)$ must contain the same number tori as the
preimage of a regular point from the corresponding region for the
Lie algebra $so(4)$. Similarly, the number of critical circles in
the preimage of non-singular points of bifurcation diagrams must
coincide. (The number of critical circles does not decrease because
all singularities are nondegenerate. Under a small perturbation of
parameters a complex singularities can decompose into several simple
ones but it does not occur in this case --- it follows from the
explicit form of the bifurcation diagrams. The number of critical
circles does not increase for the same arguments as in the proof of
the fact in Lemma \ref{L:Bif_Diag_Limit} that there is no points of
bifurcation diagrams in a neighbourhood of a regular point).

Now, since we know the types of critical points of rank $0$, the
numbers of tori and critical circles we can determine almost all
bifurcations of Liouville tori. It remains to use the stability of
bifurcations $A, B$ and $A^*$ and standard considerations of
continuity to get rid of the ambiguity for some arcs. For example,
in the case  $\varkappa =0, b=0$ the bifurcation corresponding to
the arc $P_1$ has type $C_2$ and not $2A^*$ since the bifurcation
corresponding to the arc $y_3 z_5$ has type $C_2$.

Theorem \ref{T:Bif_Diag_Kappa_Zero} is completely proved.

\end{proof}

\end{document}